\documentclass[a4paper,11pt]{article}

\usepackage[pdftex]{hyperref}
\usepackage{stmaryrd}
\usepackage[normalem]{ulem}
\usepackage{graphicx}
\usepackage{amsfonts}
\usepackage{amsmath,amssymb,mathrsfs,amsthm}
\usepackage[all]{xy}
\usepackage{color}
\usepackage{times}
\usepackage{enumerate,vmargin}
\usepackage{verbatim}
\usepackage{paralist} 
\usepackage[labelfont=bf,font=it]{caption}
\usepackage[numbers,sort&compress]{natbib}

\setpapersize{A4}
\setmarginsrb{2.5cm}{2cm}{2.5cm}{2cm}{.4cm}{4mm}{0cm}{7.5mm}

\hypersetup{
    unicode      = false,     
    pdftoolbar   = true,      
    pdfmenubar   = true,      
    pdffitwindow = true,      
    pdfnewwindow = true,      
    colorlinks   = true,      
    linkcolor    = blue,      
    citecolor    = blue,      
    filecolor    = blue,      
    urlcolor     = blue       
}

\graphicspath{{.}{figures/}}

\theoremstyle{plain}
\newtheorem{lem}{Lemma}[section]
\newtheorem{prop}[lem]{Proposition}
\newtheorem{cor}[lem]{Corollary}
\newtheorem{algo}[lem]{Algorithm}

\theoremstyle{plain}

\theoremstyle{definition}

\theoremstyle{remark}

\theoremstyle{plain}
\newtheorem{thma}[lem]{Theorem}

\theoremstyle{plain}

\newcommand{\noi}{\noindent}

\numberwithin{equation}{section}

\newcommand{\te}{\mathcal{T}}
\newcommand{\n}{^{n}}
\newcommand{\llb}{\llbracket}
\newcommand{\rrb}{\rrbracket}

\newcommand{\bp}{\boldsymbol p}

\newcommand{\carrow}{\mathop{\longrightarrow}^{n\to\infty}_d}
\newcommand{\graft}[1]{\mathop{\circledast}_{#1}}

\newcommand{\ep}{\epsilon}
\newcommand{\dhau}{\operatorname{d_{H}}}
\newcommand{\dpr}{\operatorname{d_{P}}}
\newcommand{\dgh}{\operatorname{d_{GH}}}
\newcommand{\dgp}{\operatorname{d_{GP}}}
\newcommand{\dpgh}{\operatorname{d_{pGH}}}
\newcommand{\dghp}{\operatorname{d_{GHP}}}
\newcommand{\dprnk}{\operatorname{d_P^{n,k}}}
\newcommand{\dprk}{\operatorname{d_P^{k}}}
\newcommand{\gh}{\operatorname{GH}}
\newcommand{\gp}{\operatorname{GP}}
\newcommand{\ghp}{\operatorname{GHP}}

\newcommand{\btheta}{\boldsymbol \theta}
\newcommand{\bTheta}{\boldsymbol \Theta}

\newcommand{\Sk}{\operatorname{Sk}}
\newcommand{\Lf}{\operatorname{Lf}}
\newcommand{\Br}{\operatorname{Br}}
\newcommand{\Span}{\operatorname{Span}}
\newcommand{\cut}{\operatorname{cut}}
\newcommand{\shuff}{\operatorname{shuff}}
\newcommand{\Card}{\operatorname{Card}}

\newcommand{\me}{\operatorname{mg}}
\newcommand{\Ht}{\operatorname{Ht}}
\newcommand{\Sub}{\operatorname{Sub}}
\newcommand{\supp}{\operatorname{supp}}
\newcommand{\Kt}{\operatorname{Kt}}

\newcommand{\tc}{\cut(\te)}

\newcommand{\bv}{\mathbf v}
\newcommand{\bu}{\mathbf u}
\newcommand{\bx}{\mathbf x}
\newcommand{\by}{\mathbf y}

\newcommand{\ba}{\mathbf a}

\newcommand{\bY}{\mathbf Y}
\newcommand{\bU}{\mathbf U}
\newcommand{\bV}{\mathbf V}
\newcommand{\bX}{\mathbf X}

\newcommand{\bB}{\mathbf B}

\newcommand{\bT}{T}

\newcommand{\bmu}{\boldsymbol{\mu}}

\newcommand{\cA}{\mathcal A}
\newcommand{\cB}{\mathcal B}
\newcommand{\cC}{\mathcal C}

\newcommand{\cE}{\mathcal E}
\newcommand{\cF}{\mathcal F}
\newcommand{\cG}{\mathcal G}
\newcommand{\cH}{\mathcal H}

\newcommand{\cL}{\mathcal L}
\newcommand{\cM}{\mathcal M}

\newcommand{\cP}{\mathcal P}
\newcommand{\cT}{\mathcal T}
\newcommand{\cU}{\mathcal U}

\newcommand{\cX}{\mathcal X}

\newcommand{\rR}{\mathrm R}
\newcommand{\rP}{\mathrm P}

\newcommand{\sM}{\mathscr M}
\newcommand{\sN}{\mathscr N}

\newcommand{\sI}{\mathscr I}

\newcommand{\fB}{\mathfrak B}
\newcommand{\fR}{\mathfrak R}
\newcommand{\fv}{\mathfrak v}

\newcommand{\bbD}{\mathbb D}
\newcommand{\bbR}{\mathbb R}
\newcommand{\bbT}{\mathbb T}
\newcommand{\bbU}{\mathbb U}
\newcommand{\bbP}{\mathbb P}
\newcommand{\bbE}{\mathbb E}

\newcommand{\bbM}{\mathbb M}
\newcommand{\bbN}{\mathbb N}

\newcommand{\eqd}{\stackrel{d}=}
\newcommand{\R}{{\mathbb R}}
\newcommand{\p}[1]{{\mathbb P}\left(#1\right)}
\newcommand{\pc}[1]{{\mathbb P}(#1)}
\newcommand{\cdist}[1]{\xrightarrow[#1]{\tiny d}}

\newcommand{\edge}[1]{\langle #1 \rangle}

\newcommand{\E}[1]{\mathbb E \left[ #1\right]}
\newcommand{\Egg}[1]{\mathbb E \Bigg[ #1\Bigg]}
\newcommand{\Ec}[1]{\mathbb E [#1]}
\newcommand{\I}[1]{{\mathbf 1}_{\{ #1\}}}
\newcommand{\Ic}[1]{{\mathbf 1}_{#1}}

\title{\bf Cutting down $\bp$-trees and \\
inhomogeneous continuum random trees}
\author{Nicolas Broutin
\thanks{Inria Paris--Rocquencourt, Domaine de Voluceau, 78153 Le Chesnay - France.
Email: nicolas.broutin@inria.fr}
\and Minmin Wang
\thanks{Universit\'e Pierre et Marie Curie, 4 place Jussieu, 75005 Paris - France.
Email: wangminmin03@gmail.com }
}
\date{}

\begin{document} 

\maketitle

\begin{abstract}
We study a fragmentation of the $\bp$-trees of Camarri and Pitman [\emph{Elect.\ J.\ Probab.}, 
vol.\ 5, pp.\ 1--18, 2000]. 
We give exact correspondences between the $\bp$-trees and trees which encode the 
fragmentation. We then use these results to study the fragmentation of the ICRTs  
(scaling limits of $\bp$-trees) and give distributional 
correspondences between the ICRT and the tree encoding the fragmentation. The theorems for 
the ICRT extend the ones by Bertoin and Miermont [\emph{Ann.\ Appl.\ Probab.}, vol. 23(4), 
pp.\ 1469--1493, 2013] about the cut tree of the Brownian continuum random tree. 
\end{abstract}




\section{Introduction}\label{sec:intro}

The study of random cutting of trees has been initiated by \citet{meir69} in the following 
form: Given a (graph theoretic) tree, one can proceed to chop the tree into pieces by 
iterating the following process: choose a uniformly random edge;
removing it disconnects the tree into two pieces; discard the part which does not contain 
the root and keep chopping the portion containing the root until it is reduced to a 
single node. In the present document, we consider the related version where the vertices 
are chosen at random and removed (until one is left with an empty tree); 
each such pick is referred to as a \emph{cut}. 
We will see that this version is actually much more adapted than the edge cutting procedure 
to the problems we consider here. 

The main focus in \cite{meir69} and in most of the subsequential papers has been put 
on the study of some parameters of this cutting down process, and 
in particular on how many cuts are necessary for the process to finish. This has been 
studied for a number of different models of deterministic and random trees such as 
complete binary trees of a given height, random trees arising from the divide-and-conquer 
paradigm \cite{Holmgren2010a,Holmgren2011c,IkMo2007,DrIkMoRo2009}
and the family trees of finite-variance critical Galton--Watson processes conditioned 
on the total progeny \cite{Janson06,Panholzer2006,FiKaPa2006}. 
The latter model of random trees turns out to be far more interesting, and it provides an 
\emph{a posteriori} motivation for the cutting down process. 
As we will see shortly, the cutting down process provides an interesting way to 
investigate some of the structural properties of random trees by partial destruction and 
re-combination, or equivalently as partially resampling the tree.

Let us now be more specific: if $L_n$ denotes the number of cuts required to completely cut down 
a uniformly labelled rooted tree (random Cayley tree, or equivalently condition Galton--Watson tree with Poisson offspring distribution) on $n$ nodes, 
then $n^{-1/2}L_n$ converges in distribution to a Rayleigh distribution which has 
density $xe^{-x^2/2}$ on $\R_+$. \citet{Janson06} proved that a similar result holds for 
any Galton--Watson tree with a finite-variance offspring distribution conditioned on the total progeny to be $n$. This is the parameter 
point of view. \citet*{ABH10} have shown that for the random Cayley trees, 
$L_n$ actually has the same distribution as the number of nodes on the path between two 
uniformly random nodes. Their method relies on an ``objective'' argument based 
on a coupling that associates with the cutting procedure
a partial resampling of the Cayley tree of the kind mentioned earlier: 
if one considers the (ordered) sequence of subtrees which 
are discarded as the cutting process goes on, and adds a path linking their roots, then 
the resulting tree is a uniformly random Cayley tree, and the two extremities of the path 
are independent uniform random nodes. So the properties of the parameter $L_n$ follow from a 
stronger correspondence between the combinatorial objects themselves.

This strong connection between the discrete objects can be carried to the level of their 
scaling limit, namely Aldous' Brownian continuum random tree (CRT) \cite{aldcrt3}. 
Without being too precise for now, the natural cutting procedure on the Brownian CRT
involves a Poisson rain of cuts sampled according to the length measure.
However, not all the cuts contribute to the isolation of the 
root. As in the partial resampling of the discrete setting, we 
glue the sequence of discarded subtrees  along an interval, thereby obtaining a new CRT.
If the length of the interval is well-chosen (as a function of the cutting process), 
the tree obtained is distributed 
like the Brownian CRT and the two ends of the interval are independently random leaves. 
This identifies the distribution of the discarded subtrees from the cutting procedure as the distribution 
of the forest one obtains from a spinal decomposition of the Brownian CRT. 
The distribution of the latter is intimately related to Bismut's \cite{Bismut1985a} decomposition of a Brownian excursion. See also \cite{Duquesne05} for the generalization to the L\'evy case.
Note that a similar identity has been proved by \citet{AbDe12} for general L\'evy trees 
without using a discrete approximation. 
A related example is that of the subtree prune 
and re-graft dynamics of \citet{EPW} \cite[See also][]{evans05}, which is even closer to the 
cutting procedure and truly resamples the object rather than giving a ``recursive'' decomposition.

The aim of this paper is two-fold. First we prove exact identities and give reversible transformations 
of $\bp$-trees similar to the ones for Cayley trees in \cite{ABH10}. 
The model of $\bp$-trees introduced by \citet{pit00} generalizes Cayley trees in allowing ``weights'' on 
the vertices. In particular, this additional structure of weights introduces some inhomogeneity. 
We then lift these results to the scaling limits, the inhomogeneous continuum random trees (ICRT) of
\citet{ald00}, which are closely related to the general additive coalescent \cite{ald00,
Bertoin2000a, Bertoin2001a}. 
Unlike the Brownian CRT or the stable trees (special cases of L\'evy trees), a general ICRT is not self-similar. 
Nor does it enjoy a ``branching property" as the L\'evy trees do \cite{legalllejan}.
This lack of ``recursivity'' ruins the natural approaches such as the one used in 
\cite{AbDe11,AbDe12} or the ones which would argue by comparing two fragmentations with the same 
dislocation measure but different indices of self-similarity \cite{Bertoin2002}. 
This is one of the reasons why we believe these path transformations at the level of the ICRT are interesting. 
Furthermore, 
a conjecture of \citet*[][p.~185]{ald04a} suggests that the path transformations for ICRTs
actually explain the result of \citet{AbDe12} for L\'evy trees by providing a result 
``conditional on the degree distribution''.

Second, rather than only focusing on the isolation of the root we also consider the genealogy of the 
entire fragmentation as in the recent work of \citet{Bert12} and \citet{Dieuleveut2013a} 
(who examine the case of Galton--Watson trees). 
In some sense, this consists in obtaining transformations corresponding 
to tracking the effect of the cutting down procedure on the isolation of all the 
points simultaneously. Tracking finitely many points is a simple generalization of the one-point 
results, but the ``complete'' result requires additional insight. 
The results of the present document are used in a companion paper \cite{BrWa2014a} to 
prove that the ``complete'' cutting procedure in which one tries to isolate 
every point yields a construction of the genealogy of the fragmentation on ICRTs which is 
reversible in the case of the Brownian CRT. More precisely, the genealogy of 
Aldous--Pitman's fragmentation of a Brownian CRT is another Brownian CRT, say $\cG$, 
and there exists a random transformation of $\cG$ into a real tree $\cT$ such that in the 
pair $(\cT,\cG)$ the tree $\cG$ is indeed distributed as the genealogy of the fragmentation 
on $\cT$, conditional on $\cT$. 
The proof there relies crucially on the ``bijective'' approach that we develop here.

\medskip
\noindent\textbf{Plan of the paper.}\ 
In the next section, we introduce the necessary notation and relevant background. We then present 
more formally the discrete and continuous models we are considering, and in which sense the 
inhomogeneous continuum random trees are the scaling limit of $\bp$-trees. In Section~\ref{sec:results} 
we introduce the cutting down procedures and state our main results. 
The study of cutting down procedure for $\bp$-trees is the topic of Section~\ref{sec:ptree-cutting}. 
The results are lifted to the level of the scaling limits in Section~\ref{sec:cont_cutting}.

\section{Notation, models and preliminaries}\label{sec:prelem}

Although we would like to introduce our results earlier, a fair bit of notation 
and background is in order before we can do so properly. This section may safely be skipped 
by the impatient reader and referred to later on. 



\subsection{Aldous--Broder Algorithm and  $\bp$-trees}

Let $A$ be a finite set and $\bp=(p_u, u\in A)$ be a probability measure on $A$
such that $\min_{u\in A}p_u>0$; this ensures that $A$ is indeed the support of $\bp$.
Let $\bbT_A$ denote the set of rooted trees labelled with (all the) elements
of~$A$ (connected acyclic graphs on $A$, with a distinguished vertex).
For $t\in \bbT_A$, we let $r=r(t)$ denote its root vertex.
For $u,v\in A$, we write $\{u,v\}$ to mean that $u$ and $v$ are adjacent in $t$.
We sometimes write $\edge{u,v}$ to mean that $\{u,v\}$ is an edge of $t$,
and that $u$ is on the path between $r$ and $v$ (we think of the
edges as pointing towards the root). 
For a tree $t\in \bbT_A$ (rooted at $r$, say) and a node $v\in A$, we let $t^v$ denote the tree 
re-rooted at $v$.

We usually abuse notation, but we believe it does not affect the clarity or precision of
our statements. For instance, we refer to a node $u$ in the vertex set $\fv(t)$ of a tree $t$ 
using $u\in t$. Depending on the context, we sometimes write $t\setminus \{u\}$ to denote the
forest induced by $t$ on the vertex set $\fv(t)\setminus \{u\}$.
The (in-)degree $C_u(t)$ of a vertex $u\in A$ is the number of edges of the form
$\edge{u, v}$ with $v\in A$. For a rooted tree $t$, and a node $u$ of $t$, 
we write $\Sub(t,u)$ for the subtree of $t$ rooted at $u$ (above $u$).
For $t\in \bbT_A$ and $\bV\subseteq A$, we write $\Span(t; \bV)$ for the subtree of $t$ 
spanning $\bV$ and the root of $r(t)$. So $\Span(t; V)$ is the subtree induced by $t$ 
on the set 
$$\bigcup_{u\in \bV} \llb r(t), u\rrb,$$
where $\llb u,v \rrb$ denotes collection of nodes on the (unique) path between $u$ and $v$ in $t$.
When $\bV=\{v_1,v_2,\dots, v_k\}$, we usually write $\Span(t; v_1,\dots, v_k)$ instead of
$\Span(t; \{v_1,\dots, v_k\})$.
We also write 
$$\Span^*(t;\bV):=\Span(t;\bV)\setminus \{r(t)\}.$$

As noticed by \citet{ald90} and \citet{bro89}, one can generate random trees on $A$ by 
extracting a tree from the trace of a random walk on $A$, where the sequence of steps
is given by a sequence of i.i.d.\ vertices distributed according to $\bp$.

\begin{algo}[Weighted version of Aldous--Broder Algorithm]\label{alg:ab}
Let $\bY=(Y_j, j\ge 0)$ be a sequence of independent variables with common
distribution $\bp$; further on, we say that $Y_j$ are i.i.d.\ $\bp$-nodes.
Let $\cT(\bY)$ be the graph rooted at $Y_0$ with the set of edges
\begin{equation}\label{ab}
\{\edge{ Y_{j-1}, Y_{j}}: Y_j\notin\{Y_0, \cdots, Y_{j-1}\}, j\ge 1\}.
\end{equation}
\end{algo}

The sequence $\bY$ defines a random walk on $A$, which
eventually visits every element of $A$ with probability one,
since $A$ is the support of $\bp$.
So the trace $\{\edge{Y_{j-1},Y_j}: j\ge 1\}$ of the random walk on $A$ is a
connected graph on $A$, rooted at $Y_0$.
Algorithm~\ref{alg:ab} extracts the tree $\cT(\bY)$ from the trace of the random
walk.
To see that $\cT(\bY)$ is a tree, observe that the edge $\edge{Y_{j-1},Y_j}$ is
added only if $Y_j$ has never appeared
before in the sequence. It follows easily that $\cT(\bY)$ is a connected graph
without cycles, hence a tree on $A$.
Let $\pi$ denote the distribution of $\cT(\bY)$.

\begin{lem}[\cite{ald90,bro89,evans-pitman}]\label{lem:dist}
For $t\in \bbT_A$, we have
\begin{equation}\label{eq:p-tree}
\pi(t):=\pi^{(\bp)}(t)=\prod_{u\in A} p_u^{C_u (t)}.
\end{equation}
\end{lem}

Note that $\pi$ is indeed a probability distribution on $\bbT_A$, since by Cayley's 
multinomial formula (\cite{cayley89,renyi70}), we have
\begin{equation}\label{eq:cayley}
\sum_{t\in \bbT_A} \pi(t)=\sum_{t\in \bbT_A}\prod_{u\in A} p_u^{C_u (t)}
=\left (\sum_{u\in A} p_u\right )^{|A|-1}=1.
\end{equation}
A random tree on $A$ distributed according to $\pi$ as specified by \eqref{eq:p-tree} 
is called a \emph{$\bp$-tree}. It is also called the birthday tree in the literature, for its 
connection with the general birthday problem (see \cite{pit00}). 
Observe that
when $\bp$ is the uniform distribution on $[n]:=\{1,2,\dots, n\}$, 
a $\bp$-tree is a uniformly random rooted tree on $[n]$ (a Cayley tree).
So the results we are about to present generalize the
exact distributional results in \cite{ABH10}. However, we believe that the point of
view we adopt here is a little cleaner, since it permits to make the
transformation \emph{exactly} reversible without any extra anchoring nodes
(which prevent any kind duality at the discrete level).

From now on, we consider $n\ge 1$ and let $[n]$
denote the set $\{1,2, \cdots, n\}$. We write $\bbT_n$ as a shorthand for
$\bbT_{[n]}$, the set of the rooted trees on $[n]$.
Let also $\bp=(p_i, 1\le i\le n)$ be a probability measure on~$[n]$ satisfying
$\min_{i\in [n]}p_i>0$. For a subset $A\subseteq [n]$ such that $\bp(A)>0$, we let
$\bp|_A(\,\cdot\,)=\bp(\,\cdot\,\cap A)/\bp(A)$ denote the restriction of $\bp$ on $A$, and 
write $\pi|_A:=\pi^{(\bp|_A)}$.
The following lemma says that the distribution of $\bp$-trees is invariant by
re-rooting at an independent  $\bp$-node  and ``recursive'' in a certain sense. 
These two properties are one of the keys to our results on the discrete objects.
(For a probability distribution $\mu$, we write $X\sim \mu$ to mean that 
$\mu$ is the distribution of the random variable $X$.)

\begin{lem}\label{lem:re-root}Let $T$ be a $\bp$-tree on $[n]$.
\begin{compactenum}[i)]
\item If $V$ is an independent $\bp$-node. Then, $T^V\sim \pi$.
\item Let $N$ be set of neighbors of the root in $T$. Then, for $u\in N$,
conditional on $\fv(\Sub(T, u))=\bV$, $\Sub(T, u)\sim \pi|_\bV$ independent of
$\{\Sub(T, w): w\in N, w\ne u\}$.
\end{compactenum}
\end{lem}
The first claim can be verified from \eqref{eq:p-tree}, the 
second is clear from the product form of $\pi$.

\subsection{Measured metric spaces and the Gromov--Prokhorov topology}
If $(X, d)$ is a metric space endowed with the Borel $\sigma$-algebra, we denote by $\cM_f(X)$ 
the set of finite measures on $X$ and by $\cM_1(X)$ the subset of probability measures on $X$. 
If $m\in \cM_f(X)$, we denote by $\supp(m)$ the support of $m$ on $X$, that is the smallest closed 
set $A$ such that $m(A^c)=0$.
If $f: X\to Y$ is a measurable map between two metric spaces, and if $m\in\cM_f(X)$, 
then the push-forward of $m$ is an element of $\cM_f(Y)$, denoted by $f_*m\in \cM_f(Y)$, 
and is defined by $(f_*m)(A)=m(f^{-1}(A))$ for each Borel set $A$ of $Y$. 
If $m\in \cM_f(X)$ and $A\subseteq X$, we denote by $m\!\!\upharpoonright_A$ the restriction 
of $m$ to $A$: $m\!\!\upharpoonright_A\!\!(B)=m(A\cap B)$ for any Borel set $B$. 
This should not be confused with the restriction of a probability measure, which remains a 
probability measure and is denoted by $m|_A$.

We say a triple $(X, d, \mu)$ is a \emph{measured metric space} (or sometimes a 
\emph{metric measure space}) if $(X, d)$ is a Polish space (separable and complete) 
and $\mu\in \cM_1(X)$. Two measured metric spaces $(X, d, \mu)$ and 
$(X', d', \mu')$ are said to be \emph{weakly isometric} if there exists an isometry 
$\phi$ between the supports of $\mu$ on $X$ and of $\mu'$ on $X'$ such that $(\phi)_*\mu=\mu'$. 
This defines an equivalence relation between the measured metric spaces, and 
we denote by $\bbM$ the set of equivalence classes. 
Note that if $(X,d,\mu)$ and $(X',d',\mu')$ are weakly isometric, the 
metric spaces $(X,d)$ and $(X',d')$ may not be isometric. 

We can define a metric on $\bbM$ by adapting Prokhorov's distance. 
Consider a metric space $(X,d)$ and for $\epsilon>0$, let $A^\epsilon:=\{x\in X: d(x,A)<\epsilon\}$.
Then, given two (Borel) probability measures $\mu, \nu\in \cM_1(X)$, 
the Prokhorov distance $\dpr$ between $\mu$ and $\nu$ is defined by 
\begin{equation}\label{eq: defdpr}
\dpr(\mu,\nu):=\inf\{\epsilon>0: \mu(A)\le \nu(A^\epsilon)+\epsilon
\text{~and~} \nu(A)\le \mu(A^\epsilon)+\epsilon, \text{~for all~Borel~sets~}A\}.
\end{equation}
Note that the definition of the Prokhorov distance \eqref{eq: defdpr} can be easily extended to a 
pair of finite (Borel) measures on~$X$. Then, for two measured metric spaces $(X,d,\mu)$ and 
$(X',d',\mu')$ the Gromov--Prokhorov (GP) distance between them is defined to be
$$
\dgp((X,d,\mu),(X',d',\mu')) = \inf_{Z,\phi,\psi} \dpr(\phi_\ast \mu, \psi_\ast \mu'),
$$
where the infimum is taken over all metric spaces $Z$ and isometric embeddings $\phi:\supp(\mu)\to Z$ 
and $\psi:\supp(\mu')\to Z$. It is clear that $\dgp$ depends only on the equivalence classes 
containing $(X, d,\mu)$ and $(X', d', \mu')$. 
Moreover, the Gromov--Prokhorov distance turns $\bbM$ in a Polish space.

There is another more convenient characterization 
of the GP topology (the topology induced by $\dgp$) that relies on convergence of distance matrices between random points. 
Let $\cX=(X, d, \mu)$ be a measured metric space and let $(\xi_i, i\ge 1)$ be a sequence of i.i.d.\ points of 
common distribution $\mu$. In the following, we will often refer to such a sequence as 
$(\xi_i, i\ge 1)$ as an i.i.d.\ $\mu$-sequence. We write 
$\rho^{\cX}=(d(\xi_i, \xi_j), 1\le i, j<\infty)$ 
for the distance matrix associated with this sequence.
One easily verifies that the distribution of $\rho^{\cX}$ does not depend on the particular 
element of an equivalent class of $\mathbb M$. Moreover, by Gromov's reconstruction theorem 
\cite[$3\frac{1}{2}$]{Gromov}, the distribution of $\rho^{\cX}$ characterizes $\cX$ as an 
element of $\bbM$. 

\begin{prop}[Corollary 8 of \cite{Lohr}]
If $\cX$ is some random element taking values in $\bbM$ and for each $n\ge 1$, 
$\cX_n$ is a random element taking values in $\bbM$, then $\cX_n$ converges 
to $\cX$ in distribution as $n\to\infty$ if and only if $\rho^{\cX_n}$ converges to $\rho^{\cX}$ 
in the sense of finite-dimensional distributions.
\end{prop}

\medskip
\noindent\textbf{Pointed Gromov--Prokhorov topology.}\ 
The above characterization by matrix of distances turns out to be quite handy when we want to keep track of marked points. 
Let $k\in \bbN$. If $(X, d, \mu)$ is a measured metric space and $\bx=(x_1, x_2, \cdots, x_k)\in X^k$ 
is a $k$-tuple, then we say $(X, d, \mu, \bx)$ is a \emph{$k$-pointed measured metric space}, 
or simply a pointed measured metric space. 
Two pointed metric measure spaces $(X, d, \mu, \bx)$ and $(X', d', \mu', \bx')$ are said to be 
\emph{weakly isometric} if there exists an isometric bijection 
$$\phi: \supp(\mu)\cup\{x_1, x_2, \cdots, x_k\}\to \supp(\mu')\cup\{x'_1, x'_2, \cdots, x'_k\}$$ 
such that $(\phi)_*\mu=\mu'$ and $\phi(x_i)=x'_i$, $1\le i\le k$, where $\bx=(x_1, x_2, \cdots, x_k)$ 
and $\bx'=(x'_1, x'_2, \cdots, x'_k)$.
We denote by $\bbM^*_k$ the space of weak isometry-equivalence classes of 
$k$-pointed measured metric spaces. Again, we emphasize the fact that the underlying metric spaces
$(X,d)$ and $(X',d')$ do not have to be isometric. 
The space $\bbM^*_k$ equipped with the following pointed Gromov--Prokhorov topology is a Polish space.

A sequence $(X_n, d_n, \mu_n, \bx_n)_{n\ge 1}$ of $k$-pointed measured metric spaces is said to converge to 
some pointed measured metric space $(X, d, \mu, \bx)$ in the $k$-pointed Gromov--Prokhorov topology if for any $m\ge 1$,
$$
\big(d_n(\xi^*_{n, i}, \xi^*_{n, j}), 1\le i, j\le m\big)
\carrow 
\big(d(\xi^*_{i}, \xi^*_{j}), 1\le i, j\le m\big),
$$
where for each $n\ge 1$ and $1\le i \le k$, $\xi^*_{n, i}=x_{n, i}$ if 
$\bx_n=(x_{n, 1}, x_{n, 2}, \cdots, x_{n, k})$ and $(\xi^*_{n, i}, i\ge k+1)$ is a sequence of 
i.i.d.\ $\mu_n$-points in $X_n$. 
Similarly, $\xi^*_i=x_i$ for $1\le i\le k$ and $(\xi_i^*, i\ge k+1)$ is a sequence of 
i.i.d.\ $\mu$-points in $X$. This induces the $k$-pointed Gromov--Prokhorov topology on $\bbM^*_k$.

\subsection{Compact metric spaces and the Gromov--Hausdorff metric}

\noindent\textbf{Gromov--Hausdorff metric.}\ 
Two compact subsets $A$ and $B$ of a given
metric space $(X,d)$ are compared using the Hausdorff distance $\dhau$. 
$$\dhau(A,B):=
\inf\{\epsilon>0: A \subseteq B^\epsilon \text{~and~} B\subseteq A^\epsilon\}.
$$
To compare two compact metric spaces $(X,d)$ and $(X',d')$, we first embed them
into a single metric space $(Z,\delta)$ via isometries $\phi:X\to Z$ and $\psi:X'\to Z$,
and then compare the images $\phi(X)$ and $\psi(X')$ using the Hausdorff distance
on $Z$. One then defines the Gromov--Hausdoff (GH) distance $\dgh$ by
$$
\dgh((X,d),(X',d')):=\inf_{Z,\phi, \psi} \dhau(\phi(X),\psi(X')),
$$
where the infimum ranges over all choices of metric spaces $Z$ and isometric embeddings $\phi: X\to Z$
and $\psi: X'\to Z$. Note that, as opposed to the case of the GP topology, two compact 
metric spaces that are at GH distance zero are isometric. 

\medskip
\noindent\textbf{Gromov--Hausdorff--Prokhorov metric.}\ 
Now if $(X, d)$ and $(X', d')$ are two compact metric spaces and if $\mu\in  \cM_f(X)$ and 
$\mu'\in \cM_f(X')$, one way to compare simultaneously the metric spaces and the measures is to 
define
$$
\dghp\big((X, d, \mu), (X', d', \mu')\big)
:=\inf_{Z,\phi, \psi}\Big\{\dhau\big(\phi(X), \psi(X')\big)\vee \dpr(\phi_\ast \mu, \psi_\ast \mu')\Big\},
$$
where the infimum ranges over all choices of metric spaces $Z$ and isometric embeddings $\phi: X\to Z$ 
and $\psi: X'\to Z$. If we denote by $\bbM_c$ the set of equivalence classes of compact measured 
metric spaces under measure-preserving isometries, then $\bbM_c$ is Polish when endowed with $\dghp$.

\medskip
\noindent\textbf{Pointed Gromov--Hausdorff metric.}\ 
We fix some $k\in\bbN$. Given two compact metric spaces $(X, d_X)$ and $(Y, d_Y)$, let 
$\bx=(x_1, x_2, \cdots, x_k)\in X^k$ and $\by=(y_1, y_2, \cdots, y_k)\in Y^k$. Then the pointed 
Gromov--Hausdorff metric between $(X, d_X, \bx)$ and $(Y, d_Y, \by)$ is defined to be
$$
\dpgh\big((X, d_X, \bx), (Y, d_Y, \by)\big)
:=\inf_{Z, \phi, \psi}\Big\{\dhau\big(\phi(X), \psi(Y)\big)
\vee \max_{1\le i\le k}d_Z\big(\phi(x_i), \psi(y_i)\big)\Big\},
$$
where the infimum ranges over all choices of metric spaces $Z$ and isometric embeddings 
$\phi: X\to Z$ and $\psi: X'\to Z$.
Let $\bbM_c^k$ denote the isometry-equivalence classes of those compact metric spaces with $k$ 
marked points. It is a Polish space when endowed with $\dpgh$.

\subsection{Real trees}

A \emph{real tree} is a geodesic metric space without loops.
More precisely, a metric space $(X,d,r)$ is called a (rooted)
real tree if $r\in X$ and
\begin{itemize}
    \item for any two points $x,y\in X$, there exists a continuous
    injective map $\phi_{xy}: [0, d(x,y)]\to X$ such that
    $\phi_{xy}(0)=x$ and $\phi_{xy}(d(x,y))=y$.
    The image of $\phi_{xy}$ is denoted by $\llb x,y\rrb$;
    \item if $q:[0,1]\to X$ is a continuous injective map such that
    $q(0)=x$ and $q(1)=y$, then $q([0,1])=\llb x,y\rrb$.
\end{itemize}
As for discrete trees, when it is clear from context which metric we are talking about, 
we refering to metric spaces by the sets. For instance $(\cT,d)$ is often referred to as $\cT$.

A \emph{measured (rooted) real tree} is a real tree $(X, d, r)$ equipped with a finite (Borel) 
measure $\mu\in\cM(X)$. We always assume that the metric space $(X, d)$ is complete and separable. 
We denote by $\bbT_w$ the set of the weak isometry equivalence classes of 
measured rooted real trees, equipped with the pointed Gromov--Prokhorov topology.
Also, let $\bbT^c_w$ be the set of the measure-preserving isometry equivalence classes of those  
measured rooted real trees $(X, d, r, \mu)$ such that $(X, d)$ is compact. 
We endow $\bbT^c_w$ with the pointed Gromov--Hausdorff--Prokhorov 
distance. Then both $\bbT_w$ and $\bbT^c_w$ are Polish spaces. 
However in our proofs, we do not always distinguish an equivalence class and the elements in it.

Let $(\bT, d, r)$ be a rooted real tree. For $u\in \bT$, the degree of $u$ in $\bT$, 
denoted by $\deg(u, \bT)$, is the number of connected components of $\bT\setminus \{u\}$.
We also denote by
$$
\Lf(\bT)=\{u\in \bT: \deg(u, \bT)=1\}
\quad \text{and}\quad
\Br(\bT)=\{u\in \bT: \deg(u, \bT)\ge 3\}
$$
the set of the \emph{leaves} and the set of \emph{branch points} of $\bT$,
respectively. The skeleton of $\bT$ is the complementary set of $\Lf(\bT)$ in $\bT$,
denoted by $\Sk(\bT)$. 
For two points $u, v\in \bT$, we denote by $u\wedge v$ the closest common ancestor of $u$ and $v$, 
that is, the unique point $w$ of $\llb r,u\rrb \cap \llb r,v\rrb$ such that $d(u, v)=d(u, w)+d(w, v)$. 

For a rooted real tree $(\bT, r)$, if $x\in \bT$ then the subtree of $\bT$ above $x$, denoted by 
$\Sub(\bT, x)$, is defined to be 
$$
\Sub(\bT, x):=\{u\in \bT: x\in \llb r, u\rrb\}.
$$

\medskip
\noindent\textbf{Spanning subtree.}\ Let $(\bT, d, r)$ be a rooted real tree and let 
$\bx=(x_1, \cdots, x_k)$ be $k$ points of $\bT$ for some $k\ge 1$. 
We denote by $\Span(\bT; \bx)$ the smallest connected set of $\bT$ which contains the root 
$r$ and $\bx$, that is, $\Span(\bT; \bx)=\cup_{1\le i\le k}\llb r, x_i\rrb$. 
We consider $\Span(\bT; \bx)$ as a real tree rooted at $r$ and refer to it as 
a \emph{spanning subtree} or a \emph{reduced tree} of $\bT$. 

If $(\bT, d, r)$ is a real tree and there exists some $\bx=(x_1, x_2, \cdots, x_k)\in \bT^k$ 
for some $k\ge 1$ such that $\bT=\Span(\bT; \bx)$, then the metric aspect of $\bT$ is rather 
simple to visualize. More precisely, if we write $x_0=r$ and let 
$\rho^{\bx}=(d(x_i, x_j), 0\le i, j\le k)$, then $\rho^{\bx}$ determines $(\bT, d, r)$ 
under an isometry.

\medskip
\noindent\textbf{Gluing.}\ 
If $(\bT_i, d_i), i=1, 2$ are two real trees with some distinguished points 
$x_i\in \bT_i$, $i=1, 2$, the result of the \emph{gluing} of $T_1$ and $T_2$ at $(x_1,x_2)$ 
is the metric space $(\bT_1\cup \bT_2,\delta)$, where the distance $\delta$ is defined by
$$
\delta(u, v)=\left\{
\begin{array}{ll}
d_i(u, v), & \text{ if }(u, v)\in \bT_i^2, i=1, 2;\\
d_1(u, x_1)+d_2(v, x_2), & \text{ if } u\in \bT_1, v\in \bT_2.
\end{array}\right.
$$
It is easy to verify that $(\bT_1\cup\bT_2, \delta)$ is a real tree with $x_1$ and $x_2$ 
identified as one point, which we denote by $\bT_1\graft{x_1=x_2}\bT_2$ in the following. 
Moreover, if $\bT_1$ is rooted at some point $r$, we make the convention that 
$\bT_1\graft{x_1=x_2}\bT_2$ is also rooted at $r$.

\subsection{Inhomogeneous continuum random trees}\label{subset: icrt}

The inhomogeneous continuum random tree (abbreviated as ICRT in the following) has been
introduced in \cite{pit00} and \cite{ald00}. See also \cite{ald99,ald05,ald04a} 
for studies of ICRT and related problems.

Let $\boldsymbol{\Theta}$ (the \emph{parameter space}) be the set of sequences 
$\btheta=(\theta_0,\theta_1, \theta_2, \cdots)\in \bbR_+^\infty$ such that 
$\theta_1\ge\theta_2\ge\theta_3\cdots\ge 0 $, $\theta_0\ge 0$, $\sum_{i\ge 0}\theta_i^2=1$, and 
either $\theta_0>0$ or $\sum_{i\ge 1}\theta_i=\infty$.

\medskip
\noindent\textbf{Poisson point process construction.}\ 
For each $\btheta \in \bTheta$, 
we can define a real tree $\te$ in the following way.
\begin{itemize}
\item If $\theta_0>0$, let $\rP_0=\{(u_j,v_j), j\ge 1\}$ be a Poisson point process on the 
first octant $\{(x,y): 0 \le y\le x\}$ of intensity measure $\theta_0^2dxdy$, 
ordered in such a way that $u_1<u_2<u_3<\cdots$. 
\item For every $i\ge 1$ such that $\theta_i>0$, 
let $\rP_i=\{\xi_{i,j}, j\ge 1\}$ be a homogeneous Poisson process on $\bbR_+$ of intensity 
$\theta_i$ under $\bbP$, such that $\xi_{i, 1}<\xi_{i, 2}<\xi_{i, 3}<\cdots$. 
\end{itemize}
All these Poisson processes are supposed to be mutually independent and defined on some common probability space
$(\Omega, \cF, \bbP)$.
We consider the points of all these processes 
as marks on the half line $\bbR_+$, among which we distinguish two kinds: the \emph{cutpoints} and 
the \emph{joinpoints}. A cutpoint is either $u_j$ for some $j\ge 1$ or $\xi_{i, j}$ for some 
$i\ge 1$ and $j\ge 2$. For each cutpoint $x$, we associate a joinpoint $x^*$ as follows: 
$x^*=v_j$ if $x=u_j$ for some $j\ge 1$ and $x^*=\xi_{i, 1}$ if $x=\xi_{i, j}$ for some $i\ge 1$ and $j\ge 2$.
One easily verifies that the hypotheses on $\btheta$ imply that the set of cutpoints is 
a.s.\ finite on each compact set of $\bbR_+$, while the joinpoints are dense a.s.\ everywhere. 
(See for example \cite{ald00} for a proof.) In particular,
we can arrange the cutpoints in increasing order as $0<\eta_1<\eta_2<\eta_3<\cdots$. 
This splits $\bbR_+$ into countaly intervals that we now reassemble into a tree.
We write $\eta_k^*$ for the joinpoint associated to the $k$-th cutpoint $\eta_k$. 
We define $R_1$ to be the metric space $[0, \eta_1]$ rooted at $0$. For $k\ge 1$, we let
$$
R_{k+1}:=R_k\mathop{\circledast}_{\eta_k^*=\eta_k}[\eta_k, \eta_{k+1}].
$$
In words, we graft the intervals $[\eta_k, \eta_{k+1}]$ by gluing the left end at the joinpoint 
$\eta_k^*$. Note that we have $\eta_k^*<\eta_k$ a.s., thus $\eta_k^*\in R_k$ and the 
above grafting operation is well defined almost surely. It follows from this Poisson construction 
that $(R_k)_{k\ge 1}$ is a consistent family of ``discrete'' trees which also verifies 
the ``leaf-tight" condition in \citet{aldcrt3}. 
Therefore by \cite[Theorem 3]{aldcrt3}, the complete metric 
space $\cT:=\overline{\cup_{k\ge 1}R_k}$ is a real tree and almost surely there exists a probability 
measure $\mu$, called the \emph{mass measure}, which is concentrated on the leaf set of $\cT$. 
Moreover, if conditional on $\te$, $(V_k, k\ge 1)$ is a sequence of i.i.d.\ points sampled according 
to $\mu$, then for each $k\ge 1$, the spanning tree
$\Span(T; V_1, V_2, \cdots, V_k)$
has the same unconditional distribution as $R_k$. 
The distribution of the weak isometry equivalence class of $(\cT, \mu)$ is said to be the distribution 
of an  \emph{ICRT of parameter $\btheta$}, which is a probability distribution on $\bbT_w$.
The push-forward of the Lebesgue measure 
on $\bbR_+$ defines a $\sigma$-finite measure $\ell$ on $\te$, which is concentrated on $\Sk(T)$ and 
called the \emph{length measure} of $\te$. 
Furthermore, it is not difficult to deduce the 
distribution of $\ell(R_1)$ from the above construction of $\te$:
\begin{equation}\label{eq: distD}
\p{\ell(R_1)>r}
=\p{\eta_1>r}
=e^{-\frac{1}{2}\theta_0^2r^2}\prod_{i\ge1}(1+\theta_i r)e^{-\theta_i r}, \quad r>0.
\end{equation}

In the important special case when $\btheta=(1, 0, 0, \cdots)$, the above construction
coincides with the line-breaking construction of the Brownian CRT in \cite[Algorithm 3]{aldcrt1}, 
that is, $\cT$ is the Brownian CRT. This case will be referred as the Brownian case in the sequel.
We notice that whenever there is an index $i\ge 1$ such that $\theta_i>0$, the point, denoted by 
$\beta_i$, which  corresponds to the joinpoint $\xi_{i, 1}$ is a branch point of infinite degree. 
According to \cite[Theorem 2]{ald04a}), $\theta_i$ is a measurable function of $(\cT, \beta_i)$, and 
we refer to it as the local time of $\beta_i$ in what follows.

\medskip
\noindent\textbf{ICRTs as scaling limits of $\bp$-trees.}\ 
Let $\bp_n=(p_{n1}, p_{n2}, \cdots, p_{nn})$ be a probability measure on $[n]$ such that 
$p_{n1}\ge p_{n2}\ge \cdots\ge p_{nn}>0$, $n\ge 1$. 
Define $\sigma_n\ge 0$ by $\sigma_n^2=\sum_{i=1}^n p_{ni}^2$ and denote by $\bT^n$ the 
corresponding $\bp_n$-tree, which we view as a metric space on $[n]$ with graph distance $d_{T_n}$. 
Suppose that the sequence $(\bp_n, n\ge 1)$ verifies the following 
hypothesis: there exists some parameter $\btheta=(\theta_i, i\ge 0)$ such that
\begin{equation}\tag{H}\label{H}
\lim_{n\to\infty}\sigma_n=0,
\qquad\text{and} \qquad
\lim_{n\to\infty}\frac{p_{ni}}{\sigma_n}=\theta_i, \quad \text{ for every }i\ge 1.
\end{equation}
Then, writing $\sigma_n \bT^n$ for the rescaled metric space $([n], \sigma_n d_{T^n})$, 
\citet{pit00} have shown that
\begin{equation}\label{eq: CP}
(\sigma_n \bT^n, \bp_n) \mathop{\longrightarrow}^{n\to\infty}_{d, \gp} (\te, \mu),
\end{equation}
where $\to_{d,\gp}$ denotes the convergence in distribution with respect to the Gromov--Prokhorov 
topology. 


\section{Main results}\label{sec:results}

\subsection{Cutting down procedures for $\bp$-trees and ICRT}

Consider a $\bp$-tree $\bT$. 
We perform a cutting procedure on $\bT$ by picking each time a vertex according to the 
restriction of $\bp$ to the remaining part; however, it is more convenient for us to 
retain the portion of the tree that contains a random node $V$ sampled according to $\bp$ rather 
than the root. We denote by $L(\bT)$ the number of cuts necessary until $V$ 
is finally picked, and let $X_i$, $1\le i\le L(\bT)$, be the sequence of nodes chosen.
The following identity in distribution has been already shown 
in \cite{ABH10} in the special case of the uniform Cayley tree:
\begin{equation}\label{eq: in_idl}
L(\bT)\eqd \Card \{\text{vertices on the path from the root to }V\}.
\end{equation}
In fact, \eqref{eq: in_idl} is an immediate consequence of the following result. In the above 
cutting procedure, we connect the rejected parts, which are subtrees above $X_i$ just before 
the cutting, by drawing an edge between $X_i$ and $X_{i+1}$, $i=1, 2, \cdots, L(\bT)-1$ 
(see Figure~\ref{fig:one-cutting} in Section~\ref{sec:ptree-cutting}). 
We obtain another tree on the same vertex set, which contains a path from the first cut 
$X_1$ to the random node $V$ that we were trying to isolate. 
We denote by $\cut(T,V)$ this tree which (partially) encodes the isolating process of $V$. 
We prove in Section~\ref{sec:ptree-cutting} that we have
\begin{equation}\label{eq: in_idt}
(\cut(T, V), V)\eqd (\bT, V).
\end{equation}
This identity between the pairs of trees contains a lot of information about the 
distributional structure of the $\bp$-trees, and our aim is to obtain results similar 
to \eqref{eq: in_idt} for ICRTs. The method we use relies on the discrete approximation of 
ICRT by $\bp$-trees, and a first step consists in defining the appropriate cutting 
procedure for ICRT. 

In the case of $\bp$-trees, one may pick the nodes of $T$ in the order in which they appear 
in a Poisson random measure. We do not develop it here but one should keep in mind 
that the cutting procedure may be obtained using a Poisson point process on $\R_+\times T$ 
with intensity measure $dt \otimes \bp$. In particular, this measure has a natural 
counterpart in the case of ICRTs, and it is according to this measure that the points 
should be sampled in the continuous case. 

So consider now an ICRT $\cT$. Recall that for $\btheta\ne (1,0,\dots)$, for each $\theta_i>0$ 
with $i\ge 1$, there exists a unique point, denoted by $\beta_i$, which has infinite degree. 
Let $\cL$ be the measure on $\te$ defined by
\begin{equation}\label{eq: defcL}
\cL(dx):=\theta_0^2 \ell(dx)+\sum_{i\ge 1} \theta_i\delta_{\beta_i}(dx),
\end{equation}
which is almost surely $\sigma$-finite (Lemma \ref{lem: cL}).
Proving that $\cL$ is indeed the relevant cutting measure (in a sense made precise in 
Proposition~\ref{prop: cv-Ln}) is the topic of Section~\ref{sec:cv-Ln}. 
Conditional on $\te$, let $\cP$ be a Poisson point process on $\bbR_+\times \te$ of 
intensity measure $dt\otimes \cL(dx)$ and let $V$ be a $\mu$-point on $\cT$.
We consider the elements of $\cP$ as the successive cuts on $\te$ which try to isolate 
the random point $V$. For each $t\ge 0$, define 
$$\cP_t=\{x\in \te: \exists\, s\le t \text{ such that }\, (s,x)\in \cP\},$$
and let $\cT_t$ be the part of $\cT$ still connected to $V$ at time $t$, that is 
the collection of points $u\in \cT$ for which the unique path in $\te$ from 
$V$ to $u$ does not contain any element of $\cP_t$. Clearly, $\cT_{t'}\subset\cT_t$ if $t'\ge t$. 
We set $\cC:=\{t>0 : \mu(\cT_{t-}) > \mu(\cT_t) \}.$
Those are the cuts which contribute to the isolation of $V$. 

\subsection{Tracking one node and the one-node cut tree}

We construct a tree which encodes this cutting process in a similar way that the tree 
$H=\cut(T,V)$ encodes the cutting procedure for discrete trees. 
First we construct the ``backbone'', which is the equivalent of the path we add in the discrete case. 
For $t\ge 0$, we define
$$
L_t:=\int_0^t \mu(\cT_s)ds,
$$
and $L_\infty$ the limit as $t\to\infty$ (which might be infinite).
Now consider the interval $[0, L_\infty]$, together with its Euclidean metric, that we think of as 
rooted at $0$. Then, for each $t\in \cC$ we graft $\cT_{t-}\setminus\cT_t$, the portion of the 
tree discarded at time $t$, at the point $L_t\in [0, L_\infty]$ (in the sense of the gluing 
introduced in Section~\ref{subset: icrt}). 
This creates a rooted real tree and we denote by $\cut(\cT,V)$ its completion. 
Moreover, we can endow $\cut(\cT, V)$ with a (possibly defective probability) measure $\hat{\mu}$ by 
taking the push-forward of $\mu$ under the canonical injection $\phi$ from 
$\cup_{t\in \cC}(\cT_{t-}\setminus\cT_t)$ to $\cut(\cT, V)$. 
We denote by $U$ the endpoint  $L_\infty$ of the interval $[0,L_\infty]$. 
We show in Section~\ref{sec:cont_cutting} that
\begin{thma}\label{thm:conv_cutv}
We have $L_\infty<\infty$ almost surely. Moreover, under \eqref{H} we have
$$
(\sigma_n \cut(T^n,V^n), \bp_n, V^n)
\mathop{\longrightarrow}^{n\to\infty}_{d,\gp} 
(\cut(\cT,V), \hat{\mu}, U),
$$
jointly with the convergence in \eqref{eq: CP}. 
\end{thma}

Combining this with \eqref{eq: in_idt}, we show in Section~\ref{sec:cont_cutting} that
\begin{thma}\label{thm:id_cutv}
Conditional on $\te$, $U$ has distribution $\hat{\mu}$, 
and the unconditional distribution of $(\cut(\cT,V), \hat{\mu})$ is the same as that of $(\te, \mu)$.
\end{thma}
Theorems~\ref{thm:conv_cutv} and \ref{thm:id_cutv} immediately entail that
\begin{cor}\label{cor: id_cutv}
Suppose that \eqref{H} holds. Then 
$$
\sigma_n L(\bT^n)\carrow L_\infty,
$$
jointly with the convergence in \eqref{eq: CP}. Moreover, the unconditional distribution of 
$L_\infty$ is the same as that of the distance in $\te$ between the root and a random point 
$V$ chosen according to $\mu$, given in \eqref{eq: distD}.
\end{cor}

\subsection{The complete cutting procedure}

In the procedure of the previous section, the fragmentation only takes place on the portions of the tree 
which contain the random point $V$. 
Following \citet{Bert12}, we consider a more general cutting procedure which 
keeps splitting all the connected components. 
The aim here is to describe the genealogy of the fragmentation that this cutting procedure produces.
For each $t\ge 0$, $\cP_t$ induces an equivalence relation $\sim_t$
on $\te$: for $x,y\in \cT$ we write $x\sim_t y$ if $\llb x, y\rrb \cap \cP_t=\emptyset$.
We denote by $\cT_x(t)$ the equivalence class containing $x$. In particular, we have 
$\cT_V(t)=\cT_t$.
Let $(V_i)_{i\ge 1}$ be a sequence of i.i.d.\ $\mu$-points in $\cT$.
For each $t\ge 0$, define $\mu_i(t)=\mu(\cT_{V_i}(t))$. 
We write $\bmu^\downarrow(t)$ for the sequence $(\mu_i(t), i\ge 1)$ rearranged in decreasing order. 
In the case where $\te$ is the Brownian CRT, the process $(\bmu^\downarrow(t))_{t\ge 0}$ is the 
fragmentation dual to the standard additive coalescent \cite{ald00}. In the other cases, however, 
it is not even Markov because of the presence of those branch points $\beta_i$ 
with fixed local times $\theta_i$.

As in \cite{Bert12}, we can define a genealogical tree for this fragmentation process. 
For each $i\ge 1$ and $t\ge 0$, let
$$
L_t^i:=\int_0^t \mu_i(s)ds,
$$
and let $L_\infty^i\in [0,\infty]$ be the limit as $t\to\infty$. 
For each pair $(i,j)\in \mathbb{N}^2$, let $\tau(i, j)=\tau(j, i)$ be the first moment when 
$\llb V_i,V_j\rrb$ contains an element of $\cP$ (or more precisely, its projection onto $\te$), 
which is almost surely finite by the properties of $\te$ and $\cP$.
It is not difficult to construct a sequence of increasing real trees $S_1\subset S_2\subset \cdots$ such that $S_k$ has the form of a discrete tree rooted at a point denoted $\rho_*$ with exactly $k$ leaves $\{U_1, U_2, \cdots, U_{k}\}$ satisfying
\begin{equation}\label{eq: skd}
d(\rho_*, U_i)=L^i_\infty, \quad
d(U_i, U_j)=L^i_\infty+L^j_\infty-2L^i_{\tau(i,j)}, \quad 1\le i<j \le k;
\end{equation}
where $d$ denotes the distance of $S_k$, for each $k\ge 1$.
Then we define 
$$\tc:=\overline{\cup_{k\ge 1}S_k},$$ 
the completion of the metric space $(\cup_k S_k, d)$, which is still a real tree. In the case where $\te$ is the Brownian CRT, the above definition of $\tc$ 
coincides with the tree defined by \citet{Bert12}.

Similarly, for each $\bp_n$-tree $\bT^n$, we can define a complete cutting procedure on $\bT^n$ by 
first generating a random permutation $(X_{n1}, X_{n2}, \cdots, X_{nn})$ on the vertex set $[n]$ 
and then removing $X_{ni}$ one by one. 
Here the permutation $(X_{n1}, X_{n2}, \dots, X_{nn})$ is constructed 
by sampling, for $i\ge 1$, $X_{ni}$ according to the restriction of $\bp_n$ to 
$[n]\setminus \{X_{nj}, j<i\}$.
We define a new genealogy on $[n]$ by making $X_{ni}$ an ancestor of $X_{nj}$ if $i<j$ 
and $X_{nj}$ and $X_{ni}$ are in the same connected component when $X_{ni}$ is removed. If we denote by 
$\cut(\bT^n)$ the corresponding 
genealogical tree, then the number of vertices in the path of $\cut(T^n)$ between the root 
$X_{n1}$ and an arbitrary vertex $v$ is precisely equal to the number of cuts necessary to 
isolate this vertex~$v$. We have
\begin{thma}\label{thm:conv_cut}
Suppose that \eqref{H} holds. Then, we have
$$
\big(\sigma_n\cut(\bT^n), \bp_n\big)
\mathop{\longrightarrow}^{n\to\infty}_{d,\gp}
\big(\tc, \nu\big),
$$
jointly with the convergence in \eqref{eq: CP}. 
Here, $\nu$ is the weak limit of the empirical measures $\frac{1}{k}\sum_{i=0}^{k-1}\delta_{U_i}$, 
which exists almost surely conditional on $\te$.
\end{thma}
From this, we show that
\begin{thma}\label{thm:id_cut}
Conditionally on $\te$, $(U_i, i\ge 0)$ has the distribution as a sequence of i.i.d. points of 
common law $\nu$.
Furthermore, the unconditioned distribution of the pair $(\tc, \nu)$ is the same as $(\te, \mu)$.
\end{thma}

In general, the convergence of the $\bp_n$-trees to the ICRT in \eqref{eq: CP} cannot be improved 
to Gromov--Hausdorff (GH) topology, see for instance \cite[][Example 28]{aldpitconex}. 
However, when the sequence $(\bp_n)_{n\ge 1}$ is suitably well-behaved, one does have this 
stronger convergence. (This is the case for example with $\bp_n$ the 
uniform distribution on $[n]$, which gives rise to the Brownian CRT, see also 
\cite[][Section 4.2]{ald04a}.) 
In such cases, we can reinforce accordingly the above convergences of the cut trees in the 
Gromov--Hausdorff topology. Note however that a "reasonable" condition on $\bp$ ensuring 
the Gromov--Hausdorff convergence seems hard to find. 
Let us mention a related open question in \cite[Section 7]{ald04a}, which is to determine a practical 
criterion for the compactness of a general ICRT.
 Writing $\to_{d,\ghp}$ for the convergence in distribution with
respect to the Gromov--Hausdorff--Prokhorov topology (see Section~\ref{sec:prelem}), we have
\begin{thma}\label{thm:cv_ghp}
Suppose that $\te$ is almost surely compact and suppose also as $n\to\infty$,
\begin{equation}\label{eq: introGHP}
\big(\sigma_n\bT^n, \bp_n\big)
\mathop{\longrightarrow}^{n\to\infty}_{d, \ghp} 
\big(\te, \mu\big).
\end{equation}
Then, jointly with the convergence in \eqref{eq: introGHP}, we have 
\begin{align*}
\big(\sigma_n \cut(\bT^n,V^n), \bp_n\big)
&\mathop{\longrightarrow}^{n\to\infty}_{d, \ghp} 
\big(\cut(\cT,V), \hat{\mu}\big), \\
\big(\sigma_n\cut(\bT^n), \bp_n\big)
&\mathop{\longrightarrow}^{n\to\infty}_{d,\ghp}
\big(\tc, \nu\big).
\end{align*} 
\end{thma}

\subsection{Reversing the cutting procedure}

We also consider the transformation that ``reverses'' the construction of the trees 
$\cut(\cT,V)$ defined above. Here, by reversing we mean to obtain 
a tree distributed as the primal tree $\cT$, conditioned on the cut tree being the one we need 
to transform. So for an ICRT $(\cH, d_{\cH}, \hat\mu)$ and a random point $U$ sampled according 
to its mass measure $\hat\mu$, we should construct a tree $\shuff(\cH,U)$ 
such that 
\begin{equation}\label{eq: idshuff}
(\cT,\cut(\cT,V))\eqd (\shuff(\cH,U), \cH).
\end{equation}
This reverse transformation is the one described in \cite{ABH10} for the Brownian CRT. 
For $\cH$ rooted at $r(\cH)$, 
the path between $\llb r(\cH),U\rrb$ that joins $r(\cH)$ to $U$ in $\cH$ decomposes the tree into 
countably many subtrees of positive mass
$$F_x=\{y\in \cH: U\wedge y = x\},$$
where $U\wedge y$ denotes the closest common ancestor of $U$ and $y$, that is the unique point 
$a$ such that $\llb r(\cH),U\rrb\cap \llb r(\cH), y\rrb=\llb r(\cH),a\rrb$. Informally, the tree $\shuff(\cH,U)$ 
is the metric space one obtains from $\cH$ by attaching 
each $F_x$ of positive mass at a random point $A_x$, which is sampled 
proportionally to $\hat \mu$ in the union of the $F_y$ for which $d_{\cH}(U,y)<d_{\cH}(U,x)$.
We postpone the precise definition of $\shuff(\cH, U)$ until Section~\ref{sec:shuff-one-path}.

The question of reversing the complete cut tree $\cut(\cT)$ is more delicate and is 
the subject of the companion paper \cite{BrWa2014a}. There we restrict ourselves to the 
case of a Brownian CRT: for $\cT$ and $\cG$ Brownian CRT we construct a tree $\shuff(\cG)$ 
such that 
$$
(\cT, \cut(\cT)) \eqd (\shuff(\cG), \cG).
$$
We believe that the construction there is also valid for more general ICRTs, but the 
arguments we use there strongly rely on the self-similarity of the Brownian CRT.

\medskip
\noindent\textbf{Remarks.}
\noi{\bf i.}\ 
Theorem~\ref{thm:id_cutv} generalizes Theorem~1.5 in \cite{ABH10}, which is about the Brownian CRT. 
The special case of Theorem~\ref{thm:conv_cutv} concerning the convergence of uniform Cayley trees 
to the Brownian CRT is also found there.

{\bf ii.}\
When $\te$ is the Brownian CRT, Theorem~\ref{thm:id_cut} has been proven by \citet{Bert12}. 
Their proof relies on the self-similar property of the Aldous--Pitman's fragmentation. They 
also proved a convergence similar to the one in Theorem~\ref{thm:conv_cut} for the conditioned 
Galton--Watson trees with finite-variance offspring distributions. 
Let us point out that their definition of the discrete 
cut trees is distinct from ours, and there is no ``duality'' at the discrete level for their 
definitions. Very recently, a result related to Theorem~\ref{thm:conv_cut} has been proved for 
the case of stable trees \cite{Dieuleveut2013a} (with a different notion of discrete cut tree).
Note also that the convergence of the cut trees proved in \cite{Bert12} and \cite{Dieuleveut2013a} 
is with respect to the Gromov--Prokhorov topology, so is weaker than the convergence of the 
corresponding conditioned Galton--Watson trees, which holds in the Gromov--Hausdorff--Prokhorov 
sense. In our case, the identities imply that the convergence of the cut trees is as 
strong as that of the $\bp_n$-trees (Theorem~\ref{thm:cv_ghp}).

{\bf iii.}\
\citet{AbDe12} have shown an analog of Theorem~\ref{thm:id_cutv} for the L\'evy tree, 
introduced in \cite{legalllejan}. In passing \citet{ald04a} have conjured that a L\'evy tree is a 
mixture of the ICRTs where the parameters $\btheta$ are chosen according to the distribution of 
the jumps in the bridge process of the associated L\'evy process. Then the similarity between 
Theorem~\ref{thm:id_cutv} and the result of Abraham and Delmas may be seen as a piece of evidence 
supporting this conjecture. 

\section{Cutting down and rearranging a $\bp$-tree}\label{sec:ptree-cutting}

As we have mentioned in the introduction, our approach to the theorems about
continuum random trees involves taking limits in the discrete world.
In this section, we prove the discrete results about the decomposition and the 
rearrangement of $\bp$-trees that will enable us to obtain similar decomposition and 
rearrangement procedures for inhomogeneous continuum random trees.

\subsection{Isolating one vertex}\label{sec:ptree-one-cutting}

As a warm up, and in order to present many of the important ideas, we start by
isolating a single node.
Let $T$ be a $\bp$-tree and let $V$ be an independent $\bp$-node.
We isolate the
vertex $V$ by removing each time a random vertex of $T$ and preserving only the
component containing $V$ until the time when $V$ is picked.

\medskip
\noi \textsc{The 1-cutting procedure and the $1$-cut tree.}\
Initially, we have $T_0=T$, and an independent vertex $V$.
Then, for $i\ge 1$, we choose a node $X_i$ according
to the restriction of $\bp$ to the vertex set $\fv(T_{i-1})$ of $T_{i-1}$.
We define $T_{i}$ to be the connected component of the forest induced by $T_{i-1}$ on
$\fv(T_{i-1})\setminus \{X_i\}$ which contains $V$. If $T_i=\varnothing$, or equivalently
$X_i=V$, the process stops and we set $L=L(T)=i$.
Since at least one vertex is removed at every step, the process stops in time $L\le n$.

As we destruct the tree $T$ to isolate $V$ by iteratively pruning random nodes,
we construct a tree which records the history of the destruction, that we call
the $1$-cut tree. This $1$-cut tree will, in particular, give some information
about the number of cuts which were needed to isolate $V$.
However, we remind the reader that this number of cuts is not our main
objective, and that we are after a more detailed correspondence between the initial
tree and its $1$-cut tree. We will prove that these two trees are \emph{dual} in a
sense that we will make precise shortly.

By construction, $(T_i, 0\le i< L)$ is a decreasing sequence of nonempty trees
which all contain $V$, and $(X_i, 0\le i<L)$ is a sequence of distinct vertices
of~$T=T_0$.
Removing $X_i$ from $T_{i-1}$ disconnects it into a number of connected components.
Then $T_i$ is the one of these connected components which contains $V$, or
$T_i=\varnothing$ if $X_i=V$.
If $X_i=V$ we set $F_i=T_{i-1}$, which we see as a tree rooted at $X_i=V$.
Otherwise $X_i\ne V$ and there is a neighbor $U_i$ of $X_i$ on the path
between $X_i$ and $V$ in $T_{i-1}$. Then $U_i\in T_i$ and we see $T_i$ as rooted
at $U_i$;
furthermore, we let $F_i$ be the subtree of $T_{i-1}\setminus \{U_i\}$
which contains $X_i$, seen as rooted at $X_i$. 
In other words, $T_i$ is the subtree with vertex set $\{u\in T_{i-1}: X_i\notin \llb u, V\rrb\}$ 
rooted at $U_i$ and $F_i$ is the subtree with vertex set $\{u\in T_{i-1}: X_i\in \llb u, V\rrb\}$ 
rooted at $X_i$.

When the procedure stops, we have a vector $(F_i, 1\le i \le L)$ of subtrees of $T$
which together span all of $[n]$. We may re-arrange them into a new tree, the $1$-cut
tree corresponding to the isolation of $V$ in $T$. We do this by connecting their
roots $X_1, X_2,\dots, X_L$ into a path (in this order). 
The resulting tree, denoted by $H$ is seen as rooted at
$X_1$, and carries a distinguished path or backbone $\llb X_1, V\rrb$, which we denote
by $S$, and distinguished points $U_1,\dots, U_{L-1}$.

Note that for $i=1,\dots, L-1$, we have $U_i\in T_i$. Equivalently, $U_i$ lies in the
subtree of $H$ rooted at $X_{i+1}$. In general, for a tree $t\in \bbT_n$ and $v\in [n]$,
let $x_1,\dots, x_\ell=v$ be the nodes of $\Span(t; v)$.
We define $\bbU(t,v)$ as the collection of vectors $(u_1,\dots, u_{\ell-1})$
of nodes of $[n]$ such that $u_i\in \Sub(t, x_{i+1})$, for $1\le i<\ell$.
Then by construction, for a $h\in \bbT_n$, conditional on $H=h$ and $V=v$, we 
have $L$ equal to the number of the nodes in $\Span(h; v)$ and 
$(U_1,\dots, U_{L-1})\in \bbU(h,v)$ with probability one.
For $A\subseteq [n]$, we write $\bp(A):=\sum_{i\in A}p_i$.

\begin{lem}\label{lem:joint_dist-tree-node}
Let $T$ be a $\bp$-tree on $[n]$, and $V$ be an independent $\bp$-node.
Let $h\in \bbT_n$, and $v\in [n]$ for which $\Span(h; v)$ is the path
made of the nodes $x_1,x_2,\dots, x_{\ell-1},x_\ell=v$.
Let $(u_1,\dots, u_{\ell-1})\in \bbU(h, v)$ and $w\in [n]$. Then we have
$$
\p{H=h; V=v; r(T)=w; U_i=u_i, 1\le i<\ell}
= \pi(h)\cdot \prod_{1\le i<\ell}\frac{p_{u_i}}{\bp(\Sub(h,x_{i+1}))}\cdot p_v \cdot p_w.
$$
In particular, $(H, V)\sim \pi\otimes\bp$.
\end{lem}

As a direct consequence of our construction of $H$, $L$ is the number of nodes of the 
subtree $\Span(H,V)$, which we write $\# \Span(H, V)$.
So Lemma~\ref{lem:joint_dist-tree-node} entails immediately that
\begin{prop}\label{prop:number_cuts}
Let $T$ be a $\bp$-tree and $V$ be an independent $\bp$-node. Then
$$ L\eqd \# \Span(T, V). $$
\end{prop}

\begin{proof}[Proof of Lemma~\ref{lem:joint_dist-tree-node}]
By construction, we have
$$
\{H=h; V=v\}\subset\{X_1=x_1, \cdots, X_{\ell-1}=x_{\ell-1}, X_\ell=v; L=\ell\},
$$
and the sequence $(F_i, 1\le i\le\ell)$ is precisely the sequence of subtrees $f_i$,
of $h$ rooted at $x_i$, $1\le i\le \ell$, that are obtained when one removes
the edges
$\{x_i,x_{i+1}\}$, $1\le i< \ell$ (the edges of the subgraph $\Span(h;v)$).
Furthermore, given that $L=\ell$ and the sequence of cut vertices $X_i=x_i$,
$1\le i<\ell$, in order to recover the initial tree $T$ it suffices to identify the
vertices $U_i$, $1\le i< \ell$, for which there used to be an edge $\{X_i,U_i\}$ 
(which yields the correct adjacencies) and the root of $T$.
Note that $U_i$ is a node of $T_i$, $1\le i< \ell$.
However, by construction, given that $H=h$ and $V=v$, the set of nodes
of $T_{i}$ is precisely the set of nodes of $\Sub(h,x_{i+1})$, the subtree of $h$
rooted at $x_{i+1}$.

\begin{figure}[t]
\centering
\includegraphics[scale=.7]{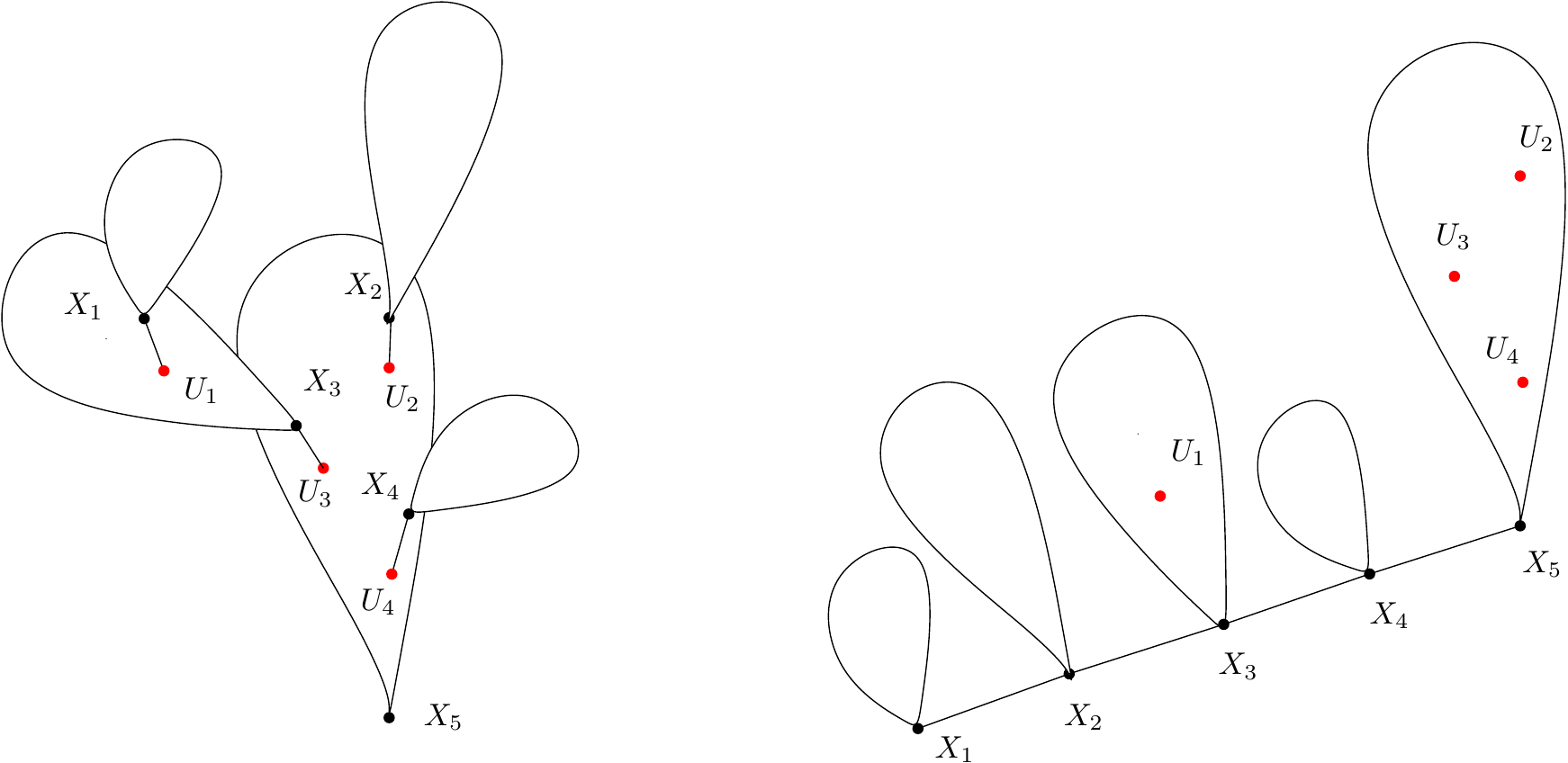}
\caption{\label{fig:one-cutting}The re-organization of the tree in the one-cutting procedure: on the 
left the initial tree $T$, on the right $H$ and the marked nodes $U_1,\dots, U_4$ where to reattach 
$X_1,\dots, X_4$ in order to recover $T$.}
\end{figure}

For $\bu=(u_1, \cdots, u_{\ell-1})\in \bbU(h, v)$, define $\tau(h, v; \bu)$
as the tree obtained from $h$ by removing the edges of $\Span(h;v)$, and reconnecting the 
pieces by adding the edges
$\{x_i, u_i\}$, for all the edges $\edge{x_i, x_{i+1}}$ in $\Span(h, v)$. (In particular, the number of 
edges is unchanged.)
We regard $\tau(h, v; \bu)$ as a tree rooted at~$r=x_1$, the root of $h$.
The the tree $T$ may be recovered by characterizing $T^{r}$, 
the tree $T$ rerooted at $r$, and the initial root $r(T)$.
We have:
$$
\{H=h; V=v; r(T)=w; U_i=u_i, 1\le i< \ell\}
=\{T^{r}=\tau(h, v; \bu); r(T)=w; X_i=x_i, 1\le i\le \ell\}.
$$
It follows that, for any nodes $u_1,u_2,\dots, u_{\ell-1}$ as above, we have
\begin{align*}
\p{H=h; V=v; r(T)=w; U_i=u_i, 1\le i<\ell}
& = \p{T=\tau(h,v; \bu)^w; V=v; X_i=x_i; 1\le i\le \ell}\\
& = \pi(\tau(h, v; \bu)^w) \cdot p_v \cdot
    \prod_{1\le i\le \ell}\frac{p_{x_i}}{\bp(\Sub(h,x_i))}.
\end{align*}
Now, by definition, the only nodes that get their (in-)degree modified in the transformation 
from $h$ to $\tau(h,v;\bu)$ are $u_i$, $x_{i+1}$, $1\le i< \ell$: every such 
$x_{i+1}$ gets one less in-edge while $u_i$ gets one more. The re-rooting at $w$ then only 
modifies the in-degrees of the extremities of the path that is reversed, namely 
$x_1=r$ and $w$.
It follows that
$$\pi(\tau(h,v;\bu)^w) 
= \pi(h) \cdot \prod_{1\le i<\ell} \frac{p_{u_i}}{p_{x_{i+1}}}\cdot \frac{p_w}{p_{x_1}}.$$
Since $\bp(\Sub(h,x_1))=1$, we have
$$
\p{H=h; V=v; r(T)=w; U_i=u_i, 1\le i<\ell}
= \pi(h)\cdot
    \prod_{1\le i< \ell}\frac{p_{u_i}}{\bp(\Sub(h,x_{i+1}))}\cdot p_v \cdot p_w,
$$
which proves the first claim.
Summing over all the choices for $\bu=(u_1,u_2,\dots, u_{\ell-1})\in \bbU(h, v)$, 
and $w\in [n]$, we obtain
\begin{align*}
\p{H=h; V=v}
=&\sum_{w\in [n]} \sum_{\bu\in \bbU(h, v)} \pi(h)\cdot
    \prod_{1\le i< \ell}\frac{p_{u_i}}{\bp(\Sub(h,x_{i+1}))}\cdot p_v \cdot p_w\\
=&\,\pi(h)\cdot p_v\cdot \sum_{\begin{subarray}{c}
        \bu=(u_1, \cdots, u_{\ell-1}):\\  u_i\in \Sub(h,x_{i+1}),1\le i<\ell
      \end{subarray}} 
    \frac{p_{u_1}}{\bp(\Sub(h,x_{2}))}\cdots\frac{p_{u_{\ell-1}}}{\bp(\Sub(h,x_{\ell}))}\\
=&\,\pi(h)\cdot p_v,
\end{align*}
which completes the proof.
\end{proof}

\medskip
\noi\textsc{The reverse 1-cutting procedure.}\
We have transformed the tree $T$ into the tree $H$, by somewhat
``knitting'' a path between the first picked random $\bp$-node $X_1$ and
the distinguished node $V$. This transform is reversible. 
Indeed, it is possible to
``unknit'' the path between $V$ and the root of $H$, and reshuffle the subtrees
thereby created in order to obtain a new tree $\tilde T$, distributed as $T$ and
in which $V$ is an independent $\bp$-node.
Knowing the $U_i$, one could do this exactly, and recover the adjacencies of $T$
(recovering $T$ also requires the information about the root $r(T)$ which has been lost).
Defining a reverse transformation reduces to finding the joint distribution of $(U_i)$ and $r(T)$,
which is precisely the statement of Lemma~\ref{lem:joint_dist-tree-node},
so that the following reverse construction is now straightforward.

Let $h\in \bbT_n$, rooted at $r$ and let $v$ be a node in $[n]$.
We think of $h$ as the tree that was obtained by the $1$-cutting procedure $\cut(T,v)$,
for some initial tree $T$.
Suppose that $\Span(h, v)$ consists of the vertices $r=x_1,x_2,\dots, x_\ell=v$.
Removing the edges of $\Span(h, v)$ from $h$ disconnects it into $\ell$
connected components which we see as rooted at $x_i$, $1\le i\le \ell$.
For $w\in \Span^*(h, v)=\Span(h,v)\setminus \{r\}$, sample a node $U_w$ according to the 
restriction of $\bp$ to $\Sub(h,w)$. 
Let $\bU=(U_w, w\in \Span^*(h, v))$ be the obtained vector. Then $\bU\in \bbU(h, v)$. 
We then define $\shuff(h, v)$ to be the rooted tree which has the adjacencies
of $\tau(h, v; \bU)$, but that is re-rooted at an independent $\bp$-node.

It should now be clear that the $1$-cutting procedure and the reshuffling operation
we have just defined are dual in the following sense.

\begin{prop}[1-cutting duality]\label{pro:one-duality}
Let $T$ be $\bp$-tree on $[n]$ and $V$ be an independent $\bp$-node.
Then,
$$(\shuff(T,V), T, V)\overset{d}{=}(T, \cut(T,V), V).$$
In particular, $(\shuff(T, V), V)\sim \pi\otimes \bp$.
\end{prop}
Note that for the joint distribution in Proposition~\ref{pro:one-duality}, 
it is necessary to re-root at
another independent $\bp$-node in order to have the claimed equality.
Indeed, $T$ and $\tau(T,V;\bU)$ have the same root almost surely,
while $T$ and $\cut(T,V)$ do not
(they only have the same root with probability $\sum_{i\ge 1} p_i^2<1$).

\begin{proof}[Proof of Proposition~\ref{pro:one-duality}]
Let $H=\cut(T,V)$ be the tree resulting from the cutting procedure.
Let $L=\#\Span(H; V)$. For $1\le i<L$, we defined nodes $U_i$,
which used to be the neighbors of $X_i$ in $T$.
For $w\in \Span^*(H;V)$, we let $U_w=U_i$ if $w=X_{i+1}$, and let $\bU$
be the corresponding vector. 
Then writing $\hat r=r(T)$, with probability one, we have
$$T=\tau(H, V; \bU)^{\hat r}.$$
By Lemma~\ref{lem:joint_dist-tree-node}, $\bU\in \bbU(H, V)$ and conditional on
$H$ and $V$, $U_w$, $w\in \Span^*(H,V)$ and $\hat r=r(T)$ are independent and distributed 
according to the restriction of $\bp$ to $\Sub(H,w)$ and $\bp$, respectively. 
So this coupling indeed gives that $T=\tau(H,V;\bU)^{\hat r}$ is distributed as $\shuff(H,V)$, 
conditional on $H$. 
Since in this coupling $(\shuff(H,V), T, V)$ is almost surely equal to $(T,H,V)$, 
the proof is complete.
\end{proof}

\medskip
\noindent\textbf{Remark.}\ \label{rem:as-coupling}
Note that the shuffle procedure would permit to obtain 
the original tree $T$ exactly if we were to use some information that might be 
gathered as the cutting procedure goes on. In this discrete case, this is rather 
clear that one could do this, since the shuffle construction only consists in 
replacing some edges with others but the vertex set remains the same. 
This observation will be used in Section~\ref{sec:reverse} to prove a similar
statement for the ICRT. There it is much less clear and the result is
slightly weaker: it is possible to couple the shuffle in such a way 
that the tree obtained is measure-isometric to the original one. 

\subsection{Isolating multiple vertices}\label{sec:isol_mult}

We define a cutting procedure analogous to the one described in
Section~\ref{sec:ptree-one-cutting}, but which continues until multiple nodes have been
isolated. Again, we let $T$ be a $\bp$-tree and, for some
$k\ge 1$, let $V_1, V_2, \cdots, V_k$ be $k$ independent vertices chosen according to $\bp$ 
(so not necessarily distinct).

\medskip
\noi \textsc{The $k$-cutting procedure and the $k$-cut tree.}\
We start with $\Gamma_0=T$. Later on, $\Gamma_i$ is meant to be the forest induced
by $T$ on the nodes that are left.
For each time $i\ge 1$, we pick a random vertex $X_i$ according to $\bp$
restricted to $\fv(\Gamma_{i-1})$, the set of the remaining vertices, and remove it.
Then among the connected components of $T\setminus\{X_1, \cdots, X_i\}$,
we only keep those containing at least one of $V_1, \cdots, V_k$.
We stop  at the first time when all $k$ vertices
$V_1,\dots, V_k$ have been chosen, that is at time
$$L^k:=\inf\{i\ge 1: \{V_1,\dots, V_k\}\subseteq \{X_1,\dots, X_i\} \}.$$
For $1\le \ell\le k$ and for $i\ge 0$, we denote by $T^\ell_i$ the connected component of
$T\setminus\{X_1, X_2, \cdots, X_i\}$ containing $V_\ell$ at time $i$, or
$T^\ell_i=\varnothing$ if $V_\ell\in \{X_1,\dots, X_i\}$.
Then $\Gamma_i$ is the graph consisting of the connected components
$T_i^\ell$, $\ell=1,\dots, k$.

Fix some $\ell\in \{1,2,\dots, k\}$, and suppose that at time $i\ge 1$,
we have $X_i\in T_{i-1}^\ell$.
If $X_i=V_\ell$, then $T_i^\ell=\varnothing$ and we define $F_i=T_{i-1}^\ell$,
re-rooted at $X_i=V_\ell$.
Otherwise, $X_i\ne V_\ell$ and there is a first node $U_i^\ell$ on the path between
$X_i$ and $V_\ell$ in $T_{i-1}^\ell$. Then $U_i^\ell\in T_i^\ell$, and we see $T_i^\ell$
as rooted at $U_i^\ell$.
Note that it is possible that $T_{i-1}^j=T_{i-1}^\ell$, for $j\ne \ell$, and that
removing $X_i$ may separate $V_\ell$ from $V_j$.
Removing from $\Gamma_{i-1}$ the edges $\{X_i,U_i^\ell\}$, 
for $1\le \ell\le k$ such that $T_i^\ell \ni X_i$,
isolates $X_i$ from the nodes $V_1,\dots, V_k$, and we define $F_i$ as the
subtree of $T$ induced on the nodes in $\Gamma_{i-1} \setminus \Gamma_i,$
so that $F_i$ is the portion of the forest $\Gamma_{i-1}$
which gets discarded at time $i$, which we see as rooted at $X_i$.

\begin{figure}[t]
\centering
\includegraphics[scale=.7]{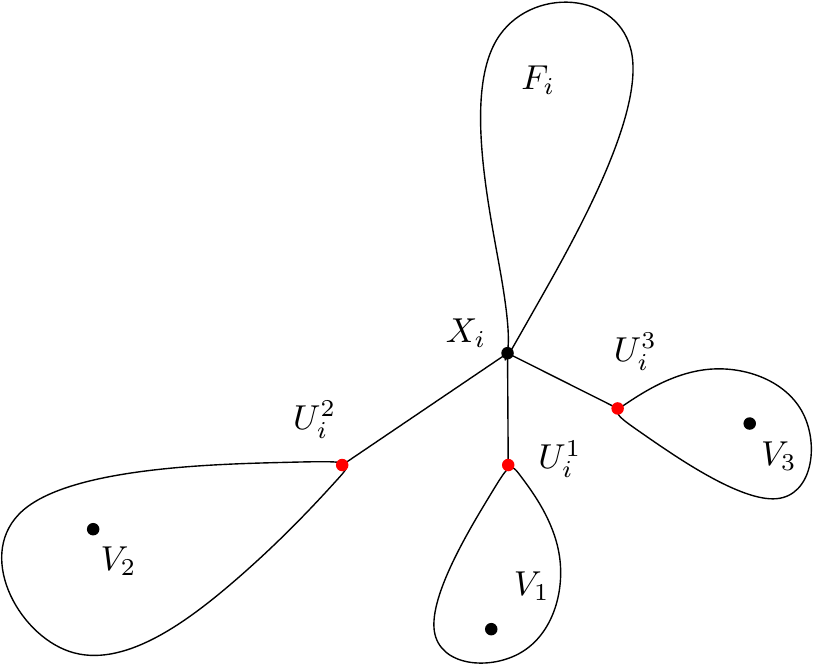}
\caption{\label{fig:cutting_Fi}The decomposition of the tree when removing the point $X_i$ from
the connected component of $\Gamma_i$ which contains $V_1,V_2$ and $V_3$.}
\end{figure}

Consider the set of effective cuts which affect
the size of $T^\ell_i$:
$$
\cE^k_\ell=\{x\in [n]: \text{there exists }i\ge 1, \text{ such that } X_i=x\in T^\ell_{i-1}\},
$$
and note that $\cE^k_1 \cup \cE^k_2 \cup \dots \cup \cE^k_k=\{X_i: 1\le i\le L^k\}$.
Let $S_k$, the \emph{$k$-cutting skeleton}, 
be a tree on $\cE^k_1 \cup \dots \cup \cE^k_k$ that is rooted at $X_1$,
and such that the vertices on the path from $X_1$ to $V_\ell$ in $S_k$
are precisely the nodes of $\cE^k_\ell$, in the order given by the indices of the cuts.
So if we view $S_k$ as a genealogical tree, then in particular, for $1\le j,\ell\le k$, 
the common ancestors
of $V_j$ and $V_\ell$ are exactly the ones in $\cE^k_j \cap \cE^k_\ell$.
The tree $S_k$ constitutes the \emph{backbone} of a tree on $[n]$ which we now
define.
For every $x\in S_k$, there is a unique $i=i(x)\ge 1$ such that $x=X_i$.
For that integer $i$ we have defined a subtree $F_i$ which contains $X_i=x$.
We append  $F_i$ to $S_k$ at $x$.
Formally, we consider the tree on $[n]$ whose
edge set consists of the edges of $S_k$ together with the edges of all $F_i$,
$1\le i\le L^k$. Furthermore, the tree is considered as rooted at $X_1$.
Then this tree is completely determined by $T$, $V_1,\dots, V_k$, and the sequence
$\bX:=(X_i, i=1, \dots, L^k)$, and
we denote this tree by $\kappa(T; V_1,\dots, V_k; \bX)$ when we want to
emphasize the dependence in $\bX$, or more simply $\cut(T, V_1,\dots, V_k)$
(in which it is implicit that the cutting sequence used in the transformation
is such that for every $i\ge 1$, $X_i$ is a $\bp$-node in $\Gamma_{i-1}$).
Clearly, if $H_k=\cut(T, V_1,\dots, V_k)$ then
$S_k=\Span(H_k; V_1,\dots, V_k)$.

It is convenient to define a \emph{canonical (total) order} $\preceq$ on the vertices of $S_k$.
\label{p:canon-order}
It will be needed later on in order to define the reverse procedure. 
For two nodes $u,v$ in $S_k$, we say that $u\preceq v$ if either
$u\in \llb X_1, v\rrb$, or
if there exists $\ell\in \{1,\dots, k\}$ 
such that $u\in \Span(S_k; V_1,\dots, V_\ell)$ but
$v\not\in \Span(S_k; V_1,\dots, V_\ell)$.

\medskip
\noi\textsc{A useful coupling.}\ \label{p:coupling}
It is useful to see all the trees
$\cut(T; V_1,\dots, V_k)$ on the same probability space, and provide
a natural but crucial coupling for which the sequence $(S_k)$ is increasing in $k$.
Let $Y_i$, $i\ge 1$, be a sequence of i.i.d.\ $\bp$-nodes.
For $k\ge 1$, we define an increasing sequence $\sigma_k$ as follows.
Let $\sigma_k(1)=1$.
Suppose that we have already defined $X^k_1,\dots, X^k_{i-1}$.
Let $\Gamma_{i-1}^k$ be the collection of
connected components of $T\setminus \{X_1^k,\dots, X_{i-1}^k\}$ which
contain at least one of $V_1,\dots, V_k$. Let
$$\sigma_k(i)=\inf\{j>\sigma_k(i-1): Y_j\in \Gamma_{i-1}^k\},$$
and define $X_i^k=Y_{\sigma_k(i)}$.
Then, for every $k$, $X^k_i$, $i\ge 1$, is a sequence of nodes sampled according
to the restriction of $\bp$ to $\Gamma_{i-1}^k$, so that $\bX^k:=(X^k_i, i\ge 1)$ can be
used to define $\cut(T, V_1,\dots, V_k)$, $k\ge 1$, in a consistent way by setting
$$\cut(T, V_1,\dots, V_k)=\kappa(T, V_1,\dots, V_k; \bX^k).$$
Suppose that the trees $H_k:=\cut(T;V_1,\dots, V_k)$, $k\ge 1$, are
constructed using the coupling we have just described.
By convention let $H_0=T$ and $\Span(T; \varnothing)=\varnothing$.

\begin{lem}\label{lem:coupl-incr}
Let $S_k=\Span(H_k; V_1,\dots, V_k)$.
Then, $S_k\subseteq S_{k+1}$ and
$$S_k=\Span(S_{k+1}; V_1,\dots, V_k).$$
\end{lem}
\begin{proof}
Let $T_i^{\ell}$ be the connected component of $\Gamma_i^k$ which contains $V_\ell$.
Let $\hat T_j^\ell$ be the connected component of
$T\setminus\{Y_1,Y_2,\dots, Y_j\}$ which contains $V_\ell$.
Then, for $\ell\le k$, we have
$$\cE_\ell^k=\{x: \exists i\ge 1, x=X_i^k\in T_{i-1}^\ell\}
=\{y: \exists j\ge 1, y=Y_j\in \hat T_{j-1}^\ell\},$$
so that $\cE_\ell^k$ does not depend on $k$.
Then $S_k$ is the tree on $\cE_1^k\cup \dots \cup \cE_k^k$ such that the nodes
on the path $\Span(S_k; V_\ell)$ are precisely the nodes of $\cE_\ell^k$, in the
order given by the cut sequence $\bX^k$.
It follows that $S_k\subseteq S_{k+1}$ and more precisely that
$S_k=\Span(S_{k+1}; V_1,\dots, V_k)$.
\end{proof}

\medskip
\noi\textbf{Remark.}\
The coupling we have just defined justifies an
\emph{ordered cutting procedure} which is very similar to the one defined in
\cite{ABH10}.
Suppose that, for some
$j,\ell\in \{1,\dots, k\}$ we have $x\in \cE^k_j\setminus \cE^k_\ell$ and
$y\in \cE^k_\ell\setminus \cE^k_j$. Write $(\tilde X_i, i\ge 1)$ for the
sequence in which we have exchanged the positions of $x$ and $y$.
Then the trees $T_i^k$, $i\ge \max\{m: X_m=x\text{ or }y\}$
are unaffected if we replace $(X_i, i\ge 1)$ by $(\tilde X_i, i\ge 1)$
in the cutting procedure. In particular, if we are only interested in the final tree
$H_k$, we can always suppose that there exist numbers
$0=m_0< m_1< m_2< \cdots < m_k\le n$ such that, for $1\le \ell\le k$, and if
$V_\ell\not\in\{V_1,\dots, V_j\}$, we have
$$
\cE^k_\ell\setminus \bigcup_{1\le j< \ell}\cE^k_j=\{X_i: m_{\ell-1}< i\le m_{\ell}\}.
$$
However, we prefer the coupling over the reordering of the sequence
since it does not involve any modification of the distribution of the cutting sequences.

\medskip
Let $\tilde T_k$ be the subtree of $H_{k-1}\setminus \Span(H_{k-1}; V_1,\dots, V_{k-1})=H_{k-1}
\setminus S_{k-1}$ 
which contains $V_k$; we agree that $\tilde T_k=\varnothing$ if
$V_k\in \Span(H_{k-1}; V_1,\dots, V_{k-1})$.

\begin{lem}\label{lem:multi-cut}
Let $T$ be a $\bp$-tree and let $V_k$, $k\ge 1$, be a sequence of i.i.d.\ $\bp$-nodes.
Then, for each $k\ge 1$:
\begin{enumerate}[i.]
\item\label{lmii}
Let $\bV\subseteq [n]$ with $\bV \ne \varnothing$, 
then conditional on $V_\ell\in \fv(\tilde T_k)=\bV$, 
the pair $(\tilde T_k, V_\ell)$
is distributed as $\pi|_\bV\otimes \bp|_\bV$,
and is independent of $(H_{k-1}\setminus \bV, V_1, \cdots, V_{k-1})$.
\item\label{lmi}
The joint distribution of $(H_k, V_1, \cdots, V_k)$ is given by $\pi\otimes \bp^{\otimes k}$.
\end{enumerate}
\end{lem}

\begin{proof}
We proceed by induction on $k\ge 1$.
Let $\tilde R_k$ denote the tree induced by $H_k$ on the vertex set
$[n]\setminus \fv(\tilde T_k)$.
For the base case $k=1$, the first claim is trivial since $\tilde T_1=T$,
and the second is exactly the statement of Lemma~\ref{lem:joint_dist-tree-node}.

Given the two subtrees $\tilde T_k$ and $\tilde R_k$, it suffices to identify where
the tree $\tilde T_k$ is grafted on $\tilde R_k$ in order to recover the tree $H_{k-1}$.
By construction, the edge connecting $\tilde T_k$ and $\tilde R_k$ in $H_{k-1}$
binds the root of $\tilde T_k$ to a node of $\Span(\tilde R_k; V_1,\dots, V_{k-1})$.
Let $t\in \bbT_\bV$, $r\in \bbT_{[n]\setminus \bV}$, $v_k\in \bV$ and
$v_i\in [n]\setminus \bV$ for $1\le i<k$. Write $\bv_{k-1}=\{v_1,\dots, v_{k-1}\}$.
For a given node $x\in \Span(r; \bv_{k-1})$, let $j_x(r,t)$
(the joint of $r$ and $t$ at $x$) be the tree obtained from
$t$ and $r$ by adding an edge between $x$ and the root of $t$.
By the induction
hypothesis, $(H_{k-1}, V_1, \cdots, V_{k-1})$ is  distributed
like a $\bp$-tree together with $k-1$ independent $\bp$-nodes.
Furthermore $V_k$ is independent of $(H_{k-1}, V_1, \cdots, V_{k-1})$.
It follows that
\begin{align*}
\pc{\tilde T_k=t; \tilde R_k=r; V_i=v_i, 1\le i\le k}
& = \sum_{x\in \Span(r; \bv_{k-1}) }\pc{H_{k-1}= j_x(r,t); V_i=v_i, 1\le i\le k} \notag \\
& = \sum_{x\in \Span(r; \bv_{k-1})}\prod_{i\in \bV}p_i^{C_i(t)} \cdot
\prod_{j\in [n]\setminus \bV}p_j^{C_j(r)}\cdot p_x
\cdot \prod_{1\le i\le k}p_{v_i} \notag \\
& = \prod_{i\in \bV}p_i^{C_i(t)}  \cdot p_{v_k} \cdot \prod_{j\in [n]\setminus \bV}p_j^{C_j(r)} \cdot \bp(\Span(r; \bv_{k-1})) \cdot
\prod_{1\le i<k} p_{v_i}.
\end{align*}
By summing over $t$ and $r$ and applying Cayley's multinomial formula, we deduce that
conditional on $v(\tilde T_k)=\bV\neq \varnothing$,
$(\tilde T_k, V_k)$ is independent of $(\tilde R_k, V_1,\dots, V_{k-1})$ and distributed
according to $\pi|_\bV\otimes \bp|_\bV$, which establishes the first claim for $k$.

Now, conditional on the event $\{V_k\in S_{k-1}\}$, the vertex
$V_k$ is distributed according to the
restriction of $\bp$ to $S_{k-1}$. In this case,
$H_k=H_{k-1}$ so that by the induction hypothesis
\begin{equation}\label{eq:Vk-in}
\text{on~}\{V_k\in S_{k-1}\},\qquad
(H_k, V_1,\dots, V_k)\sim \pi \otimes \bp^{k-1} \otimes \bp|_{S_{k-1}}.
\end{equation}
On the other hand, if $V_k\not\in S_{k-1}$,
then $\fv(\tilde T_k)\neq \varnothing$ and conditional on $\fv(\tilde T_k)=\bV$,
we have $(\tilde T_k, V_k)\sim \pi|_\bV\otimes \bp|_\bV$. In that case,
$H_k$ is obtained from $H_{k-1}$ by replacing $\tilde T_k$ by $\cut(\tilde T_k, V_k)$.
We have already proved that, in this case, $(\tilde T_k,V_k)$ is independent of $\tilde R_k$,
and Lemma~\ref{lem:joint_dist-tree-node} ensures that the replacement does not alter
the distribution. In other words,
\begin{equation}\label{eq:Vk-out}
\text{on~}\{V_k\not\in S_{k-1}\},\qquad
(H_k, V_1,\dots, V_k)\sim \pi \otimes \bp^{k-1}\otimes \bp|_{[n]\setminus S_{k-1}}.
\end{equation}
Since $V_k\sim \bp$ is independent of everything else, conditional on $S_{k-1}$, 
the event $\{V_k\in S_{k-1}\}$ occurs precisely with probability $\bp(S_{k-1})$,
so that putting \eqref{eq:Vk-in} and \eqref{eq:Vk-out} together completes the proof of the
induction step.
\end{proof}

\begin{figure}[t]
\centering
\includegraphics[scale=.7]{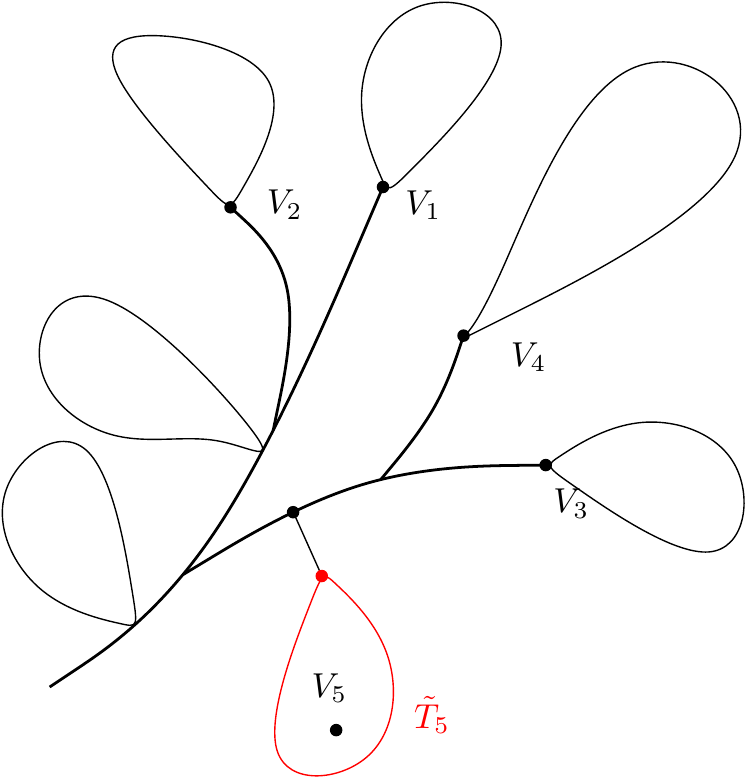}
\caption{\label{fig:induction} In order to obtain $\cut(T,V_1,\dots, V_k)$ from 
$\cut(T,V_1,\dots, V_{k-1})$, it suffices to transform the subtree $\tilde T_k$ of 
$\cut(T, V_1,\dots, V_{k-1})\setminus S_{k-1}$ which contains $V_k$.}
\end{figure}

\begin{cor}\label{cor:total_cuts}
Suppose that $T$ is a $\bp$-tree and that $V_1,\dots, V_k$ are $k\ge 1$ independent
$\bp$-nodes, also independent of $T$. Then,
$$S_k \eqd \Span(T; V_1,\dots, V_k).$$
In particular, the total number of cuts needed to isolate $V_1,\dots, V_k$ in $T$ is distributed as the number of vertices of $\Span(T; V_1,\dots, V_k)$.
\end{cor}

\medskip
\noi\textsc{Reverse $k$-cutting and duality.}\
As when we were isolating a single node $V$ in Section~\ref{sec:ptree-one-cutting},
the transformation that yields $H_k=\cut(T, V_1,\dots, V_k)$ is reversible.
To reverse the $1$-cutting procedure, we ``unknitted'' the path between $X_1$ and $V$.
Similarly, to reverse the $k$-cutting procedure, we ``unknit'' the backbone $S_k$ and by doing 
this obtain a collection of subtrees; then we re-attach these pendant subtrees  
at random nodes, which are chosen in suitable subtrees in order to obtain 
a tree distributed like the initial tree $T$.

\begin{figure}[t]
\centering
\includegraphics[scale=.7]{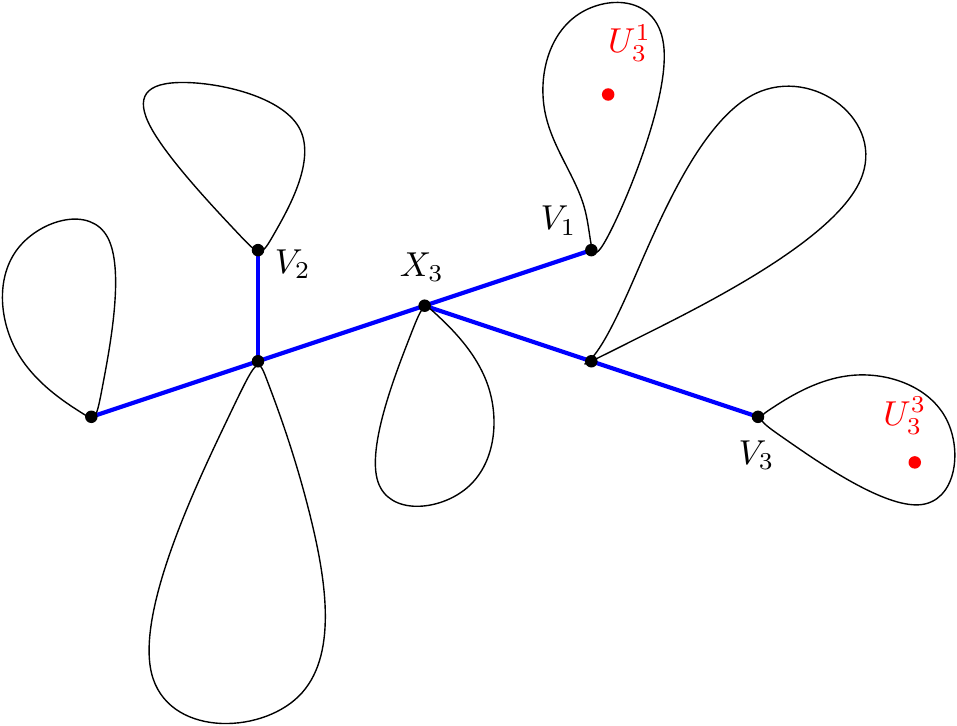}
\caption{\label{fig:cut-tree}The $3$-cut tree and the marked points $U_3^1$, $U_3^3$ corresponding to 
the cut node $X_3$. The backbone is represented by the subtree in thick blue.}
\end{figure}

For every $i$, the subtree $F_i$, rooted at $X_i$, was initially attached to
the set of nodes
$$\cU_i:=\{U_i^j: 1\le j\le k \text{ such that } T_{i-1}^j\ni X_i \}.$$
The corresponding edges have been replaced by some
edges which now lie in the backbone $S_k$.
So, to reverse the cutting procedure knowing the sets $\cU_i$,
it suffices to remove all the edges of $S_k$, and to re-attach $X_i$ to every node
in $\cU_i$. In other words, defining a reverse $k$-cutting transformation knowing
only the tree $H_k$ and the distinguished nodes $V_1,\dots, V_k$ reduces to
characterizing the distribution of the sets $\cU_i$.

Consider a tree $h\in \bbT_n$, and $k$ nodes $v_1,v_2,\dots, v_k$ not
necessarily distinct.
Removing the edges of $\Span(h; v_1,\dots, v_k)$ from $h$ disconnects it into
connected components $f_x$, each containing a single vertex $x$ of
$\Span(h; v_1,\dots, v_k)$.
For a given edge $\edge{x,w}$ of $\Span(h; v_1,\dots, v_k)$,
let $u_w$ be a node in $\Sub(h,w)$. Let $\bu$ be the vector of the
$u_w$, sorted according to the canonical order of $w$ on $\Span(h; v_1,\dots, v_k)$
(see p.\ \pageref{p:canon-order}).
For a given tree $h$ and $v_1,\dots, v_k$, we let $\bbU(h, v_1,\dots, v_k)$ be
the set of such vectors $\bu$.
For $\bu\in \bbU(h, v_1,\dots, v_k)$, define $\tau(h, v_1,\dots, v_k; \bu)$ as
the graph obtained from $h$ by
removing every edge $\edge{x,w}$ of $\Span(h; v_1,\dots, v_k)$
and replacing it by $\{x, u_w\}$. We regard $\tau(h, v_1,\dots, v_k; \bu)$
as rooted at the root of $h$.

\begin{lem}\label{lem:reverse-tree}Suppose that $h\in \bbT_n$, and that
$v_1,v_2,\dots, v_k$ are $k$ nodes of $[n]$, not necessarily distinct.
Then for every $\bu\in \bbU(h, v_1,\dots, v_k)$, $\tau(h, v_1,\dots, v_k; \bu)$ is a tree on $[n]$.
\end{lem}
\begin{proof}
Write $t:=\tau(h, v_1,\dots, v_k; \bu)$.
We proceed by induction on $n\ge 1$. For $n=1$,
$t=h$ is reduced to a single node; so $t$ is a tree.

Suppose now that for any tree $t'$ of size at most $n-1$, any $k\ge 1$,
any nodes $v_1,v_2,\dots,v_k\in \fv(t')$, and any $\bu'\in \bbU(t', v_1,\dots, v_k)$,
the graph $\tau(t', v_1,\dots, v_k; \bu')$ is a tree.
Let $N$ be the set of neighbors of the root $x_1$ of $h$.
For $y\in N$, define $\bv_y$ the subset of $\{v_1,\dots, v_k\}$ containing the
vertices which lie in $\Sub(h,y)$.
If $\bv_y\ne \varnothing$, let also $\bu_y\in \bbU(\Sub(h,y), \bv_y)$ be obtained 
from $\bu$ by keeping only the vertices $u_w$ for $w\in \Span^*(\Sub(h,y), \bv_y)$, 
still in the canonical order. Then, by construction,
the subtrees $\Sub(h,y)$, with $y\in N$ such that $\bv_y\ne \varnothing$ are 
transformed regardless of one another, and the others, for which $\bv_y=\varnothing$, 
are left untouched. 
So the graph
$\tau(h, v_1,\dots, v_k; \bu)$ induced on $[n]\setminus \{x_1\}$ consists
precisely of $\tau(\Sub(h,y), \bv_y; \bu_y)$, $y\in N$. By the induction
hypothesis, these subgraphs are actually trees.
Then $\tau(h, v_1,\dots, v_k; \bu)$ is simply obtained by adding the node $x_1$
together with the edges $\{x_1, u_y\}$, for $y\in N$, where
$u_y\in \Sub(h,y)$. In other words, each such edge connects $x_1$ to a different
tree $\tau(\Sub(h,y), \bv_y; \bu_y)$ so that the resulting graph is also a tree.
\end{proof}

For a given tree $h$ and $v_1,\dots, v_k\in [n]$
let $\bU\in \bbU(h, v_1,\dots, v_k)$ be obtained by sampling  $U_w$
according to the restriction of $\bp$ to $\Sub(h,w)$, for every
$w\in \Span^*(h, v_1,\dots, v_k)$.
Finally, we define the $k$-shuffled tree $\shuff(h; v_1,\dots, v_k)$ to be the
tree $\tau(h, v_1,\dots, v_k; \bU)$ re-rooted at an independent $\bp$-node.

We have the following result, which expresses the fact that the
$k$-cutting and $k$-shuffling procedures are truly reverses of one another.

\begin{prop}[$k$-cutting duality]\label{pro:K-duality}
Let $T$ be a $\bp$-tree and let $V_1,\dots, V_k$ be $k$ independent
$\bp$-nodes, also independent of $T$. Then, we have the following duality
$$(\shuff(T, V_1,\dots, V_k), T, V_1,\dots, V_k)
\eqd (T, \cut(T, V_1,\dots, V_k), V_1,\dots,V_k).$$
In particular, 
$(\shuff(T,V_1,\dots, V_k), V_1,\dots, V_k) \sim \pi \otimes \bp^{\otimes k}$.
\end{prop}
\begin{proof}
We consider the coupling we have defined on page~\pageref{p:coupling}: 
We let $H_k=\cut(T, V_1,\dots, V_k)$ for a $\bp$-tree $T$ rooted at $\hat r=r(T)$, 
and for every edge $\edge{x,w}$ of $\Span(H_k;V_1,\dots, V_k)$ we let $U_w$ be the 
unique node of $\Sub(H_k,w)$ which used to be connected to $x$ in the initial tree $T$. 
This defines the vector $\bU=(U_w, w\in \Span^*(H_k;V_1,\dots, V_k))$.
We show by induction on $k\ge 1$ that $\tau(H_k, V_1,\dots, V_k; \bU)^{\hat r}=T$ 
and that the joint distribution of $(H_k, \hat r, V_1,\dots, V_k, \bU)$ is that 
required by the construction above, so that 
$$(\tau(H_k, V_1,\dots, V_k; \bU)^{\hat r}, H_k, V_1,\dots, V_k)
\eqd (\shuff(H_k, V_1,\dots, V_k), H_k, V_1,\dots, V_k).$$
Since $H_k\eqd T$, this would complete the proof.

For $k=1$, the statement corresponds precisely to the construction of the proof 
of  Proposition~\ref{pro:one-duality}. As before, for $\ell\le k$, we let 
$S_\ell=\Span(H_k; V_1,\dots, V_\ell)$. If $k\ge 2$, let $\tilde R_k$ be the 
connected component of $H_k\setminus S_{k-1}$ which contains 
$V_k$, or $\tilde R_k=\varnothing$ if $V_k\in S_{k-1}$. 
In the latter case, $T=\tau(H_k, V_1,\dots, V_{k-1}, \bU)^{\hat r}$ 
and the joint distribution of $(H_k, \hat r, V_1,\dots, V_{k-1}, \bU)$ is correct by 
the induction hypothesis. 
Otherwise, let $\bU_k$ denote the sub-vector of $\bU$ consisting of the components 
$U_w$ for $w\in \Span^*(\tilde R_k, V_k)$, and let 
$\bU_{1,k-1}=(U_w, w\in \Span^*(H_k; V_1,\dots, V_{k-1}))$. 
If $\theta\in S_k$ is the unique point such that $\tilde R_k=\Sub(H_k, \theta)$
(that is, $\theta$ is the root of $\tilde R_k$),
then removing $\tilde R_k$ from $H_k$ and replacing it by $\tau(\tilde R_k, V_k; \bU_k)^{U_\theta}$ 
yields precisely the tree 
$H_{k-1}:=\cut(T; V_1,\dots, V_{k-1})$. 
Also, the distribution of $(\tilde R_k, U_\theta, V_k, \bU_k)$ is correct, since conditional on
the vertex set $\tilde R_k$ is distributed as $\pi|_{\fv(\tilde R_k)}$ (Lemma~\ref{lem:multi-cut}\ref{lmii}). 
Note that this transformation does not modify the distribution of $\bU_{1,k-1}$. 
By the induction hypothesis, $T=\tau(H_{k-1}, V_1,\dots, V_{k-1}; \bU_{1,k-1})^{\hat r}$.
Since conditionally on $S_{k-1}=\Span(H_k; V_1,\dots, V_{k-1})$ we have $V_k\in S_{k-1}$ 
with probability $\bp(S_{k-1})$, the proof is complete. 
\end{proof}

\subsection{The complete cutting and the cut tree.}
\label{sec:complete_cut}

For $n$ a natural number, we may also easily apply the previous procedure until
all $n$ nodes have been chosen. In this case, the cutting procedure continues
recursively in \emph{all} the connected components. The number of cuts is now
completely irrelevant (it is a.s.\ equal to $n$), 
and we define the forward transform as follows.
Let $T$ be a $\bp$-tree and let $(X_i, i\ge 1)$ be a sequence of
elements of $[n]$ such that $X_i$ is sampled according to the restriction of $\bp$
to $[n]\setminus \{X_1,\dots, X_{i-1}\}$.
Let $\Gamma_i=T\setminus\{X_1,\dots, X_i\}$; we stop precisely at time $n$, when
$\{X_1,\dots, X_n\}=[n]$ and $\Gamma_n=\varnothing$.

For every $k\in [n]$, define $T_i^{\edge{k}}$ as the connected component of $\Gamma_i$
which contains the vertex $k$, or $T_i^{\edge{k}}=\varnothing$ if $k\in \{X_1,\dots, X_i\}$.
For each $i=1,\dots, n$, let $\cU_i$ denote the set of neighbors of $X_i$
in $\Gamma_{i-1}$. 
Then we can write $\cU_i=\{U_i^{\edge{k}}: 1\le k\le n \text{ such that } T_{i-1}^{\edge{k}} \ni X_i\}$ 
where $U_i^{\edge{k}}$ is the unique element of $\cU_i$ which lies in $T_i^{\edge{k}}$.
The cuts which affect the connected component containing $k$ are
$$\cE_{\edge{k}}:=\{x \in [n]: \exists i\ge 1, X_i=x \in T_{i-1}^{\edge{k}}\}.$$

We claim that there exists a tree $G$ such that
for every $k\in [n]$, the path $\llb X_1, k \rrb$ in $G$ is precisely made
of the nodes in $\cE_{\edge{k}}$, in the order in which they appear in the
permutation $(X_1,X_2,\dots, X_n)$.
In the following, we write $\cut(T):=G$. The following proposition justifies the claim.

\begin{prop}\label{pro:limit_k_cutting}
Let $T$ be a $\bp$-tree, and let $V_k$, $k\ge 1$, be i.i.d.\ $\bp$-nodes, independent of $T$. 
Then, as $k\to \infty$,
$$\cut(T, V_1,\dots, V_k)\cdist{} \cut(T).$$
\end{prop}
\begin{proof}
We rely on the coupling we introduced in Section~\ref{sec:isol_mult}.
Since, for $k\ge 1$, we have $V_1,\dots, V_k\in S_k$ and $S_k\subseteq S_{k+1}$,
the tree $S_k$ converges almost surely to a tree on $[n]$, so that
$\lim_{k\to\infty} \cut(T; V_1,\dots, V_k)$ indeed exists with probability one.
In particular, although $\cut(T; V_1,\dots, V_k)$ certainly depends on 
$V_1,\dots, V_k$, the limit only depends on the sequence $(X_i, i\ge 1)$. 
Indeed, $K:=\inf\{k\ge 1 : [n]=\{V_1,\dots, V_k\}\}$ is a.s.\ finite, 
and for every $k\ge K$, one has $\cut(T; V_1,\dots, V_k)=\cut(T; X_1,\dots, X_n)$. 
We then write $\cut(T):=\cut(T; X_1,\dots, X_n)$.
\end{proof}



\begin{thma}[Cut tree]
\label{thm:complete_forward}
Let $T$ be a $\bp$-tree on $[n]$. Then, we have $\cut(T)\sim \pi$.
\end{thma}
\begin{proof}In the coupling defined in Section~\ref{sec:isol_mult}, we have
$$S_k = \Span(\cut(T); V_1, V_2,\dots, V_k)\to \cut(T)$$
almost surely as $k\to\infty$.
However, by Corollary~\ref{cor:total_cuts}, $S_k$ is distributed like
$\Span(T; V_1,\dots, V_k)$, so that $S_k\to T$ in distribution, as $k\to \infty$,
which completes the proof.
\end{proof}

\medskip
\noi\textsc{Shuffling trees and the reverse transformation.}\
Given a tree $g\in\bbT_n$ that we know is $\cut(t)$ for some tree $t\in \bbT_n$,
and the collections of sets $\cU_x$, $x\in [n]$,
we cannot recover the initial tree $t$ exactly, for the information about the
root has  been lost. 
However, the structure of $t$ as an unrooted tree is easily
(in this case, trivially) recovered by connecting every node $x$ to all the nodes
in $\cU_x$.
We now define the reverse operation, which samples the sets $\cU_x$ with the
correct distribution conditional on $g$, and produces a tree $\tilde T$ distributed as $T$
conditionally on $\cut(T)=g$.

Consider a tree $g\in \bbT_n$, rooted at $r\in [n]$. For each edge $\edge{x,w}$
of the tree $g$, let $U_w$ be a random element sampled according to the
restriction of $\bp$ to $\Sub(g,w)$.
Let $\bU\in \bbU(g):=\bbU(g, 1,2,\dots, n)$ be the vector of the $U_w$,
sorted using the canonical order on $g$ 
with distinguished nodes $1,2,\dots, n$.
Let $\tau(g, [n]; \bU)$ denote the graph on $[n]$ whose edges are
$\{x,U_w\}$, for $\edge{x, w}$ edges of $g$.
Then, $\tau(g, [n]; \bU)$ is a tree
(Lemma~\ref{lem:reverse-tree})
and we write $\shuff(g)$ for the random rerooting of $\tau(g, [n]; \bU)$ at
an independent $\bp$-node.

\begin{prop}\label{pro:limit-k-shuff}
Let $G$ be a $\bp$-tree, and $(V_k, k\ge 1)$ a sequence of i.i.d.\
$\bp$-nodes. Then, as $k\to\infty$,
$$\shuff(G; V_1,\dots, V_k) \cdist{} \shuff(G).$$
\end{prop}
\begin{proof}
We prove the claim using a coupling which we build using the random
variables $U_w$, $w\neq r$. For $k\ge 1$, we let $\bU_k$ be the subset of
$\bU$ containing the $U_w$ for which $w\in \Span^*(G; V_1,\dots, V_k)$,
in the canonical order on $\Span^*(G; V_1,\dots, V_k)$.
Then for $k\ge1$, $\bU_k\in \bbU(G, V_1,\dots, V_k)$ and since
$\Span(G; V_1,\dots, V_k)$ increases to $T$, the number of
edges of $\tau(G; V_1,\dots, V_k; \bU_k)$ which are constrained by the
choices in $\bU_k$ increases until they are all constrained. It follows that
$$\tau(G; V_1,\dots, V_k; \bU_k)\to \tau(G; 1,2,\dots, n; \bU)$$
almost surely, as
$k\to \infty$. Re-rooting all the trees at the same random $\bp$-node proves the
claim.
\end{proof}


We can now state the duality for the complete cutting procedure.
It follows readily from
the distributional identity  in Proposition~\ref{pro:K-duality}
$$(T, \cut(T, V_1,\dots, V_k)) \eqd (\shuff(T, V_1,\dots, V_k), T).$$
and the fact that $\cut(T; V_1,\dots, V_k)\to \cut(T)$ and
$\shuff(T; V_1,\dots, V_k)\to \shuff(T)$ in distribution as $k\to\infty$
(Propositions~\ref{pro:limit_k_cutting} and \ref{pro:limit-k-shuff}).

\begin{prop}[Cutting duality]\label{pro:complete-duality}
Let $T$ be a $\bp$-tree. Then, we have the following duality in distribution
$$(T, \cut(T))\eqd (\shuff(T), T).$$
In particular, $\shuff(T)\sim \pi$.
\end{prop}

\section{Cutting down an inhomogeneous continuum random tree}\label{sec:cont_cutting}

From now on, we fix some $\btheta=(\theta_0, \theta_1, \theta_2, \cdots)\in \bTheta$.
We denote by $I=\{i\ge 1: \theta_i>0\}$ the index set of those $\theta_i$ with nonzero values.
Let $\cT$ be the real tree obtained from the Poisson point process construction in Section~\ref{subset: icrt}.
We denote by $\mu$ and $\ell$ its respective mass and length measures.
Recall the measure $\cL$ defined by
$$
\cL(dx)=\theta_0^2 \ell(dx)+\sum_{i\in I} \theta_i\delta_{\beta_i}(dx),
$$
where $\beta_i$ is the branch point of local time $\theta_i$ for $i\in I$. 
The hypotheses on $\btheta$ entail that $\cL$ has infinite total mass. 
On the other hand, we have
\begin{lem}\label{lem: cL}
Almost surely, $\cL$ is a $\sigma$-finite measure concentrated on the skeleton of $\te$. 
More precisely, if $(V_i, i\ge 1)$ is a sequence of independent points sampled according to $\mu$, 
then for each $k\ge 1$, we have $\bbP$- almost surely
$$
\cL(\Span(\te; V_1, V_2, \cdots, V_k))<\infty.
$$
\end{lem}

\begin{proof}
We consider first the case $k=1$. Recall the Poisson processes $(\rP_j, j\ge 0)$ in the Section
\ref{subset: icrt} and the notations there. We have seen that 
$\Span(\te; V_1)$ and $R_1$ have the same distribution.
Then we have
$$
\cL(\Span(\te; V_1))
\overset{d}{=}\theta_0^2\,\eta_1+\sum_{i\ge 1}\theta_i\,\delta_{\xi_{i,1}}([0, \eta_1]).
$$
By construction, $\eta_1$ is either $\xi_{j, 2}$ for some $j\ge 1$ or $u_1$. 
This entails that on the event $\{\eta_1\in \rP_j\}$, we have $\eta_1<\xi_{i, 2}$ for all 
$i\in \bbN\setminus\{ j\}$. Then,
$$
\Egg{\sum_{i\ge 1} \theta_i\,\delta_{\xi_{i,1}}([0,\eta_1])}
=\sum_{j\ge 1}\Egg{\sum_{i\ge 1 }\theta_i\cdot\I{\xi_{i,1}\le \eta_1}\I{\eta_1=\xi_{j, 2}}}
+\Egg{\sum_{i\ge 1}\theta_i\cdot\I{\xi_{i,1}< \eta_1}\I{\eta_1=u_1}}.
$$
Note that the event $\{\xi_{j,1}\le \eta_1\}\cap\{\eta_1=\xi_{j, 2}\}$ always occurs. 
By breaking the first sum on $i$ into $\theta_j+\sum_{i\ne j}\theta_i\,\I{\xi_{i,1}< 
\eta_1<\xi_{i,2}}$ and re-summing over $j$, we obtain
\begin{align*}
\Egg{\sum_{i\ge 1} \theta_i\,\delta_{\xi_{i,1}}([0,\eta_1])}
&= \sum_{j\ge 1}\theta_j\,\p{\eta_1\in \rP_j}
+\sum_{j\ge 0}\Egg{\sum_{\substack{i\ge 1, i\ne j}}
\theta_i\cdot\I{\xi_{i,1}<\eta_1<\xi_{i,2}}\I{\eta_1\in \rP_j}}\\
&= \sum_{j\ge 1}\theta_j\,\p{\eta_1\in \rP_j}
+ \sum_{j\ge 0}\sum_{i\ne j}\Ec{\theta_i^2\eta_1e^{-\theta_i\eta_1} \I{\eta_1\in \rP_j}}\\
&\le 1+\sum_{i\ge 1} \theta_i^2\cdot\Ec{\eta_1},
\end{align*}
where we have used the independence of $(\mathrm{P}_j, j\ge 0)$ in the second equality.
The distribution of $\eta_1$ is given by \eqref{eq: distD}. 
If $\theta_0>0$, we have $\p{\eta_1>r}\le \exp(-\theta_0^2 r^2/2)$; 
otherwise, we have $\p{\eta_1>r}\le (1+\theta_1r)e^{-\theta_1 r}$. 
In either case, we are able to show that $\Ec{\eta_1}<\infty$. Therefore,
$$
\E{\cL(\Span(\te; V_1))}
=\theta_0^2\,\Ec{\eta_1}+\Egg{\sum_{i\ge 1}\theta_i\,\delta_{\xi_{i,1}}([0, \eta_1])}<\infty.
$$
In general, the variables $V_1, V_2, \cdots, V_k$ are exchangeable, therefore
$$
\E{\cL(\Span(\te; V_1, V_2, \cdots, V_k))}
\le k\E{\cL(\Span(\te; V_1))}<\infty,
$$
which proves that $\cL$ is almost surely finite on the trees spanning finitely many random leaves. 
Finally, with probability one, $(V_i, i\ge 1)$ is dense in $\te$, 
thus $\Sk(\te)=\cup_{k\ge 1}\rrb r(\cT), V_i\llb$ (see for example \cite[Lemma 5]{aldcrt3}). 
This concludes the proof.
\end{proof}

We recall the Poisson point process $\cP$ of intensity measure $dt\otimes \cL(dx)$, whose points we have used to define both the one-node-isolation procedure and the complete cutting procedure. As a direct consequence of Lemma \ref{lem: cL}, $\cP$ has finitely many atoms on $[0, t]\times \Span(\te; V_1, V_2, \cdots, V_k)$ for all $t>0$ and $k\ge 1$, almost surely. This fact will be implicitly used in the sequel.


\subsection{An overview of the proof}\label{sec:ICRT_overview}

Recall the hypothesis \eqref{H} on the sequence of the probability measures $(\bp_n, n\ge 1)$:
\begin{equation}\tag{H}
\sigma_n=\left(\sum_{i=1}^{n}p^2_{ni}\right)^{1/2}\overset{n\to\infty}{\longrightarrow} 0,
\qquad\text{and} \qquad
\lim_{n\to\infty}\frac{p_{ni}}{\sigma_n}=\theta_i, \quad \text{ for every }i\ge 1.
\end{equation}
Recall the notation $\bT^n$ for a $\bp_n$-tree, which, from now on, we consider 
as a measured metric space, equipped with the graph distance and the probability measure $\bp_n$.
Camarri--Pitman \cite{pit00} have proved that under hypothesis \eqref{H}, 
\begin{equation}\label{eq: G-P_ptree}
(\sigma_n \bT^n, \bp_n)\xrightarrow[d,\gp]{n\to\infty} (\te, \mu).
\end{equation}
This is equivalent to the convergence of the reduced subtrees: 
For each $n\ge 1$ and $k\ge 1$, write $R^n_k=\Span(\bT^n; \xi^n_1,\cdots, \xi^n_k)$ for 
the subtree of $\bT^n$ spanning the points $\{\xi^n_1, \cdots, \xi^n_k\}$, which are $k$ random 
points sampled independently with distribution $\bp_n$. 
Similarly, let $R_k=\Span(\te; \xi_1,\dots, \xi_k)$ be the subtree of
$\te$ spanning the points $\{\xi_1, \cdots, \xi_k\}$, where $(\xi_i, i\ge 1)$ is an i.i.d.\ 
sequence of common law $\mu$. 
Then \eqref{eq: G-P_ptree} holds if and only if for each $k\ge 1$,
\begin{equation}\label{eq: cvRnk}
\sigma_n R^n_k
\xrightarrow[d,\gh]{n\to\infty} R_k.
\end{equation}

However, even if the trees converge, one expects that for the cut trees to converge, 
one at least needs that the measures which are used to sample the cuts also converge 
in a reasonable sense.
Observe that $\cL$ has an atomic part, which, as we shall see, is the scaling limit of large $\bp_n$-weights.
Recall that $\bp_n$ is sorted: $p_{n1}\ge p_{n2}\ge \cdots p_{nn}$.  
For each $m\ge 1$, we denote by $\fB^n_m=(1,2, \cdots,m)$ the vector of the $m$ $\bp_n$-heaviest 
points of $\bT^n$, which is well-defined at least for $n\ge m$. 
Recall that for $i\ge 1$, $\beta_i$ denotes the branch point in $\te$ of local time $\theta_i$,
and write $\fB_m=(\beta_1, \beta_2, \cdots, \beta_m)$. Then \citet{pit00} also proved that
\begin{equation}\label{eq: G-P_ptree'}
\big(\sigma_n \bT^n, \bp_n, \fB^n_m\big)\carrow \big(\te, \mu, \fB_m\big)
\end{equation}
with respect to the $m$-pointed Gromov--Prokhorov topology, which will allow us 
to prove the following convergence of the cut-measures. 
Let
\begin{equation}\label{eq:def_Ln}
\cL_n=\sum_{i\in [n]}\frac{p_{ni}}{\sigma_n} \cdot \delta_{i}=\sigma_n^{-1}\bp_n.
\end{equation}

Recall the notation $m\!\!\upharpoonright_A$ for the (non-rescaled) restriction of a measure 
to a subset $A$.

\begin{prop}\label{prop: cv-Ln}
Under hypothesis \eqref{H}, we have
\begin{equation}\label{rr}
\big(\sigma_n R^n_{k},\cL_n\!\!\upharpoonright_{R^n_k}\big)
\carrow \big(R_{k},\cL\!\!\upharpoonright_{R_k}\big), \quad \forall k\ge 1,
\end{equation}
with respect to the Gromov--Hausdorff--Prokhorov topology. 
\end{prop}

The proof uses the techniques developed in \cite{pit00,ald04a} and is postponed until 
Section~\ref{sec:cv-Ln}. 
We prove in the following subsections that the convergence in Proposition \ref{prop: cv-Ln} is 
sufficient to entail convergence of the cut trees. To be more precise, we denote by $V^n$ 
a $\bp_n$-node independent of the $\bp$-tree $T^n$, and recall that in the construction of 
$H^n:=\cut(T^n,V^n)$, the node $V^n$ ends up at the extremity of the path 
upon which we graft the discarded subtrees. 
Recall from the construction of $\cH:=\cut(\cT,V)$ in Section \ref{sec:results} that 
there is a point $U$, which is at distance $L_\infty$ from the root. 
In Section~\ref{sec: pfcv-hatt}, we prove Theorem~\ref{thm:conv_cutv}, that is: if \eqref{H} holds, 
then
\begin{equation}\label{eq: cv-hatt}
(\sigma_n H^n, \bp_n,V^n) 
\mathop{\longrightarrow}^{n\to\infty}_{d,\gp}
(\cH,\hat \mu, U),
\end{equation}
jointly with the convergence in \eqref{rr}.
From there, the proof of Theorem~\ref{thm:id_cutv} is relatively short, and we provide it 
immediately (taking Theorem~\ref{thm:conv_cutv} or equivalently \eqref{eq: cv-hatt} for 
granted).

\begin{proof}[Proof of Theorem~\ref{thm:id_cutv}]
For each $n\ge 1$, let $(\xi^n_{i})_{i\ge 1}$ be a sequence of i.i.d.\ points of common law 
$\bp_n$, and let $\xi^n_{0}=V^n$. 
Let $(\xi_i)_{i\ge 1}$ be a sequence of i.i.d.\ points of common law $\hat\mu$,
and let $\xi_0=U$. 
We let 
$$\rho_n=(\sigma_nd_{H^n}(\xi^n_{i}, \xi^n_{j}))_{i, j\ge 0}
\qquad \text{and}\qquad
\rho^*_n=(\sigma_n d_{H^n}(\xi^n_{i}, \xi^n_{j}))_{i, j\ge 1}
$$ 
the distance matrices in $\sigma_n H^n=\sigma_n \cut(\bT^n,V_n)$ induced by the 
sequences $(\xi^n_i)_{i\ge 0}$ and $(\xi^n_i)_{i\ge 1}$, respectively.
According to Lemma \ref{lem:joint_dist-tree-node}, the distribution of 
$\xi^n_{0}=V^n$ is $\bp_n$, 
therefore $\rho_n$ is distributed as $\rho^*_n$. 
Writing similarly  
$$\rho=(d_{\cH}(\xi_{i}, \xi_{j}))_{i, j\ge 0}
\qquad \text{and}\qquad 
\rho^*=(d_{\cH}(\xi_{i}, \xi_{j}))_{i, j\ge 1},$$ 
where $d_{\cH}$ denotes the distance of $\cH=\cut(\cT,V)$, 
\eqref{eq: cv-hatt} entails that $\rho_n\to \rho$ in the sense of finite-dimensional distributions. 
Combined with the previous argument, we deduce that $\rho$ and $ \rho^*$ have the same distribution.
However, $\rho^*$ is the distance matrix of an i.i.d.\  sequence of 
law $\hat\mu$ on $H^n$. And the distribution of $\rho$ determines that of $V$. 
As a consequence, the law of $U$  
is $\hat{\mu}$.

For the unconditional distribution of $(\cH, \hat{\mu})$, it suffices to apply the 
second part of Lemma~\ref{lem:joint_dist-tree-node}, which says that 
$(H^n, \bp_n)$ is distributed like $(\bT^n, \bp_n)$.
Then comparing \eqref{eq: cv-hatt} with \eqref{eq: G-P_ptree} 
shows that the unconditional distribution of $(\cH, \hat{\mu})$ is that of 
$(\cT,\mu)$.
\end{proof}

In order to prove the similar statement for the sequence of complete cut trees $G^n=\cut(\bT^n)$ 
that is Theorem~\ref{thm:conv_cut}, the construction of the limit metric space $\cG=\cut(\cT)$ 
first needs to be justified by resorting to Aldous' theory of continuum random trees \cite{aldcrt3}. 
The first step consists in proving that the backbones of $\cut(\bT^n)$ converge. 
For each $n\ge 1$, let $(V^n_i, i\ge 1)$ be a sequence of i.i.d.\ points of law $\bp_n$. 
Recall that we defined $\cut(\cT)$ using an increasing family  
$(S_k)_{k\ge 1}$, defined in \eqref{eq: skd}.
We show in Section~\ref{sec: pfcv_tc} that
\begin{lem}\label{lem: cv_tc}
Suppose that \eqref{H} holds. Then, for each $k\ge 1$, we have
\begin{equation}\label{eq: cv_tc}
\sigma_n\Span(\cut(\bT^n); V^n_1, \cdots, V^n_{k})
\xrightarrow[d,\gh]{n\to\infty} 
S_k,
\end{equation}
jointly with the convergence in \eqref{rr}.
\end{lem}

Combining this with the identities for the discrete trees in Section~\ref{sec:ptree-cutting}, 
we can now prove Theorems~\ref{thm:conv_cut} and~\ref{thm:id_cut}.
\begin{proof}[Proof of Theorem \ref{thm:conv_cut}]
By Theorem \ref{thm:complete_forward}, $(\cut(\bT^n), \bp_n)$ and $(\bT^n, \bp_n)$ have the same 
distribution for each $n\ge 1$. 
Recall the notation $R^n_k$ for the subtree of $\bT^n$ spanning $k$ i.i.d.\ $\bp_n$-points. 
Then  for each $k\ge 1$ we have 
$$S^n_k:=\Span(\cut(\bT^n), V_1^n,\dots, V_k^n)\eqd R^n_k.$$ 
Now comparing \eqref{eq: cv_tc} with \eqref{eq: cvRnk}, we deduce 
immediately that, for each $k\ge 1$,
$$S_k\eqd R_k.$$
In particular the family $(S_k)_{k\ge 1}$ is consistent and leaf-tight in the sense of 
\citet{aldcrt3}.
This even holds true almost surely conditional on $\te$.
According to Theorem 3 and Lemma 9 of \cite{aldcrt3}, this entails that conditionally on 
$\cut(\te)$, the empirical measure $\frac 1 k \sum_{i=1}^{k}\delta_{U_i}$ converges weakly to 
some probability measure $\nu$ on $\tc$ such that $(U_i, i\ge 1)$ has the distribution of a sequence 
of i.i.d.\ $\nu$-points. This proves the existence of $\nu$. Moreover, 
$$S_k\eqd \Span(\cut(\te), \xi_1, \cdots, \xi_k),$$ 
where $(\xi_i, i\ge 1)$ is an i.i.d.\ $\mu$-sequence. Therefore, \eqref{eq: cv_tc} entails 
that $(\sigma_n\cut(\bT^n), \bp_n)\to (\cut(\cT), \nu)$ in distribution with respect 
to the Gromov--Prokhorov topology.
\end{proof}

\begin{proof}[Proof of Theorem \ref{thm:id_cut}]
According to Theorem 3 of \cite{aldcrt3} the distribution 
of $(\tc, \nu)$ is characterized by the family $(S_k)_{k\ge 1}$. Since $S_k$ and 
$R_k$ have the 
same distribution for $k\ge 1$, it follows that $(\tc,\nu)$ is distributed like $(\te, \mu)$.
\end{proof}


\subsection{Convergence of the cut-trees $\cut(\bT^n,V^n)$: Proof of Theorem~\ref{thm:conv_cutv}}
\label{sec: pfcv-hatt}

In this part we prove Theorem \ref{thm:conv_cutv} taking Proposition~\ref{prop: cv-Ln} for granted. 
Let us first reformulate \eqref{eq: cv-hatt} in the terms of the distance matrices, which is what we 
actually show in the following. 
For each $n\in\bbN$, let $(\xi^n_i, i\ge 2)$ be a sequence of random points of $\bT^n$ 
sampled independently according to the mass measure~$\bp_n$. 
 
We set $\xi^n_1=V^n$ and let $\xi^n_0$ be the root of $H^n=\cut(\bT^n,V^n)$. 
Similarly, let $(\xi_i, i\ge 2)$ be a sequence of i.i.d.\ $\mu$-points and let $\xi_1=V$. 
Recall that the mass measure $\hat{\mu}$ of $\cH=\cut(\te, V)$ is defined to be the push-forward 
of $\mu$ by the canonical injection $\phi$. We set $\widehat{\xi}_i=\phi(\xi_i)$ for $i\ge 2$, 
$\widehat\xi_1=U$ and $\widehat\xi_0$ to be the root of $\cH$.

Then the convergence in \eqref{eq: cv-hatt} is 
equivalent to the following: 
\begin{equation}\label{eq: cvhW}
  \big(\sigma_n d_{H^n}(\xi^n_i, \xi^n_j), 0\le i<j<\infty\big)
  \mathop{\longrightarrow}^{n\to\infty}_d 
  \big(d_{\cH}(\widehat{\xi}_i, \widehat{\xi}_j), 0\le i<j<\infty\big),
\end{equation}
jointly with
\begin{equation}\label{fdd}
\big(\sigma_n d_{T^n}(\xi^n_i, \xi^n_j), 1\le i<j<\infty\big)
  \mathop{\longrightarrow}^{n\to\infty}_d 
 \big(d_{\te}(\xi_i, \xi_j), 1\le i<j<\infty\big),
\end{equation}
in the sense of finite-dimensional distributions.  Notice that \eqref{fdd}
is a direct consequence of \eqref{eq: G-P_ptree}.
In order to express the terms in \eqref{eq: cvhW} with functionals of the cutting process, we introduce the 
following notations. For $n\in\bbN$, let $\cP_n$ be a Poisson point process on $\R_+\times \bT^n$ with intensity measure
$dt\otimes \cL_n$, where $\cL_n=\bp_n/\sigma_n$. For $u, v\in\bT^n$, 
recall that $\llb u, v\rrb$ denotes the path between $u$ and $v$.
For $t\ge 0$, we denote by $\bT^n_t$ the set of nodes 
still connected to $V^n$ at time $t$: 
$$
\bT^n_t:=\{x\in \bT^n: [0, t]\times \llb V^n, x\rrb \cap \cP_n=\varnothing\}.
$$
Recall that the remaining part of $\cT$ at time $t$ is 
$\cT_t=\{x\in \cT: [0,t]\times \llb V, x\rrb \cap \cP = \varnothing\}$. 
We then define
\begin{equation}\label{eq: defLn}
L^{n}_t:=\Card\big\{s\le t: \bp_n(\bT^n_s)<\bp_n(\bT^n_{s-})\big\}
\overset{a.s.}{=}\Card\big\{(s, x)\in \cP_n: s\le t, x\in \bT^n_{s-}\big\}.
\end{equation}
This is the number of cuts that affect the connected component containing $V^n$ 
before time $t$. In particular, 
$L^n_\infty:=\lim_{t\to\infty}L^n_t$ has the same distribution as $L(\bT^n)$ in the notation of 
Section~\ref{sec:ptree-cutting}. Indeed, this follows from the coupling on page~\pageref{p:coupling} 
and the fact that if $\cP_n=\{(t_i, x_i): i\ge 1\}$ such that $t_1\le t_2\le \cdots$ then $(x_i)$ 
is an i.i.d.\ $\bp_n$-sequence.
Let us recall that $L_t$, the continuous analogue of $L^n_t$, 
is defined by $L_t=\int_0^t \mu(\cT_s)ds$ in Section \ref{sec:results}. 
For $n\in\bbN$ and $x\in \bT^n$, we define the pair $(\tau_n(x), \varsigma_n(x))$ to be 
the element of $\cP_n$ separating $x$ from $V^n$
$$
\tau_n(x):= \inf\{t>0: [0, t]\times \llb V^n, x\rrb\cap \cP_n\ne\varnothing\},
$$
with the convention that $\inf\varnothing=\infty$. 
In words, $\varsigma_n(x)$ is the first cut that appeared on $\llb V^n, x\rrb$. 
For $x\in \cT$, $(\tau(x),\varsigma(x))$ is defined similarly.
We notice that almost surely 
$\tau(\xi_j)<\infty$ for each $j\ge 2$, since $\tau(\xi_j)$ is exponential with
rate $\cL(\llb V, \xi_j\rrb)$, which is  positive almost surely.
Furthermore, it follows from our construction of $H^n=\cut(\bT^n,V^n)$ that for $n\in\bbN$ 
and $i, j\ge 2$,
\begin{align*}
&d_{H^n}(\xi^n_0, \xi^n_1)=L^n_\infty-1, \\
&d_{H^n}(\xi^n_0, \xi^n_j)
=L^n_{\tau_n(\xi^n_j)}-1+d_{\bT^n}\big(\xi^n_j, \varsigma_n(\xi^n_j)\big);\\ 
&d_{H^n}(\xi^n_1, \xi^n_j)
=L^n_\infty-L^n_{\tau_n(\xi^n_j)}+d_{\bT^n}\big(\xi^n_j, \varsigma_n(\xi^n_j)\big),
\end{align*}
while for $i, j\ge 2$,
\begin{align*}
&d_{\cH}(\widehat{\xi}_0, \widehat{\xi}_1)=L_\infty,\\
&d_{\cH}(\widehat{\xi}_0, \widehat{\xi}_j)
=L_{\tau(\xi_j)}+d_{\te}\big(\xi_j, \varsigma(\xi_j)\big);\\
&d_{\cH}(\widehat{\xi}_1, \widehat{\xi}_j)
=L_\infty-L_{\tau(\xi_j)}+d_{\te}\big(\xi_j, \varsigma(\xi_j)\big).
\end{align*}
For $n\in\bbN$ and $i, j\ge 2$, if we define the event
\begin{align}\label{eq:def_anij}
\cA_n(i, j)&:=\{\tau_n(\xi^n_i)=\tau_n(\xi^n_j)\}
\overset{a.s.}{=}\{\varsigma_n(\xi^n_i)=\varsigma_n(\xi^n_j)\},
\end{align}
and $\cA_n^c(i, j)$ its complement, then on the event $\cA_n(i, j)$, 
we have $d_{H^n}(\xi^n_i, \xi^n_j)=d_{\bT^n}(\xi^n_i, \xi^n_j)$. 
Similarly we define $\cA(i, j):=\{\tau(\xi_i)=\tau(\xi_j)\}$, and note that 
$\cA(i,j)=\{\varsigma(\xi_i)=\varsigma(\xi_j)\}$ almost surely. 
Recall that \eqref{eq: G-P_ptree} implies that 
$\sigma_n d_{\bT^n}(\xi^n_i, \xi^n_j)\to d_{\te}(\xi_i, \xi_j)$. 
Now, on the event $\cA^c_n(i, j)$, we have 
$$
d_{H^n}(\xi^n_i, \xi^n_j)=
\big|L^n_{\tau_n(\xi^n_j)}-L^n_{\tau_n(\xi^n_i)}\big|+d_{\bT^n}\big(\xi^n_j, \varsigma_n(\xi^n_j)\big)
+d_{\bT^n}\big(\xi^n_i, \varsigma_n(\xi^n_i)\big), 
$$
if $n\in\bbN$, and 
$$
d_{\cH}(\widehat{\xi}_i, \widehat{\xi}_j)
=\big|L_{\tau(\xi_j)}-L_{\tau(\xi_i)}\big|+d_{\te}\big(\xi_j, \varsigma(\xi_j)\big)
+d_{\te}\big(\xi_i, \varsigma(\xi_i)\big), 
$$
for the limit case.
Therefore in order to prove \eqref{eq: cvhW}, it suffices to show the joint convergence of 
the vector
$$
\Big(\Ic{\cA_n(i, j)}, \tau_n(\xi^n_i), \sigma_n d_{\bT^n} \big(\xi^n_j, 
\varsigma_n(\xi^n_j)\big), \big(\sigma_n L^n_t, t\in \bbR_+\cup\{\infty\}\big)\Big)
$$
to the corresponding quantities for $\te$, for each $i, j\ge 2$.
We begin with a lemma.
\begin{lem}\label{lem: mun}
Under \eqref{H},  we have the following joint convergences as $n\to\infty$:
 \begin{equation}\label{eq: mun}
 (\bp_n(\bT^{n}_t))_{t\ge 0}\overset{d}\to (\mu(\cT_t))_{t\ge 0},
 \end{equation}
in Skorokhod $J_1$-topology, along with
 \begin{align}\label{eq: an}
 \big(\Ic{\cA_n(i, j)}, 2\le i, j\le k \big)
 &\overset{d}{\to} 
 \big(\Ic{\cA(i, j)}, 2\le i, j\le k\big),\\
\label{eq: taun}
 \big(\tau_n(\xi^n_j),  2\le j\le k \big)
 &\overset{d}{\to} \big(\tau(\xi_j), 2\le j\le k\big),\qquad\text{and}\\
\label{eq: xn}
 \big(\sigma_n d_{\bT^n} \big(\xi^n_j, \varsigma(\xi^n_j)\big), 2\le j \le k \big)
 &\overset{d}{\to}
 \big(d_{\te}\big(\xi_j,\varsigma_n(\xi_j)\big), 2\le j\le k \big),
 \end{align}
for each $k\ge 2$, and jointly with the convergence in \eqref{rr}.
\end{lem}

\begin{proof}
Recall Proposition~\ref{prop: cv-Ln}, which says that, for each $k\ge 2$
$$
(\sigma_n R^n_k, \cL_n\!\!\upharpoonright_{R^n_k})\mathop{\longrightarrow}^{n\to\infty}_d 
(R_k, \cL\!\!\upharpoonright_{R_k}),
$$
in Gromov--Hausdorff--Prokhorov topology. By the properties of the Poisson point process, 
this entails that for $t\ge 0$,
\begin{equation}\label{eq: hh}
(\sigma_n R_k^n, \cP_n\!\!\upharpoonright_{[0, t]\times R^n_k})
\overset{d}{\to}
( R_k, \cP\!\!\upharpoonright_{[0, t]\times R_k}),
\end{equation}
in Gromov--Hausdorff--Prokhorov topology, jointly with the convergence in \eqref{rr}. 
For each $n\in\bbN$, the pair $(\tau_n(\xi^n_i), \varsigma_n(\xi^n_i))$ corresponds to the first 
jump of the point process $\cP_n$ restricted to $\llb V^n_1, \xi^n_i\rrb$. 
We notice that for each pair $(i, j)$ such that 
$2\le i, j\le k$, the event $\cA_n(i, j)$ occurs if and only if 
$\tau_n(\xi^n_i\wedge \xi^n_j)\le \min\{\tau_n(\xi^n_i),\tau_n(\xi^n_j)\}$. 
Similarly, $(\tau(\xi_i),\varsigma(\xi_i))$ is the first point of $\cP$ on $\R\times 
\llb V_1, \xi_1\rrb$, 
and $\cA(i,j)$ occurs if and only if $\tau(\xi_i\wedge \xi_j)\le \min\{\tau(\xi_i),\tau(\xi_j)\}$.
Therefore, the joint convergences in \eqref{eq: an}, \eqref{eq: taun} 
and \eqref{eq: xn} follow from \eqref{eq: hh}. On the other hand, we have
$$
\I{\xi^n_i\in \bT^n_t}=\I{t<\tau_n(\xi^n_i)}, \quad t\ge 0, n\ge 1
$$
For each fixed $t\ge 0$, this sequence of random variables converge to 
$\I{t<\tau(\xi_i)}=\I{\xi_i\in \cT_t}$ by \eqref{eq: hh}.
By the law of large numbers, 
$k^{-1}\sum_{1\le i\le k}\I{t<\tau_n(\xi^n_j)}\to \bp_n(T^n_t)$ almost surely.
Then we can find a sequence $k_n\to\infty$ slowly enough such that (see also \cite[Section 2.3]{aldcoal})
$$
\frac{1}{k_n}\sum_{i=1}^k \I{t<\tau_n(\xi^n_j)}\overset{d}{\to}\mu(\cT_t).
$$
This entails that, as $n\to\infty$, 
\begin{equation}\label{eq:cv_mun}
\bp_n(\bT^n_t)\overset{d}{\to}\mu(\cT_t).
\end{equation}
Using \eqref{eq:cv_mun} for a sequence of times $(t_m, m\ge 1)$ dense in $\bbR_+$ and combining 
with the fact that $t\mapsto \mu(\cT_t)$ is decreasing, 
we obtain  the convergence in \eqref{eq: mun}, jointly with \eqref{eq: an}, \eqref{eq: taun}, 
\eqref{eq: xn} and \eqref{rr}.
\end{proof}

\begin{prop}\label{prop: cutnb}
Under \eqref{H}, we have
\begin{equation}\label{eq: cutnb}
(\sigma_n L^{n}_t, t\ge 0)\mathop{\longrightarrow}^{n\to\infty}_{d} (L_{t}, t\ge 0)
\end{equation}
with respect to the uniform topology, and jointly with the convergences in \eqref{eq: an}, \eqref{eq: taun} and \eqref{eq: xn}. In particular, this entails that
$L_\infty<\infty$ almost surely. Moreover we have
\begin{equation}\label{eq: idL}
L_\infty\overset{d}{=}d_{\cT}(r(\cT), V),
\end{equation}
where $V$ is a random point of distribution $\mu$. 
The distribution of $d_{\cT}(r(\cT), V)$ is given in \eqref{eq: distD}.
\end{prop}

The above proposition is a consequence of the following lemmas.
\begin{lem}\label{lem: sp1}
Jointly with \eqref{eq: an}, \eqref{eq: taun} and \eqref{eq: xn}, we have for any $m\ge 1$ and $(t_i, 1\le i\le m)\in \bbR^m_+$,
$$
\bigg(\int_0^{t_i} \bp_n(\bT^n_s)ds, \,1\le i\le m\bigg)
\carrow \bigg(\int_0^{t_i} \mu(\cT_s)ds,\, 1\le i\le m\bigg).
$$
\end{lem}

\begin{proof}
This is a direct consequence of Lemma~\ref{lem: mun}. 
\end{proof}

\begin{lem}\label{lem: sp2}
If we let 
$$M^n_t:=\sigma_n L^n_t-\int_0^t \bp_n(\bT^n_s)ds, \quad n\ge 1;$$
then under the hypothesis that $\sigma_n\to 0$ as $n\to\infty$, the sequence of variables 
$(\sigma_n M^n_t, n\ge 1)$ converges to $0$ in $L^2$ as $n\to\infty$. 
Moreover, this convergence is uniform on compacts.
\end{lem}

In particular, Lemma \ref{lem: sp1} and Lemma \ref{lem: sp2} combined entail that for any fixed $t\ge 0$, 
$\sigma_n L^n_t\to L_t$ in distribution. However, to obtain the convergence of 
$\sigma_n L^n_\infty$ to $L_\infty$ in distribution
we need the following tightness condition.
\begin{lem}\label{lem: tg}
Under \eqref{H}, 
for every $\delta>0$,
\begin{equation}\label{eq: tg}
\lim_{t\to\infty}\limsup_{n\to\infty}\p{\sigma_n\big(L^{n}_\infty-L^{n}_{t}\big)\ge \delta}=0,
\end{equation}
\end{lem}

\begin{proof}[Proof of Lemma \ref{lem: sp2}]
Let $\mathrm{N}^n_t=\Card\{(s, x)\in \cP_n: s\le t\}$ be the counting process of $\cP_n$. 
Then $(\mathrm{N}^n_t, t\ge 0)$ is a Poisson process of rate $1/\sigma_n$. We write $d\mathrm{N}^n$ 
for the Stieltjes measure associated with $t\mapsto \mathrm{N}^n_t$. 
For $t\ge 0$, let 
$$
\sM_t\n:= L_t\n-\int_{[0,t]}\bp_n(T\n_{s-})d\mathrm{N}^n_s,  
\quad\text{and}\quad
\sN^n_t:=\sigma_n\int_{[0, t]} \bp_n(\bT^n_{s-})d\mathrm N^n_s-\int_0^t \bp_n(\bT^n_s)ds.
$$
We notice that, by the definition of $L^n_t$,
$$
\sM^n_t=\sum_{(s, x)\in \cP_n:\, s\le t}\Big(\I{x\,\in \bT^n_{s-}}-\bp_n(\bT^n_{s-})\Big).
$$
Since $\sigma_n^{-1}\bp_n=\cL_n$, conditionally on $\bT^n_{s-}$, 
$\I{x\,\in \bT^n_{s-}}$ is a Bernoulli random variable of mean 
$\bp_n(\bT^n_{s-})$. Therefore, we have 
\begin{equation}\label{eq: mnind}
\bbE[\mathscr M^n_t \,|\, (\mathrm N^n_s)_{s\le t}]=0.
\end{equation}
From this, we can readily show that $\sM^n$ is a martingale. On the other hand, classical results on 
the Poisson process entail that $\mathscr N^n$ is also a martingale. Once combined, we see that 
$\mathrm M^n= \sigma_n \sM^n+\mathscr N^n$ itself is a martingale.
Therefore, by Doob's maximal inequality for the $L^2$-norms of martingales, we obtain
for any $t\ge 0$,
$$
\bbE\bigg[\sup_{s\le t}(\mathrm{M}^{n}_{s})^2\bigg]
\le 4\bbE\big[(\mathrm{M}^{n}_{t})^2\big]
=4\bbE\big[(\sigma_n\mathscr{M}^{n}_{t})^2\big]+4\bbE\big[(\mathscr{N}^{n}_{t})^2\big],
$$
as a result of \eqref{eq: mnind}. Direct computation shows that
\begin{align*}
\bbE\big[(\mathscr{M}^{n}_{t})^2\big]
=\bbE\bigg[\frac{1}{\sigma_n}\int_0^t \Big(\bp_n(\bT^n_s)-\bp_n^2(\bT^n_s)\Big)ds\bigg],
\quad\text{and}\quad
\bbE\big[(\mathscr{N}^{n}_{t})^2\big]=\bbE\bigg[\sigma_n\int_0^t \bp_n^2(\bT^n_s)ds\bigg].
\end{align*}
As a consequence, for any fixed $t$,
$$\bbE\bigg[\sup_{s\le t}(\mathrm{M}^{n}_{s})^2\bigg]
\le 4\sigma_n\bbE\bigg[\int_0^t \bp_n(\bT^n_s)ds\bigg]
\le 4\sigma_n t\to 0,
$$
as $n\to\infty$.
\end{proof}

We need an additional argument to prove Lemma~\ref{lem: tg}. 
For each $n\in\bbN$ and $s\ge 0$, let $\zeta^n(s):=\inf\{t>0: L^{n}_t\ge \lfloor s\rfloor\}$ be the 
right-continuous inverse of $L^n_t$. Recall that from the construction of $H^n=\cut(\bT^n, V^n)$, there 
is a correspondence between the vertex sets of the remaining tree at step $\ell-1$ and the subtree in $H$ at $X_\ell$. Then it follows Lemma~\ref{lem:joint_dist-tree-node} that
$$
\big(\fv(\bT^n_{\zeta^n(s)}), 0\le s< L^n_\infty\big)
\eqd\big(\fv(\Sub(\bT^n,  x^n_s)), 0\le s<1 + d_{\bT^n}(r(\bT^n), V^n) \big),
$$
where $x^n_s$ is the point on the path $\llb r(\bT^n), V^n\rrb$ at distance $\lfloor s\rfloor$ 
from $r(\bT^n)$. 
In particular, this entails 
\begin{equation}\label{eq: idmuTn}
\big(\bp_n\big(\bT^n_{\zeta^n(s)}\big), 0\le s< L^n_\infty\big)
\eqd\big(\bp_n\big(\Sub(\bT^n, x^n_s)\big), 0\le s<1+ d_{\bT^n}(r(\bT^n), V^n)\big).
\end{equation}
The limit of the right-hand side is easily identified using the 
convergence of $\bp$-trees in \eqref{eq: G-P_ptree}.
Combined with \eqref{eq: idmuTn}, this will allow us to prove Lemma~\ref{lem: tg} 
by a time-change argument.

Let $V$ be a random point of $\te$ of distribution $\mu$. 
For $0\le s\le d_{\te}(r(\te), V)$, let $x_s$ be the point in 
$\llb r(\te), V\rrb$ at distance $s$ from $r(\te)$, or $x_s=V$ if $\ell> d(r(\te),V)$.  
Similarly, we set $x^n_s=V^n$ if $s\ge 1+d_{\bT^n}(r(\bT^n), V^n)$.
\begin{lem}\label{lem: sl}
Under \eqref{H}, we have 
$$
\Big(\sigma_n L^n_\infty, \big(\bp_n\big(\bT^n_{\zeta^n(s/\sigma_n)}\big)\big)_{s\ge 0}\Big)
\mathop{\longrightarrow}^{n\to\infty}_d 
\Big(d_{\te}(r(\te), V), \big(\mu(\Sub(\te, x_s))\big)_{s\ge 0}\Big),
$$
where the convergence of the second coordinates is with respect to the Skorokhod $J_1$-topology.
\end{lem}

\begin{proof}
Because of \eqref{eq: idmuTn} and the fact $\sigma_n\to 0$, it suffices to prove that 
$$
\Big(\bp_n\big(\Sub(\bT^n, x^n_{s/\sigma_n}), s\ge 0\Big)
\mathop{\longrightarrow}^{n\to\infty}_d 
\Big( \mu\big(\Sub(\te, x_{s}\big), s\ge 0\Big),
$$
with respect to the Skorokhod $J_1$-topology, jointly with 
$\sigma_n d_{\bT^n}(r(\bT^n), V^n)\to d_{\te}(r(\te), V)$ in distribution.
Recall that $(\xi^n_{i}, i\ge 2)$ is a sequence of 
i.i.d.\ points of common law $\bp_n$ and set $\xi^n_{0}=V^n$, $\xi^n_{1}=r(\bT^n)$, for $n\in\bbN$. 
Note that $(\xi^n_{i}, i\ge 0)$ is still an i.i.d.\  sequence. Then it follows from 
\eqref{eq: G-P_ptree} that 
$$(\sigma_n d_{\bT^n}(\xi^n_{i}, \xi^n_{j}), i, j\ge 0) 
\cdist{}
(d_\te(\xi_i,\xi_j), i,j\ge 0)
$$
in the sense of finite-dimensional distributions.
Taking $i=0$ and $j=1$, we get the convergence 
$$\sigma_n d_{\bT^n}(V^n, r(\bT^n))\overset{d}{\to} d_{\te}(V, r(\te)).$$
On the other hand, for $i\ge 1$, $\xi^n_i\in \Sub(\bT^n, x^n_s)$ if and only if 
$d_{\bT^n}(\xi^n_{i}\wedge V^n, r(\bT^n))\ge s$.
Since for any rooted tree $(\bT, d, r)$ and $u, v\in \bT$ we have
$2d(r, u\wedge v)=d(r, u)+d(r, v)-d(u, v)$, 
we deduce that for any $k, m\ge 1$ and $(s_j, 1\le j\le m)\in \bbR^m_+$, 
$$
\Big(\I{\xi^n_{i}\in \Sub(\bT^n, x^n_{s_j/\sigma_n})}, 1\le i\le k, 1\le j\le m\Big)
\overset{d}{\to} 
\Big(\I{\xi_{i}\in \Sub(\te, x_{s_j})}, 1\le i\le k, 1\le j\le m\Big),
$$
jointly with $\sigma_n d_{\bT^n}(V^n, r(\bT^n))\overset{d}{\to} d_{\te}(V, r(\te))$. 
Then the argument used to establish \eqref{eq:cv_mun} shows
the convergence of $(\bp_n(\Sub(\bT^n,x^n_{s/\sigma_n}), n\ge 1)$ in the sense of 
finite-dimensional distributions.
The convergence in the Skorokhod topology follows from the monotonicity of the function
$s\mapsto \bp_n(\Sub(\bT^n, x^n_s))$.
\end{proof}

\begin{proof}[Proof of Lemma \ref{lem: tg}]
Let us begin with a simple observation on the Skorokhod $J_1$-topology.
Let $\mathbb{D}^\uparrow$ be the set of those functions $x: \bbR_+\to [0, 1]$ which are nondecreasing and 
c\`adl\`ag. We endow $\mathbb D^\uparrow$ with the Skorokhod $J_1$-topology. Taking $\ep>0$ and 
$x\in \mathbb{D}^\uparrow$, we denote by $\kappa_\ep(x)=\inf\{t>0: x(t)> \ep\}$. The following is a well-known fact. 
A proof can be found in \cite[Ch.\ VI, p.\ 304, Lemma~2.10]{Jacodbook}

\medskip
\noindent\textbf{FACT }\space 
If $x_n\to x$ in 
$\mathbb{D}^\uparrow$, $n\to\infty$ and $t\mapsto x(t)$ is strictly increasing, 
then $\kappa_\ep(x_n)\to\kappa_\ep(x)$ as $n\to\infty$.

\medskip
If $x=(x(t), t\ge 0)$ is a process with c\`adl\`ag paths and $t_0\in \mathbb{R}_+$, we denote by 
$\mathrm{R}_{t_0}[x]$ the reversed process of $x$ at $t_0$:
$$
\rR_{t_0}[x](t)=x\big((t_0-t)-\big)
$$
if $t< t_0$ and $\rR_{t_0}[x](t)=x(0)$ otherwise. For each $n\ge 1$, let 
$x_n(t)=\bp_n(\bT^n_{\zeta^n(t)})$, $t\ge 0$ and denote by $\Lambda_n=\mathrm R_{L^n_\infty}[x_n]$ 
the reversed process at $L^n_\infty$. Similarly, let $y(t)=\mu(\Sub(\te, x_t))$, $t\ge 0$ and 
denote by $\Lambda=\rR_{D}[y]$ for $D=d_{\te}(V, r(\te))$.
Then almost surely $\Lambda_n\in \bbD^\uparrow$ for $n\in\bbN$ and $\Lambda\in \bbD^\uparrow$. 
Moreover, Lemma~\ref{lem: sl} 
says that
\begin{equation}
\big(\Lambda_n(t/\sigma_n), t\ge 0\big)
\mathop{\longrightarrow}^{n\to\infty}_d
\big(\Lambda(t), t\ge 0\big)
\end{equation}
in $\bbD^\uparrow$. 
From the construction of the ICRT in Section \ref{subset: icrt} it is not difficult to show that 
$t\mapsto \Lambda(t)$ is strictly increasing. 
Then by the above FACT, we have $\sigma_n\kappa_\ep(\Lambda_n)\to \kappa_\ep(\Lambda)$ in 
distribution, for each $\ep>0$. In particular, we have  
for any fixed $\delta>0$,
\begin{equation}\label{eq: vv}
\lim_{\ep\to 0}\limsup_{n\to\infty}\bbP\big(\sigma_n\kappa_\ep(\Lambda_n)\ge \delta\big)
\le\lim_{\ep\to 0}\bbP\big(\kappa_\ep(\Lambda)\ge \delta\big)=0,
\end{equation}
since almost surely $\Lambda(t)>0$ for any $t>0$.

By Lemma~\ref{lem: mun}, the sequence $((\bp_n(\bT^n_t))_{t\ge 0}, n\ge 1)$ is tight in the 
Skorokhod topology. Combined with the fact that, 
for each fixed $n$, $\bp_n(\bT^n_t)\searrow 0$ as $t\to\infty$ almost surely, 
this entails that for any fixed $\ep>0$
\begin{equation}\label{eq: uu}
\lim_{t_0\to\infty}\limsup_{n\to\infty}\p{\sup_{t\ge t_0}\bp_n(\bT^n_t)\ge\ep}=0.
\end{equation}
Now note that if $L^n_t=k\in \bbN$, then $\bT^n_t=\bT^n_{\zeta^n(k)}$ a.s.\
since no change occurs until the time of the next cut, in particular we have
$$
\bp_n(\bT^n_t)=\bp_n\big(\bT^n_{\zeta^n(L^n_t)}\big) \quad \text{a.s.},
$$
from which we deduce that
$$
\big\{\bp_n(\bT_{t_0}^n)<\ep\big\}\subseteq \big\{\kappa_\ep(\Lambda_n)\ge L^n_\infty-L^n_{t_0}\big\}
\quad \text{a.s.},
$$
Then we have
$$
\big\{\sigma_n\big(L^n_\infty-L^n_{t_0}\big)\ge \delta\big\}
\cap\big\{\sup_{t\ge t_0}\bp_n(\bT_t^n)<\ep\big\}
\subseteq \big\{\sigma_n\kappa_\ep(\Lambda_n)\ge \delta\big\}, \quad \text{a.s.}
$$
Therefore,
\begin{align*}
&\limsup_{n\to\infty}\bbP\big(\sigma_n (L^n_\infty-L^n_{t_0})\ge \delta\big)\\
&\le\,\limsup_{n\to\infty}\p{\sup_{t\ge t_0}\bp_n(\bT^n_t)\ge\ep }+
\limsup_{n\to\infty}\p{\sigma_n\big(L^n_\infty-L^n_{t_0}\big)\ge \delta 
\text{~and~} \sup_{t\ge t_0}\bp_n(\bT^n_t)<\ep }\\
&\le \limsup_{n\to\infty}\p{\sup_{t\ge t_0}\bp_n(\bT^n_t)\ge\ep}
+\limsup_{n\to\infty}\p{\sigma_n\kappa_\ep(\Lambda_n)\ge \delta}.
\end{align*}
In above, if we let first $t_0\to\infty$ and then $\ep\to 0$, we obtain \eqref{eq: tg} 
as a combined consequence of \eqref{eq: vv} and \eqref{eq: uu}. 
\end{proof}

\begin{proof}[Proof of Proposition \ref{prop: cutnb}]
We fix a sequence of $(t_m, m\ge 1)$, which is dense in $\bbR_+$. Combining Lemmas~\ref{lem: sp1} 
and~\ref{lem: sp2}, we obtain, for all $k\ge 1$,
\begin{equation}\label{eq: dense}
\big(\sigma_n L^n_{t_m}, 1\le m\le k\big)\carrow \big(L_{t_m}, 1\le m\le k\big),
\end{equation}
jointly with the convergences in \eqref{eq: an}, \eqref{eq: taun}, \eqref{eq: xn} and \eqref{rr}. 
We deduce from this and Lemma~\ref{lem: tg} that $L_\infty<\infty$ a.s.\ and
\begin{equation}\label{eq: linfty}
\sigma_n L^n_\infty \carrow L_\infty,
\end{equation}
jointly with \eqref{eq: an}, \eqref{eq: taun}, \eqref{eq: xn} and \eqref{rr}, by Theorem 4.2 of 
\cite[Chapter 1]{billingsley}. Combined with the fact that $t\mapsto L_t$ is continuous and increasing, 
this entails the uniform convergence in \eqref{eq: cutnb}. Finally, the distributional identity \eqref{eq: idL} 
is a direct consequence of Lemma \ref{lem: sl}.
\end{proof}

\begin{proof}[Proof of Theorem \ref{thm:conv_cutv}]
We have seen that $L_\infty<\infty$ almost surely. Therefore the cut tree 
$(\cut(\cT,V), \hat{\mu})$ is well 
defined almost surely. Comparing the expressions of $d_{H^n}(\xi^n_i, \xi^n_j)$ 
given at the beginning of this subsection with those of $d_{\cH}(\xi_i, \xi_j)$, we obtain from Lemma \ref{lem: mun} and 
Proposition~\ref{prop: cutnb} the convergence in \eqref{eq: cvhW}. This concludes the proof.
\end{proof}

\medskip
\noi\textbf{Remark.}\ 
Before concluding this section, let us say a few more words on the proof of 
Proposition~\ref{prop: cutnb}. The convergence of $(\sigma_nL^n_t, t\ge 0)$ to $(L_t, t\ge 0)$ on 
any finite interval follows mainly from the convergence in Proposition \ref{prop: cv-Ln}. 
The proof here can be easily adapted to the other models of random trees, see \cite{Bert12, LWV}. 
On the other hand, our proof of the tightness condition \eqref{eq: tg} depends on the specific 
cuttings on the birthday trees, which has allowed us to deduce the distributional identity 
\eqref{eq: idmuTn}. In general, the convergence of $L^n_\infty$ may indeed fail. 
An obvious example is the classical record problem (see Example 1.4 in~\cite{Janson06}), where we 
have $L^n_t\to L_t$ for any fixed $t$, while $L^n_\infty\sim \ln n$ and therefore is not tight 
in $\bbR$.


\subsection{Convergence of the cut-trees $\cut(\bT^n)$: Proof of Lemma~\ref{lem: cv_tc}}
\label{sec: pfcv_tc}

Let us recall the settings of the complete cutting down procedure for $\te$: $(V_i, i\ge 1)$ is an 
i.i.d.\ sequence of common law $\mu$; $\cT_{V_i}(t)$ is the equivalence class of $\sim_t$ 
containing $V_i$, whose mass is denoted by $\mu_i(t)$; and $L^i_t=\int^t_0\mu_i(s)ds$. 
The complete cut-tree $\cut(\te)$ is defined as the complete separable metric space $\overline{\cup_k S_k}$. 
We introduce some corresponding notations for the discrete cuttings on $\bT^n$. 
For each $n\ge 1$, we sample a sequence of  i.i.d.\  points $(V^n_i, i\ge 1)$ on $\bT^n$ of 
distribution $\bp_n$. Recall $\cP_n$ the Poisson point process on $\bbR_+\times T^n$ of intensity $dt\otimes \cL_n$.
We define
\begin{align*}
\mu_{n,i}(t)&:=\bp_n(\{u\in \bT^n: [0, t]\times \llb u, V^n_i\rrb\cap \cP_n=\varnothing\}), \\
L^{n,i}_t&:=\Card\{s\le t: \mu_{n, i}(s)<\mu_{n, i}(s-)\},
\qquad t\ge 0, i\ge 1;\\
\tau_n(i,j)&:=\inf\{t\ge 0: [0, t]\times \llb V^n_i, V^n_j\rrb\cap \cP_n\ne\varnothing\}, 
\qquad 1\le i,j<\infty.
\end{align*}
By the construction of $G^n=\cut(\bT^n)$, we have
\begin{align} \label{eq: sknd}
& d_{G^n}(V^n_i, r(G^n))=L^{n,i}_\infty-1, \\ \notag
& d_{G^n}(V^n_i, V^n_j)=L^{n,i}_\infty+L^{n,j}_\infty-2L^{n,i}_{\tau_n(i,j)}
, \quad 1\le i, j<\infty
\end{align}
where $L^{n,i}_\infty:=\lim_{t\to \infty}L^{n,i}_t$ is the number of  
cuts necessary to isolate $V^n_i$. 
The proof of Lemma \ref{lem: cv_tc} is quite similar to that of Theorem \ref{thm:conv_cutv}. 
We outline the main steps but leave out the details.

\begin{proof}[Sketch of proof of Lemma \ref{lem: cv_tc}]
First, we can show with essentially the same proof of Lemma~\ref{lem: mun} that we have the 
following joint convergences: for each $k\ge 1$,
\begin{equation}\label{eq: mun'}
\big(\big(\mu_{n,i}(t), 1\le i\le k\big), t\ge 0\big)
\mathop{\longrightarrow}^{n\to\infty}_d 
\big(\big(\mu_i(t), 1\le i\le k\big), t\ge 0\big),
\end{equation}
with respect to Skorokhod $J_1$-topology, jointly with
\begin{equation}\label{eq: taun'}
\big(\tau_n(i,j), 1\le i, j\le k\big)
\mathop{\longrightarrow}^{n\to\infty}_d \big(\tau(i,j), 1\le i, j\le k\big),
\end{equation}
jointly with the convergence in \eqref{rr}.
Then we can proceed, with the same argument as in the proof of Lemma~\ref{lem: sp2}, 
to showing that for any $k, m\ge 1$ and $(t_j, 1\le j\le m)\in \bbR^m_+$,
$$
\bigg(\int_0^{t_i} \mu_{n, i}(s)ds, \,1\le j\le m, 1\le i\le k\bigg)
\carrow 
\bigg(\int_0^{t_i} \mu_i(s)ds,\, 1\le j\le m, 1\le i\le k\bigg).
$$
Since the $V^n_i, i\ge 1$ are i.i.d.\ $\bp_n$-nodes on $\bT^n$,
each process $(L^{n, i}_t)_{t\ge 0}$ has the same distribution as $(L^n_t)_{t\ge 0}$ defined 
in \eqref{eq: defLn}. 
Then Lemmas~\ref{lem: sp2} and~\ref{lem: tg} hold true for each $L^{n, i}$, $i\ge 1$. 
We are able to show
\begin{equation}\label{eq: cutnb'}
\big(\sigma_n L^{n,i}_t, 1\le i\le k\big)_{t\ge 0}
\mathop{\longrightarrow}^{n\to\infty}_d 
\big(L^i_t, 1\le i\le k\big)_{t\ge 0},
\end{equation}
with respect to the uniform topology, jointly with the convergences 
\eqref{eq: taun'} and \eqref{rr}. Comparing \eqref{eq: sknd} with \eqref{eq: skd}, 
we can easily conclude.
\end{proof}

In general, the convergence in \eqref{eq: G-P_ptree} does not hold in the Gromov--Hausdorff 
topology. However, in the case where $\te$ is a.s.\ compact 
and the convergence \eqref{eq: G-P_ptree} does hold in the Gromov--Hausdorff sense, 
then we are able to show that one indeed has GHP convergence as claimed in Theorem~\ref{thm:cv_ghp}. 
In the following proof, we only deal with the case of convergence of $\cut(\bT^n)$.  
The result for $\cut(\bT^n,V^n)$ can be obtained using similar arguments and we omit the details.

\begin{proof}[Proof of Theorem~\ref{thm:cv_ghp}]
We have already shown in Lemma \ref{lem: cv_tc} the joint convergence of the spanning subtrees: 
for each $k\ge 1$,
\begin{equation}\label{eq: RnkSnk}
\big(\sigma_n R^n_k, \sigma_n S^n_k\big)
\mathop{\longrightarrow}^{n\to\infty}_{d,\gh} \big(R_k, S_k\big).
\end{equation}
We now show that for each $\ep>0$,
\begin{equation}\label{eq: tightRk*}
\lim_{k\to\infty}\lim_{n\to\infty}
\p{\max\big\{ \dgh(R^n_k, \bT^n), \dgh(S^n_k, \cut(\bT^{n}))\big\}\ge \ep/\sigma_n}=0.
\end{equation}
Since the couples $(S^n_k, \cut(\bT^n))$ and $(R^n_k, \bT^n)$ have the 
same distribution, it is enough to prove that for each $\ep>0$,
\begin{equation}\label{eq: tightRk}
\lim_{k\to\infty}\lim_{n\to\infty}\p{\sigma_n \dgh(R^n_k, \bT^n)\ge \ep}=0.
\end{equation}
Let us explain why this is true when $(\sigma_n T^n,\bp_n)\to (\cT,\mu)$ in distribution in the 
sense of GHP. Recall the space $\bbM^k_c$ 
of equivalence classes of $k$-pointed compact metric spaces, equipped with the $k$-pointed Gromov--Hausdorff metric.
For each $k\ge 1$ and $\ep>0$, we set
$$
A(k,\ep):=\big\{(T, d, \bx)\in \bbM^k_c: \dgh\big(T, \Span(T; \bx)\big)\ge \ep\big\}.
$$
It is not difficult to check that $A(k, \ep)$ is a closed set of $\mathbb M^k_c$.
Now according to the proof of Lemma 13 of \cite{miermont09}, the mapping from $\bbM_c$ to 
$\mathbb{M}^k_c$: $(T, \mu)\mapsto m_k(T, A(k, \ep))$ is upper-semicontinuous, where $\bbM_c$ is 
the set of equivalence classes of compact measured metric spaces, equipped with the Gromov--Hausdorff--Prokhorov 
metric and $m_k$ is defined by
$$
m_k(T, A(k,\ep)):=\int_{T^k}\mu^{\otimes k}(d\bx)\boldsymbol{1}_{\{[T, \bx]\in A(k, \ep)\}}.
$$
Applying the Portmanteau Theorem for upper-semicontinuous mappings 
\cite[][p.\ 17, Problem 7]{billingsley}, we obtain
$$
\limsup_{n\to\infty} \E{m_k\big( (\sigma_n \bT^n, \bp_n), A(k,\ep)\big)}
\le \E{m_k\big((\te, \mu), A(k,\ep)\big)},
$$
or, in other words,
$$
\limsup_{n\to\infty} 
\p{\sigma_n \dgh\big(\bT^n, R^n_{k}\big)\ge \ep} \le \p{\dgh\big(\te, R_{k}\big)\ge \ep}
\xrightarrow[k\to\infty]{} 0,
$$
since $\dgh(R_k, \te)\to 0$ almost surely for $\te$ is compact \cite{aldcrt3}. 
This proves \eqref{eq: tightRk} and thus \eqref{eq: tightRk*}. 
By \cite[Ch.\ 1, Theorem 4.5]{billingsley}, \eqref{eq: RnkSnk} 
combined with \eqref{eq: tightRk*} entails the joint convergence in distribution of 
$(\sigma_n \bT^n, \sigma_n \cut(\bT^n))$ to $(\te, \cut(\cT))$ in the Gromov--Hausdorff topology. 
To strengthen to the Gromov--Hausdorff--Prokhorov convergence, one can adopt the arguments 
in Section 4.4 of \cite{miermontBr} and we omit the details.
\end{proof}


\section{Reversing the one-cutting transformation}
\label{sec:reverse}

In this section, we justify the heuristic construction of $\shuff(\cH, U)$ given in 
Section~\ref{sec:results} for an ICRT $\cH$ and a uniform leaf $U$.
The objective is to define formally the shuffle operation in such a way that the identity 
\eqref{eq: idshuff} hold. 
In Section~\ref{sec:shuff-one-path}, we rely on weak convergence arguments to justify the construction 
of $\shuff(\cH, U)$ by showing it is the limit of the discrete construction 
in~Section \ref{sec:ptree-one-cutting}.
In Section~\ref{sec:dist_cuts}, we then determine from this result the distribution of the 
cuts in the cut-tree $\cut(\cT,V)$ and prove that with the right coupling, the 
shuffle can yield the initial tree back (or more precisely, a tree that is in the same 
GHP equivalence class, which is as good as it gets).


\subsection{Construction of the one-path reversal $\shuff(\cH, U)$}
\label{sec:shuff-one-path}


\begin{figure}[b]
\centering
\includegraphics[scale=.9]{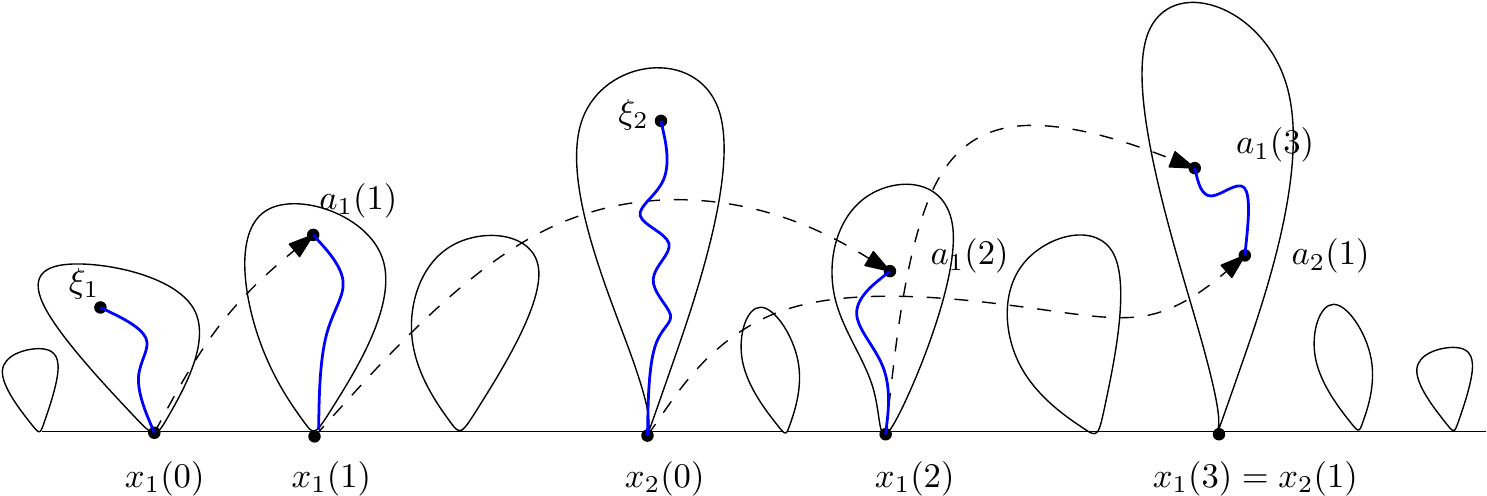}
\caption{\label{fig:shuffle}An example with $\sI(1, 2)=3$, $\sI(2,1)=1$ and $\me(1, 2)=4$. 
The dashed lines indicate the identifications where the root of the relevant subtrees 
are sent to. The blue lines represent the location of the path between $\xi_1$ and $\xi_2$ 
before the transformation.}
\end{figure}

Let $(\cH, d_{\cH}, \mu_{\cH})$ be an ICRT rooted at $r(\cH)$, and let $U$ be a random point in $\cH$ of distribution $\mu_{\cH}$. 
Then $\cH$ is the disjoint union of the following subsets:
$$
\cH=\bigcup_{x\in\llb r(\cH), U\rrb} F_x 
\quad \text{where}\quad F_x:=\{u\in \bT: \llb r(\cH), u\rrb\cap\llb r(\cH), U\rrb=\llb r, x\rrb\}.
$$
It is easy to see that $F_x$ is a subtree of $\bT$. It is nonempty ($x\in F_x$), 
but possibly trivial ($F_x=\{x\}$). 
Let $\bB:=\{x\in \llb r(\cH), U\rrb: \mu_{\cH}(F_x)>0\}\cup\{U\}$, and for $x\in \bB$, 
let $S_x:=\Sub(\bT, x)\setminus F_x$, which is the union of those $F_y$ such that 
$y\in \bB$ and $d_{\cH}(U, y)<d_{\cH}(U, x)$. 
Then for each $x\in \bB\setminus\{U\}$, we associate an attaching point $A_x$, 
which is independent and sampled according to $\mu_{\cH}|_{S_x}$, the restriction of $\mu_{\cH}$ to $S_x$. 
We also set $A_U=U$.

Now let $(\xi_i, i\ge 1)$ be a sequence of i.i.d.\ points of common law $\mu_{\cH}$. 
The set $\boldsymbol{\cF}:=\cup_{x\in \bB}F_x$ has full mass with probability one. 
Thus almost surely $\xi_i\in \boldsymbol{\cF}$ for each $i\ge 0$.
We will use $(\xi_i)_{i\ge 1}$ to span the tree $\shuff(\cH, U)$ and the point $\xi_1$ is 
the future root of $\shuff(\cH, U)$.
For each $\xi_i$, we define inductively two sequences $\bx_i:=(x_i(0), x_i(1), \cdots)\in \bB$ 
and $\ba_i:=(a_i(0), a_i(1), \cdots)$: we set $a_i(0)=\xi_i$,  and, for $j\ge 0$,
$$
x_i(j)=a_i(j)\wedge U, \qquad\text{and}\qquad a_i(j+1)=A_{x_i(j)}.
$$
By definition of $(A_x, x\in \bB)$, the distance $d_{\cH}(r(\cH), x_i(k))$ is increasing in $k\ge 1$.
For each $i, j\ge 1$, we define the merging time
$$
\me(i,j):=\inf\{k\ge 0: \exists l\le k \text{~and~} x_i(l)=x_j(k-l)\},
$$
with the convention $\inf\varnothing=\infty$. 
Another way to present $\me(i,j)$ is to consider the graph on $\bB$ with the edges 
$\{x,A_x\wedge U\}$, $x\in \bB$, then $\me(i,j)$ is the graph distance 
between $\xi_i\wedge U$ and $\xi_j \wedge U$.
On the event $\{\me(i, j)<\infty\}$, there is a path 
in this graph that has only finitely many edges, and the two walks $\bx_i$ and $\bx_j$ 
first meet at a point $y(i,j)\in \bB$ (where by first, we mean with minimum distance to the 
root $r(\cH)$). In particular, if we set $\sI(i, j), \sI(j, i)$ to be the respective 
indices of the element $y(i, j)$ appearing in $\bx_i$ and $\bx_j$, that is, 
$$
\sI(i, j)=\inf\{k\ge 0: x_i(k)=y(i, j)\} 
\quad\text{and}\quad
\sI(j, i)=\inf\{k\ge 0: x_j(k)=y(i, j)\},
$$
with the convention that $\sI(i, j)=\sI(j, i)=\infty$ if $\me(i ,j)=\infty$, then
$\me(i,j)=\sI(i,j)+\sI(j,i)$.  
Write $\Ht(u)=d(u,  u\wedge U)$ for the height of $u$ in the one of $F_x$, $x\in \bB)$, 
containing it. 
On the event $\{\me(i, j)<\infty\}$ we define $\gamma(i,j)$ which is meant to be 
the new distance between $\xi_i$ and $\xi_j$:
$$
\gamma(i,j):=
\sum_{k=0}^{\sI(i, j)-1}\Ht(a_i(k))
+\sum_{k=0}^{\sI(j, i)-1}\Ht(a_j(k))
+d_{\cH}(a_i(\sI(i, j)), a_j(\sI(j, i))), 
$$
with the convention if $k$ ranges from $0$ to $-1$, the sum equals zero.

The justification of the definition relies on weak convergence arguments:
Let $\bp_n$, $n\ge 1$, be a sequence of probability measures such that \eqref{H} holds 
with $\btheta$ the parameter of $\cH$. Let $H^n$ be a $\bp_n$-tree and $U^n$ a $\bp_n$-node. Let $(\xi^n_i)_{i\ge 1}$
be a sequence of i.i.d.\ $\bp_n$-points.
Then, the quantities $S_x^n$, $\bB^n$, $\bx^n$, $\ba^n$, and
$\me^n(i,j)$ are defined for $H^n$ in the same way as $S_x$, $\bB$, $\bx$, $\ba$,
and $\me(i,j)$ have been defined for $\cH$. Let $d_{H^n}$ denote the graph distance on $H^n$. 
There is only a slight difference in the 
definition of the distances
$$
\gamma^n(i,j):=
\!\!\!\sum_{k=0}^{\sI^n(i, j)-1}\!\!\!\Big(\Ht(a^n_i(k))+1\Big)
+\!\!\!\sum_{k=0}^{\sI^n(j, i)-1}\!\!\!\Big(\Ht(a^n_j(k))+1\Big)
+d_{H^n}\Big(a_i(\sI^n(i, j)), a^n_j(\sI^n(j, i))\Big),
$$
to take into account the length of the edges $\{x,A^n_x\}$, for $x\in \bB^n$. 
In that case, the sequence $\bx^n$ (resp. $\ba^n$) is eventually constant and equal to $U^n$ so that 
$\me^n(i,j)<\infty$ with probability one. Furthermore, the unique tree defined by the 
distance matrix $(\gamma^n(i,j):i,j\ge 1)$ is easily seen to have the same distribution as the one 
defined in Section~\ref{sec:ptree-one-cutting}, since the attaching 
points are sampled with the same distributions and $(\gamma^n(i, j): i, j\ge 1)$ coincides with the 
tree distance after attaching. 
Recall that we have re-rooted $\shuff(H^n, U^n)$ at a random point of law $\bp_n$. 
We may suppose this point is $\xi^n_1$.
Therefore we have (Proposition \ref{pro:one-duality})
\begin{equation}\label{eq: idshuffTn}
(\shuff(H^n,U^n), H^n)\overset{d}{=}(H^n, \cut(H^n,U^n)),
\end{equation}
by Lemma \ref{lem:joint_dist-tree-node}. 

In the case of the ICRT $\cH$, it is a priori not clear that $\bbP(\me(i,j)<\infty)=1$. 
We prove that  
\begin{thma}\label{thm: shuff}
For any ICRT $(\cH, \mu_{\cH})$ and a $\mu_{\cH}$-point $U$, we have the following assertions:
\begin{enumerate}[a)]
\item almost surely for each $i, j\ge 1$, we have $\me(i,j)<\infty$;
\item almost surely the distance matrix $(\gamma(i,j), 1\le i,j<\infty)$ defines a CRT, 
denoted by $\shuff(\cH,U)$;
\item $(\shuff(\cH, U), \cH)$ and $(\cH, \cut(\cH, V))$ have the same distribution.
\end{enumerate}
\end{thma}
 
The main ingredient in the proof of Theorem~\ref{thm: shuff} is the following lemma:

\begin{lem}\label{lem: cvpp}
Under \eqref{H}, for each $k\ge 1$, we have the following convergences
\begin{align}\label{eq: cvgm}
\big(\sigma_nd_{H^n}(r(H^n), \bx_i^n(j)), 1\le i\le k, 0\le j\le k\big)
&\carrow \big(d_{\cH}(r(\cH), \bx_i(j)), 1\le i\le k, 0\le j\le k\big),\\ \label{eq: cvpp}
\big(\bp_n(S^n_{x^n_i(j)}), 1\le i\le k, 0\le j\le k\big) &\carrow \big(\mu_{\cH}(S_{x_i(j)}), 1\le i\le k, 0\le j\le k\big),
\end{align}
and 
\begin{equation}\label{eq: cvag}
\big(\sigma_n H^n, (a^n_i(j), 1\le i\le k, 0\le j\le k) \big)
\carrow \big(\cH, (a_i(j), 1\le i\le k, 0\le j\le k)\big), 
\end{equation}
in the weak convergence of the pointed Gromov--Prokhorov topology.
\end{lem}

\begin{proof}
Fix some $k\ge 1$. We argue by induction on $j$. For $j=0$, we note that $a^n_i(0)=\xi^n_i$ and 
$x^n_i(0)=\xi^n_i\wedge U^n$. Then the convergences in \eqref{eq: cvag} and \eqref{eq: cvgm} for $j=0$ 
follows easily from \eqref{eq: G-P_ptree}. On the other hand, we can prove \eqref{eq: cvpp} with the 
same proof as in Lemma~\ref{lem: sl}.
Suppose now \eqref{eq: cvgm}, \eqref{eq: cvpp} and \eqref{eq: cvag} hold true for some $j\ge 0$.
We Notice that $a^n_i(j+1)$ is independently sampled according to $\bp_n$ restricted to $S_{x^n_i(j)}$, 
we deduce \eqref{eq: cvag} for $j+1$ from \eqref{eq: G-P_ptree}. Then the convergence in \eqref{eq: cvgm} 
also follows for $j+1$, since $x^n_i(j+1)=a^n_i(j)\wedge U^n$. 
Finally,the very same arguments used in the proof of Lemma~\ref{lem: sl} show that \eqref{eq: cvpp} holds 
for $j+1$.
\end{proof}

\begin{proof}[Proof of Theorem \ref{thm: shuff}]
\noindent\emph{Proof of a)}\space  
By construction, $\shuff(H^n,U^n)$ is the reverse transformation of the one 
from $H^n$ to $\cut(H^n,U^n)$ in the sense that each attaching ``undoes'' a cut. 
In consequence, since $\me^n(i,j)$ is the number of cuts to undo in order to get 
$\xi^n_i$ and $\xi^n_j$ in the same connected component, $\me^n(i,j)$ has the same 
distribution as the number of the cuts that fell on the path $\llb \xi^n_i, \xi^n_j\rrb$. 
But the latter is stochastically bounded by a Poisson variable $N_n(i, j)$ of mean 
$d_{H^n}(\xi^n_i, \xi^n_j)\cdot E_n(i, j)$, where $E_n(i, j)$ is an independent exponential variable 
of rate $d_{H^n}(U^n, \xi^n_i\wedge \xi^n_j)$. Indeed, each cut is a point of the Poisson 
point process $\cP^n$ and no more cuts fall on $\llb \xi^n_i, \xi^n_j\rrb$ after the time of the
first cut on $\llb U^n, \xi^n_i\wedge \xi^n_j\rrb$. 
But the time of the first cut on $\llb U^n, \xi^n_i\wedge \xi^n_j\rrb$ has the same 
distribution as $E_n(i, j)$ and is independent of $\cP^n$ restricted on  $\llb \xi^n_i, \xi^n_j\rrb$. 
The above argument shows that
\begin{equation}\label{eq: j+j}
\me^n(i,j)=\sI^n(i, j)+\sI^n(j, i)\le_{st} N_n(i, j), \quad i, j\ge 1, n\ge 1,
\end{equation}
where $\le_{st}$ denotes the stochastic domination order. It follows from \eqref{eq: G-P_ptree} 
that, jointly with the convergence in \eqref{eq: G-P_ptree}, we 
have $N_n(i, j)\to N(i, j)$ in distribution, as $n\to\infty$, where $N(i, j)$ is a 
Poisson variable with parameter $d_{\cH}(\xi_i, \xi_j)\cdot E(i, j)$ with 
$E(i, j)$ an independent exponential variable of rate 
$d_{\cH}(U, \xi_i\wedge \xi_j)$, which is positive with probability one. 
Thus the sequence $(\me^n(i,j), n\ge 1)$ is tight in $\bbR_+$.

On the other hand, observe that for $x\in \bB$, $\bbP(A_x\in F_y)=\mu_{\cH}(F_y)/\mu_{\cH}(S_x)$ 
if $y\in \bB$ and $d_{\cH}(U, y)<d_{\cH}(U, x)$. In particular, for two distinct points $x, x'\in \bB$,
$$
\p{\exists \ y\in \bB \text{ such that } A_x\in F_y, A_{x'}\in F_y}
=\sum_{y}\frac{\mu_{\cH}^2(F_y)}{\mu_{\cH}(S_x)\mu_{\cH}(S_{x'})},
$$
where the sum is over those $y\in \bB$ such that $d_{\cH}(U, y)<\min\{d_{\cH}(U, x),d_{\cH}(U, x')\}$. 
Similarly, for $n\ge 1$, 
$$
\p{\exists \ y\in \bB^n \text{ such that } A^n_x\in F^n_y, A^n_{x'}\in F^n_y}
=\sum_{y}\frac{\bp_n^2(F^n_y)}{\bp_n(S^n_x)\bp_n(S^n_{x'})}.
$$
Then it follows from \eqref{eq: cvgm} and the convergence of the masses in Lemma~\ref{lem: sl} 
that
\begin{align*}
\p{\sI^n(i, j)=1; \sI^n(j, i)=1}&=\bbP\big(\exists \ y\in \bB^n \text{ such that } A^n_{x_i(0)}\in F^n_y, A^n_{x_j(0)}\in F^n_y\big) \\
&\overset{n\to\infty}{\longrightarrow} \p{\sI(i, j)=1; \sI(j, i)=1}.
\end{align*}
By induction and Lemma \ref{lem: cvpp}, this can be extended to the following: 
for any natural numbers $k_1,k_2\ge 0$, we have
$$
\p{\sI^n(i, j)=k_1; \sI^n(j, i)=k_2 } \overset{n\to\infty}{\longrightarrow} \p{\sI(i, j)=k_1; \sI(j, i)=k_2}.
$$
Combined with the  tightness of $(\me^n(i,j), n\ge 1)=(\sI^n(i, j)+\sI^n(j, i), n\ge 1)$, 
this entails that 
\begin{equation}\label{eq: cv_sI}
(\sI^n(i, j), \sI^n(j, i))\carrow (\sI(i, j), \sI(j, i)), \quad i, j\ge 1
\end{equation}
jointly with \eqref{eq: cvgm} and \eqref{eq: cvag}, using the usual subsequence arguments.
In particular, $\sI(i, j)+\sI(j, i)\le_{st}N(i,j)<\infty$ almost surely, 
which entails that $\me(i, j)<\infty$ almost surely, for each pair $(i, j)\in \bbN\times \bbN$.

\medskip
\noindent\emph{Proof of b)}\space 
It follows from \eqref{eq: cvag}, \eqref{eq: cv_sI} and the expression of $\gamma(i, j)$ that
\begin{equation}\label{eq: cvdelta}
\big(\sigma_n\gamma^n(i, j), i, j\ge 1\big)
\carrow \big(\gamma(i, j), i, j\ge 1 \big), 
\end{equation}
in the sense of finite-dimensional distributions, jointly with the Gromov--Prokhorov 
convergence of $\sigma_n H^n$ to $\cH$ in \eqref{eq: G-P_ptree}.  
However by \eqref{eq: idshuffTn}, the distribution of $\shuff(H^n,U^n)$ is identical to 
$H^n$. Hence, the unconditional distribution of $(\gamma(i,j), 1\le i,j<\infty)$ is that of 
the distance matrix of the ICRT $\cH$. We can apply Aldous' CRT theory \cite{aldcrt3} to conclude that 
for a.e.\ $\cH$, the distance matrix $(\gamma(i, j), i, j\ge 1)$ defines a CRT, denoted by 
$\shuff(\cH, U)$. Moreover, there exists a mass measure $\tilde{\mu}$, such that if 
$(\tilde{\xi}_i)_{i\ge 1}$ is an i.i.d.\ sequence of law $\tilde{\mu}$, then
$$
(d_{\shuff(\cH, U)}(\tilde{\xi}_i, \tilde{\xi}_j), 1\le i, j<\infty)
\eqd (\gamma(i, j), 1\le i, j<\infty).
$$
Therefore, we can rewrite \eqref{eq: cvdelta} as
\begin{equation}\label{eq: cvshuff_n}
\big(\sigma_n\shuff(H^n,U^n), \sigma_n H^n\big)
\carrow \big(\shuff(\cH,U), \cH\big),
\end{equation}
with respect to the Gromov--Prokhorov topology.

\medskip
\noindent\emph{Proof of c)}\space
This is an easy consequence of \eqref{eq: idshuffTn} and \eqref{eq: cvshuff_n}. 
Let $f, g$ be two arbitrary bounded functions continuous in the Gromov--Prokhorov topology. 
Then \eqref{eq: cvshuff_n} and the continuity of $f, g$ entail that
\begin{align*}
\E{f(\shuff(\cH,U))\cdot g(\cH)}
&=\lim_{n\to\infty}\E{f(\sigma_n\shuff(H^n,U^n)) \cdot g(\sigma_n H^n)}\\
&=\lim_{n\to\infty}\E{f(\sigma_n H^n) \cdot g(\sigma_n\cut(H^n,U^n)}\\
&=\E{f(\cH)\cdot g(\cut(\cH,U))},
\end{align*}
where we have used \eqref{eq: idshuffTn} in the second equality. 
Thus we obtain the identity in distribution in~c).
\end{proof}

\subsection{Distribution of the cuts}\label{sec:dist_cuts}

According to Proposition~\ref{pro:one-duality} and \eqref{eq: idshuffTn}, the attaching points 
$a^n_i(j)$ have the same distribution as the points where the cuts used to be connected to in 
the $\bp_n$ tree $H^n$. Then Theorem~\ref{thm: shuff} suggests that the weak limit $a_i(j)$ should 
play a similar role for the continuous tree. Indeed in this section, we show that $a_i(j)$ represent 
the ``holes" left by the cutting on $(\cT_t)_{t\ge 0}$. 

Let $(\cT, d_{\cT}, \mu)$ be the ICRT in Section \ref{sec:cont_cutting}. The $\mu$-point $V$ is isolated 
by successive cuts, which are elements of the Poisson point process $\cP$. Now let $\xi'_1, \xi'_2$ be 
two independent points sampled according to $\mu$. We plan to give a description of the image of the 
path $\llb \xi'_1, \xi'_2\rrb$ in the cut tree $\cut(\cT, V)$, which turns out to be dual to the 
construction of one path in $\shuff(\cH, U)$. 

During the cutting procedure which isolates $V$, the path $\llb \xi'_1, \xi'_2\rrb$ is pruned 
from the two ends into segments. See Figure~\ref{fig: 1-path_tc}.
Each segment is contained in a distinct portion $\Delta\cT_t:=\cT_{t-}\setminus \cT_t$, which is 
discarded at time $t$. Also recall that $\Delta\cT_t$ is grafted on the interval $[0, L_\infty]$ to 
construct $\cut(\cT, V)$. The following is just a formal reformulation: 
\begin{figure}[htp]
\centering
\includegraphics[scale=1]{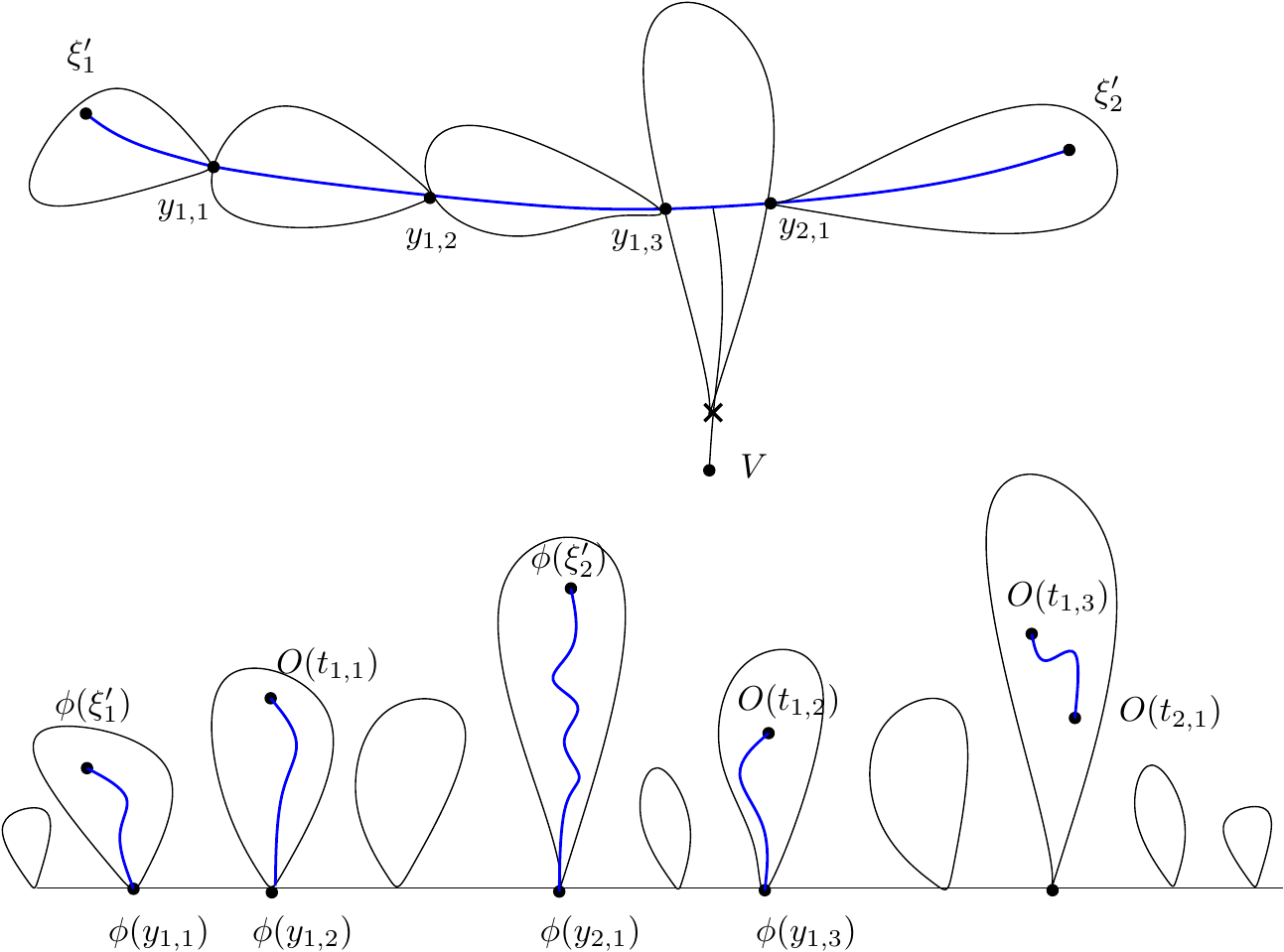}
\caption{\label{fig: 1-path_tc}An example with $M_1=3$ and $M_2=1$. 
Above, the cuts partition the path between $\xi'_1$ and $\xi'_2$ into segments. 
The cross represents the first cut on $\llb\xi'_1, V\rrb\cap\llb \xi'_2, V\rrb$. 
Below, the image of these segments in $\cut(\cT, V)$. }
\end{figure}

\begin{lem}\label{lem: 1-path_tc}
Let 
$$
(t_{1, 1}, y_{1, 1}), (t_{1, 2}, y_{1, 2}), \cdots, (t_{1, M_1}, y_{1, M_1})
\qquad \text{ and } \qquad 
(t_{2, 1}, y_{2, 1}), (t_{2, 2}, y_{2, 2}), \cdots, (t_{2, M_2}, y_{2, M_2}),
$$
be the respective (finite) sequences of cuts on $\llb \xi'_1, V\rrb\cap\llb\xi'_1, \xi'_2\rrb$ and 
$\llb \xi'_2, V\rrb\cap\llb\xi'_1, \xi'_2\rrb$ such that
$0<t_{i, 1}<t_{i, 2}<\cdots<t_{i, M_i}<\infty$ for $i=1, 2$.
Then the points $\{y_{i, j}: 1\le j\le M_i, i=1, 2\}$ partition of the path $\llb\xi'_1, \xi'_2\rrb$ into
segments and: 
\begin{itemize}
    \item for $i=1, 2$, $\llb \xi'_i, y_{i, 1}\rrb\subset\Delta\cT_{t_{i, 1}}$;
    \item for $j=1, 2, \cdots, M_i-2$, $\rrb y_{i, j}, y_{i, j+1}\rrb\subset \Delta\cT_{t_{i, j+1}}$.
\end{itemize}
Finally, writing
$$
t_{me}:=\inf\{t>0: \cP_t\cap \llb\xi'_1, V\rrb\cap\llb \xi'_2, V\rrb\ne\varnothing\}<\infty,
$$
$\rrb y_{1, M_1}, y_{2, M_2}\llb$ is contained in $\Delta\cT_{t_{me}}$.
\end{lem}

\begin{proof}
It suffices to prove that $M_1, M_2$ are finite with probability $1$. The other statements are 
straightforward from the cutting procedure. But an argument similar to the one used in the proof of a) of 
Theorem~\ref{thm: shuff} shows that $M_1+M_2$ is stochastically bounded by a Poisson variable with 
mean $d_{\cT}(\xi'_1, \xi'_2)\cdot t_{me}$, which entails that $M_1, M_2$ are finite almost surely.
\end{proof}
 
Recall that $\cut(\cT, V)$ is defined so as to be a complete metric space. Denote by $\phi$ the 
canonical injection from $ \cup_{t\in \cC}\Delta\cT_{t} $ to $\cut(\cT, V)$.  
For $1\le j\le M_i-2$ and $i=1, 2$, it is not difficult to see that there exists some 
point $O(t_{i, j})$ of $\cut(\cT, V)$ such that the closure of $\phi(\rrb y_{i, j}, y_{i, j+1}\rrb)$ is 
$\llb O(t_{i, j}), \phi(y_{i, j+1})\rrb$. Similarly, the closure of $\phi(\rrb y_{1, M_1}, y_{2, M_2}\llb)$ 
is equal to $\llb O(t_{1, M_1}), O(t_{2, M_2})\rrb$, with $O(t_{1, M_1}), O(t_{2, M_2})$ two leaves contained 
in the closure of $\phi(\Delta\cT_{t_{me}})$. Comparing this with Theorem~\ref{thm: shuff}, one 
may suspect that $\{O(t_{1, j}): 1\le j\le M_1\},\{O(t_{2, j}): 1\le j\le M_2\}$ should have the same 
distribution as $\{a_1(j): 1\le j\le \sI(1, 2)\}, \{a_2(j): 1\le j\le \sI(2, 1)\}$. 
This is indeed true. In the following, we show a slightly more general result about all the points. 
For each $t\in \cC=\{t>0: \mu(\Delta\cT_t)>0\}$, let $x(t)\in \cT$ be the point such that $(t, x(t))\in \cP$. Then we can define $O(t)$ to be the point of $\cut(\cT, V)$ which marks the ``hole" left by the cutting at $x(t)$. More precisely, 
let  $(t', x')$ be the first element after time $t$ of $\cP$ on $ \llb r(\cT), x(t)\llb$.
Then there exists some point $O(t)$ such that
the closure of $\phi(\rrb x(t), x'\rrb)$ in $\cut(\cT, V)$ is $\llb O(t), \phi(x')\rrb$.

\begin{prop}\label{prop: distcut}
Conditionally on $\cut(\cT, V)$, the collection $\{O(t), t\in \cC\}$ is independent, and each $O(t)$  has distribution $\hat{\mu}$ restricted to $\cup_{s>t}\phi(\cT_{s-}\setminus\cT_s)$.
\end{prop}

\begin{proof}
It suffices to show that $\{O(t), t\in \cC\}$ has the same distribution as the collection of attaching 
points $\{A_x, x\in \bB\}$ introduced in the previous section. 
Observe that if we take $(\cH, U)=(\cut(\cT, V), L_\infty)$ and 
replace $\{A_x, x\in \bB\}$ with $\{O(t), t\in\cC\}$, then it follows that $\shuff(\cH, U)$ is isometric 
to $\cT$, since the two trees are metric completions of the same distance matrix with probability one.
In particular, we have
\begin{equation}\label{eq: id_shuff_pf}
(\shuff(\cH, U), \cH)\eqd (\cT, \cut(\cT, V)).
\end{equation}
Therefore, to determine the distribution of $\{O(t), t\in \cC\}$, we only need to argue that the 
distribution of $\{A_x, x\in \bB\}$ is the unique distribution for which \eqref{eq: id_shuff_pf} holds. 
To see this, we notice that \eqref{eq: id_shuff_pf} implies that the distribution of $(\gamma(i, j))_{i, j\ge 1}$ 
is unique. But from the distance matrix $(\gamma(i, j))_{i, j\ge 1}$ (also given $\cH$ and $(\xi_i, i\ge 1)$), 
we can recover $(a_i(1), i\ge 1)$, which is a size-biased resampling of $(A_x, x\in \bB)$.  
Indeed, the sequence $(\xi_k)_{k\ge 1}$ is everywhere dense in $\cH$. For $x\in \bB$, let $(\xi_{m_k}, k\ge 1)$ 
be the subsequence consisting of the $\xi_i$ contained in $F_x$. Then $a_i(1)\in F_x$ if and only if 
$\liminf_{k\to\infty}\gamma(i, m_k)-\Ht(\xi_i)=0$,  where $\Ht(\xi_i)=d_{\cH}(\xi_i, \xi_i\wedge U)$. 
Moreover, if the latter holds, we also have $d_{\cH}(a_i(1), \xi_{m_k})=\gamma(i, m_k)-\Ht(\xi_i)$. 
By Gromov's reconstruction theorem \cite[$3\frac{1}{2}$]{Gromov}, we can determine $a_i(1)$ for each $i\ge 1$. 
By the previous arguments, this concludes the proof. 
\end{proof}

The above proof also shows that if we use $(O(t), t\in \cC)$ to define the points 
$(A_x, x\in \bB)$ then the shuffle operation yields a tree that is undistinguishable from the 
original ICRT $\cT$. 

\section{Convergence of the cutting measures: Proof of Proposition~\ref{prop: cv-Ln}}
\label{sec:cv-Ln}

Recall the setting at the beginning of Section~\ref{sec:ICRT_overview}. 
Then proving Proposition~\ref{prop: cv-Ln} amounts to show that for each $k\ge 1$, we have
\begin{equation}
\big(\sigma_n R^n_{k},\cL_n\!\!\upharpoonright_{R^n_k}\big)
\carrow \big(R_{k}(\te),\cL\!\!\upharpoonright_{R_k}\big)
\end{equation}
in Gromov--Hausdorff--Prokhorov topology. 
Observe that the Gromov--Hausdorff convergence is clear from \eqref{eq: G-P_ptree}, 
so that it only remains to prove the convergence of the measures.

\medskip
\noindent\textbf{Case $1$} 
We first prove the claim assuming that $\theta_i>0$ for every $i\ge 0$.
In this case, define
$$m_n:=\min\bigg\{j: \sum_{i=1}^j \Big(\frac{p_{ni}}{\sigma_n}\Big)^2
\ge \sum_{i\ge 1} \theta_i^2\bigg\},$$
and observe that $m_n<\infty$ since $\sum_{i\le n} (p_{ni}/\sigma_n)^2=1\ge \sum_{i\ge 1} \theta_i^2$.
Note also that $m_n\to \infty$. Indeed, for every integer $k\ge 1$, 
since $p_{ni}/\sigma_n\to \theta_i$, for $i\ge 1$, and $\theta_{k+1}>0$,
we have, for all $n$ large enough,
$$\sum_{i=1}^k \Big(\frac{p_{ni}}{\sigma_n}\Big)^2 < \sum_{i\ge 1} \theta_i^2,$$
so that $m_n>k$ for all $n$ large enough. 
Furthermore $\lim_{j\to\infty} \theta_j=0$, and \eqref{H} implies that
\begin{equation}\label{eq: pmn}
\lim_{n\to\infty} \frac{p_{m_n}}{\sigma_n}=0.
\end{equation}
Combining this with the definition of $m_n$, it follows that, as $n\to\infty$,
\begin{equation}\label{eq: mn}
\sum_{i\le m_n}\Big(\frac{p_{ni}}{\sigma_n}\Big)^2\to\sum_{i\ge 1}\theta_i^2.
\end{equation}

If $n, k, M\ge 1$, we set
$$
\cL_n^*=\sum_{m_n < i\le n}\frac{p_{ni}}{\sigma_n}\delta_{i},  
\quad\text{ and }\quad
\Sigma(n,k,M)=\sum_{M < i\le m_n}\frac{p_{ni}}{\sigma_n}\I{i\in R^n_k}.
$$
Let $\ell_n$ denote the (discrete) length measure on $\bT^n$. 
Clearly, $\sigma_n\ell_n$ is the length measure of the rescaled tree $\sigma_n\bT^n$, seen as a real tree.
\begin{lem}\label{lem: Ln23}
Suppose that \eqref{H} holds. Then, for each $k\ge 1$, we have the following assertions:
\begin{enumerate}[a)]
\item
as $n\to\infty$, in probability
\begin{equation}\label{eq: Ln3}
\dpr\big(\cL_n^*\!\!\upharpoonright_{R^n_k}, 
\theta_0^2\sigma_n\ell_n\!\!\upharpoonright_{R^n_k}\big)\to 0;
\end{equation}
\item
for each $\ep>0$, there exists $M=M(k,\ep)\in \bbN$ such that
\begin{equation}\label{eq: Ln2}
\limsup_{n\to\infty}\bbP\big(\Sigma(n,k,M)\ge \ep\big)\le \ep;
\end{equation}
\end{enumerate}
\end{lem}

Before proving Lemma~\ref{lem: Ln23}, 
let us first explain why this entails Proposition~\ref{prop: cv-Ln}.
\begin{proof}[Proof of Proposition \ref{prop: cv-Ln} in the Case $1$]
By Skorokhod representation theorem and a diagonal argument, we can assume that
the convergence $(\sigma_n \bT^n, \mu_n, \fB^n_m)\to (\te, \mu, \fB_m)$,
holds almost surely in the $m$-pointed Gromov--Prokhorov topology for all $m\ge 1$. 
Since the length measure $\ell_n$ (resp.\ $\ell$) depends continuously on the metric of $\bT^n$ 
(resp.\ the metric of $\te$), according to Proposition 2.23 of \cite{LWV} this implies that,
for each $k\ge 1$,
\begin{equation}\label{eq: elln}
  \big(\sigma_n R^n_k, \theta_0^2\sigma_n\ell_n\!\!\upharpoonright_{R^n_k}\big)
  \to \big(R_k, \theta_0^2\ell\!\!\upharpoonright_{R_k}\big),
\end{equation}
almost surely in the Gromov--Hausdorff--Prokhorov topology.
On the other hand, we easily deduce from the convergence of the 
vector $\fB^n_m$ and \eqref{H} that, for each fixed $m\ge 1$,
\begin{equation}\label{eq: Ln1}
\Bigg(\sigma_n R^n_k, \sum_{i=1}^m\frac{p_{ni}}{\sigma_n}\delta_i\!\!\upharpoonright_{R^n_k}\Bigg)
\to \Bigg(R_k, \sum_{i=1}^m\theta_i\delta_{\cB_i}\!\!\upharpoonright_{R_k}\Bigg),
\end{equation}
almost surely in the Gromov--Hausdorff--Prokhorov topology. 
In the following, we write $\dprnk$ (resp.\ $\dprk$) for the Prokhorov distance on 
the finite measures on the set $R^n_k$ (resp.\ $R_k$). 
In particular, since the measures below are all restricted to either $R^n_k$ or $R_k$, 
we omit the notations $\upharpoonright_{R^n_k}$, $\upharpoonright_{R_k}$ 
when the meaning is clear from context. We write
$$
\Kt_m(\cL):=\theta_0^2\ell+\sum_{i=1}^m \theta_i \delta_{\cB_i}
$$
for the cut-off measure of $\cL$ at level $m$.
By Lemma~\ref{lem: cL}, the restriction of $\cL$ to $R^n_k$ is a finite measure. 
Therefore, $\Kt_m(\cL)\to \cL$ almost surely in $\dprk$ as $m\to\infty$.

Now fix some $\ep>0$. By Lemma~\ref{lem: Ln23} we can choose some $M=M(k, \ep)$ such that 
\eqref{eq: Ln2} holds, as well as
\begin{equation}\label{eq: Ktm}
\pc{\dprk(\Kt_M(\cL), \cL)\ge\ep}\le \ep.
\end{equation}
Define now the approximation 
$$
\vartheta_{n, M}:=\theta_0^2\sigma_n\ell_n+\sum_{i\le M}\frac{p_{ni}}{\sigma_n}\delta_i.
$$
Then recalling the definition of $\cL_n$ in \eqref{eq:def_Ln}, and using \eqref{eq: Ln2} and 
\eqref{eq: Ln3}, we obtain 
\begin{equation}\label{eq: vartheta}
\limsup_{n\to\infty}\pc{\dprnk(\vartheta_{n, M}, \cL_n)\ge \ep}\le \ep.
\end{equation}
We notice that
\begin{equation}\label{eq: cvvartheta}
 \big(\sigma_n R^n_k, \vartheta_{n, M}\big)\to \big(R_k, \Kt_M(\cL)\big)
\end{equation}
almost surely in the Gromov--Hausdorff--Prokhorov topology as a combined consequence 
of \eqref{eq: elln} and \eqref{eq: Ln1}. 
Finally, by the triangular inequality, we deduce from \eqref{eq: Ktm}, \eqref{eq: vartheta} and 
\eqref{eq: cvvartheta} that
$$
\limsup_{n\to\infty}\p{\dghp\big((\sigma_nR^n_k, \cL_n), (R_k, \cL)\big)\ge 2\ep}\le 2\ep,
$$
for any $\ep>0$, which concludes the proof.
\end{proof}

\begin{proof}[Proof of Lemma \ref{lem: Ln23}]
We first consider the case $k=1$. Define
\begin{equation}\label{eq:def_dist_fun}
D_n:=d_{\bT^n}\big(r(\bT^n), V^n_1\big), 
\qquad\text{and}\qquad
F^\cL_n(l):=\cL^*_n\big(\bB(r(\bT^n), l)\cap R^n_1\big),
\end{equation}
where $\bB(x, l)$ denotes the ball in $\bT^n$ centered at $x$ and with radius $l$. 
Then the function $F^\cL_n$ determines the measure $\cL^*_n\!\!\upharpoonright_{R^n_1}$ 
in the same way a distributional function determines a finite measure of $\bbR_+$.
Let $(X^n_{j}, j\ge 0)$ be a sequence of i.i.d.\ random variables of distribution $\bp_n$. 
We define $\fR^n_0=0$, and for $m\ge 1$,
$$
\fR^n_m=\inf\big\{j> \mathfrak R^n_{m-1}: X^n_{j}\in \{X^n_{1}, X^n_{2}, \cdots, X^n_{j-1}\}\big\}
$$
the $m$-th repeat time of the sequence. 
For $l\ge 0$, we set
$$
F_n(l):=\sum_{j=0}^{l\wedge\,(\fR^n_1-1\!)} \sum_{i>m_n}\frac{p_{ni}}{\sigma_n}\I{X^n_{j}=\,i}.
$$
According to the construction of the birthday tree in~\cite{pit00} and Corollary 3 there, we have 
\begin{equation}\label{eq: idfcln}
\big(D_n, F^\cL_n(\cdot)\big)\overset{d}{=}\big(\mathfrak R^n_1-1, F_n(\cdot)\big).
\end{equation}
Let $q_n\ge 0$ be defined by $q_n^2=\sum_{i>m_n}p_{ni}^2$. 
Then \eqref{eq: mn} entails $\lim_{n\to\infty}q_n/\sigma_n=\theta_0$. 
For $l\ge 0$, we set
$$
Z_n(l):=\left|F_n(l)-\frac{q_n^2}{\sigma_n}\big((l+1)\wedge \fR^n_1\big)\right|.
$$

We claim that $\sup_{l\ge 0}Z_n(l)\to 0$ in probability as $n\to\infty$.
To see this, observe first that
$$
Z_n(l)
=\left|\sum_{j=0}^{l\wedge\,(\fR^n_1-1\!)} \left(\sum_{i>m_n}
\frac{p_{ni}}{\sigma_n}\I{X^n_{j}=\,i}-\frac{q_n^2}{\sigma_n}\right)\right|,
$$
where the terms in the parenthesis are independent, centered, 
and of variance $\chi_n:=\sigma_n^{-2}\sum_{i>m_n} p_{ni}^3-\sigma_n^{-2}q_n^4$. 
Therefore, Doob's maximal inequality entails that for any fixed number $N>0$,
\begin{align*}
\E{\bigg(\sup_{l\ge 0}Z_n(l)\boldsymbol{1}_{\{\mathfrak R^n_1\le N/\sigma_n\}}\bigg)^2}
&\le \E{\left(\sup_{l< \lfloor N/\sigma_n\rfloor}\sum_{j=0}^l\left(\sum_{i>m_n}
\frac{p_{ni}}{\sigma_n}\I{X^n_{j}=\,i}-\frac{q_n^2}{\sigma_n}\right)\right)^2}\\
&\le 4N\sigma_n^{-1}\chi_n\\
&\le 4N\frac{q_n^2}{\sigma_n^2}\frac{p_{nm_n}+q_n^2}{\sigma_n} \to 0
\end{align*}
by \eqref{eq: pmn} and the fact that $q_n/\sigma_n\to \theta_0$.
In particular, it follows that
\begin{equation}\label{eq: Zn}
\sup_{l\ge 0} Z_n(l)\I{\mathfrak R^n_1\le N/\sigma_n}\to 0,
\end{equation}
in probability as $n\to \infty$. 
On the other hand, the convergence of the $\bp_n$-trees in \eqref{eq: G-P_ptree} implies that 
the family of distributions of $(\sigma_n D_n, n\ge 1)$ is tight. 
By \eqref{eq: idfcln}, this entails that
\begin{equation}\label{eq: tgrenk}
\lim_{N\to\infty}\limsup_{n\to\infty}\p{\fR^n_1>N/\sigma_n}=0.
\end{equation}
Combining this with \eqref{eq: Zn} proves the claim.

The generalized distribution function as in \eqref{eq:def_dist_fun} for the discrete length 
measure $\ell_n$ is $l\mapsto l\wedge D_n$. 
Thus, since $\sup_l Z_n(l)\to 0$ in probability, the identity in \eqref{eq: idfcln} and 
$q_n/\sigma_n\to \theta_0$ imply that 
$$
\dpr\big(\cL^*_n\!\!\upharpoonright_{R^n_1}, 
\theta_0^2\sigma_n\ell_n\!\!\upharpoonright_{R^n_1}\big)\to 0
$$
in probability as $n\to\infty$. This is exactly \eqref{eq: Ln3} for $k=1$.

In the general case where $k\ge 1$, we set
$$
D_{n, 1}:=D_n, \quad \quad D_{n, m}:=d_{\bT^n}\big(b_n(m), V^n_m\big), \quad m\ge 2,
$$
where $b_n(m)$ denotes the branch point of $\bT^n$ between $V^n_m$ and $R^n_{m-1}$, i.e, 
$b_n(m)\in R^n_{m-1}$ such that $\llb r(\bT^n), V^n_m\rrb\cap R^n_{m-1}=\llb r(\bT^n), b_n(m)\rrb$. 
We also define
$$
F^\cL_{n,1}(l):=F^\cL_n, \qquad \text{and}\qquad
F^\cL_{n, m}(l):=\cL ^*_n\big(\bB(b_n(m), l)\, \cap \, \rrb b_n(m), V^n_m\rrb\big), \quad m\ge 2.
$$
Then conditional on $\mathfrak R^n_k$, the vector 
$(F^\cL_{n, 1}(\cdot), \cdots, F^\cL_{n, k}(\cdot))$ determines the measure 
$\cL^*_n\!\!\upharpoonright_{R^n_k}$ for the same reason as before. If we set
$$
F_{n, 1}(l):=F_n(l), \qquad \text{and}\qquad
F_{n, m}(l):=\sum_{j=\fR^n_{m-1}\!+1}^{l\wedge\,(\fR^n_m-1\!)} \sum_{i>m_n}\frac{p_{ni}}{\sigma_n}\boldsymbol{1}_{\{X^n_{j}=\,i\}},\quad m\ge 2,
$$
then Corollary 3 of \cite{pit00} entails the equality in distribution
$$
\big(\big(D_{n, m}, F^\cL_{n, m}(\cdot)\big), 1\le m\le k\big)
\overset{d}{=}
\big(\big(\mathfrak R^n_m-\mathfrak R^n_{m-1}-1, F_{n, m}(\cdot)\big), 1\le m\le k\big)
$$
Then by the same arguments as before we can show that 
$$
\max_{1\le m\le k}\sup_{l\ge 0}\left|F_{n, m}(l)-
\frac{q_n^2}{\sigma_n}\Big(l\wedge (\fR^n_m-\fR^n_{m-1}-1)\Big)\right|\to 0
$$
in probability as $n\to \infty$. This then implies \eqref{eq: Ln3} by the same type of 
argument as before.

Now let us consider \eqref{eq: Ln2}. The idea is quite similar. For each $M\ge 1$, we set
$$
\tilde{Z}_{n, M}:=\sum_{j=0}^{\fR^n_1-1} \sum_{M<i\le m_n}\frac{p_{ni}}{\sigma_n}\I{X^n_{j}=\,i}.
$$
Then 
$$
\E{\tilde{Z}_{n, M}\I{\fR^n_1\le N/\sigma_n}}
\le N\bigg(\sum_{M<i\le m_n}\frac{p_{ni}^2}{\sigma_n^2}\bigg).
$$
Using \eqref{eq: pmn}, \eqref{H} and the fact that $\sum_i \theta_i^2<\infty$,
we can easily check that for any fixed $N$,
\begin{equation}\label{eq: tildeZ}
\lim_{M\to\infty}\limsup_{n\to\infty}\E{\tilde{Z}_{n, M}\I{\fR^n_1\le N/\sigma_n}}=0.
\end{equation}
By Markov's inequality, we have
$$
\pc{\tilde{Z}_{n, M}>\ep}
\le \ep^{-1}\Ec{\tilde{Z}_{n, M}\I{\fR^n_1\le N/\sigma_n}}+\pc{\fR^n_1>N/\sigma_n}.
$$
According to  \eqref{eq: tgrenk} and \eqref{eq: tildeZ}, we can first choose some $N=N(\ep)$ 
then some $M=M(N(\ep), \ep)=M(\ep)$ such that $\limsup_n \pc{\tilde{Z}_{n, M}>\ep}<\ep$. 
On the other hand, Corollary 3 of \cite{pit00} says that $\Sigma(n, 1, M)$ is distributed like 
$\tilde{Z}_{n, M}$. Then we have shown \eqref{eq: Ln2} for $k=1$. 
The general case can be treated in the same way, and we omit the details.
\end{proof}

So far we have completed the proof of Proposition \ref{prop: cv-Ln} in the case where 
$\btheta$ has all strictly positive entries. The other cases are even simpler:

\noindent\textbf{Case }$2$. Suppose that $\theta_0=0$, we take $m_n=n$ and the same argument follows.

\noindent\textbf{Case }$3$. Suppose that $\btheta$ has a finite length $I$, then it suffices to 
take $m_n=I$. We can proceed as before. 

{\small
\setlength{\bibsep}{.3em}
\bibliographystyle{abbrvnat}
\bibliography{refs}
}

\end{document}